\documentclass[11pt,fleqn,a4paper]{article} 

\usepackage{amsmath,amssymb,amsthm,amsfonts} 
\usepackage[mathscr]{eucal}

\usepackage{hyperref}
\hypersetup{colorlinks, linkcolor=blue, citecolor=blue, urlcolor=blue}

\allowdisplaybreaks

\makeatletter
\long\def\@makecaption#1#2{%
  \vskip\abovecaptionskip\footnotesize
  \sbox\@tempboxa{#1. #2}%
  \ifdim \wd\@tempboxa >\hsize
    #1. #2\par
  \else
    \global \@minipagefalse
    \hb@xt@\hsize{\hfil\box\@tempboxa\hfil}%
  \fi
  \vskip\belowcaptionskip}
\makeatother

\newcommand{\p}{\partial}
\newcommand{\sgn}{\mathop{\rm sgn}\nolimits}

\newlength{\mylength}
\newcommand{\solution}{\hspace*{-\mylength}\bullet\quad}

\newtheorem{theorem}{Theorem}
\newtheorem{lemma}[theorem]{Lemma}
\newtheorem{corollary}[theorem]{Corollary}
\newtheorem{proposition}[theorem]{Proposition}
{\theoremstyle{definition}

}

\newcommand{\todo}[1][\null]{\ensuremath{\clubsuit}}

\newcommand{\noprint}[1]{}
\newcommand{\lsemioplus}{\mathbin{\mbox{$\lefteqn{\hspace{.77ex}\rule{.4pt}{1.2ex}}{\in}$}}}

\begin{document}
\par\noindent {\LARGE\bf
Extended symmetry analysis of remarkable\\ (1+2)-dimensional Fokker--Planck equation
\par}

\vspace{5.5mm}\par\noindent{\large
\large Serhii D. Koval$^{\dag\ddag}$, Alexander Bihlo$^\dag$ and Roman O. Popovych$^{\S}$
}
	
\vspace{5.5mm}\par\noindent{\it\small
$^\dag$Department of Mathematics and Statistics, Memorial University of Newfoundland,\\
$\phantom{^\ddag}$\,St.\ John's (NL) A1C 5S7, Canada
\par}

\vspace{2mm}\par\noindent{\it\small
$^\ddag$\,Department of Mathematics, Kyiv Academic University, 36 Vernads'koho Blvd, 03142 Kyiv, Ukraine
\par}	

\vspace{2mm}\par\noindent{\it\small
$^\S$\,Mathematical Institute, Silesian University in Opava, Na Rybn\'\i{}\v{c}ku 1, 746 01 Opava, Czech Republic\\
$\phantom{^\S}$Institute of Mathematics of NAS of Ukraine, 3 Tereshchenkivska Str., 01024 Kyiv, Ukraine
\par}

\vspace{5mm}\par\noindent
E-mails:
skoval@mun.ca, abihlo@mun.ca, rop@imath.kiev.ua
	
\vspace{6mm}\par\noindent\hspace*{10mm}\parbox{140mm}{\small
We carry out the extended symmetry analysis of an ultraparabolic Fokker--Planck equation with three independent variables,
which is also called the Kolmogorov equation and
is singled out within the class of such Fokker--Planck equations by its remarkable symmetry properties.
In particular, its essential Lie invariance algebra is eight-dimensional,
which is the maximum dimension within the above class.
We compute the complete point symmetry pseudogroup of the Fokker--Planck equation using the direct method,
analyze its structure and single out its essential subgroup.
After listing inequivalent one- and two-dimensional subalgebras of the essential and maximal Lie invariance algebras of this equation,
we exhaustively classify its Lie reductions, carry out its peculiar generalized reductions
and relate the latter reductions to generating solutions with iterative action of Lie-symmetry operators.
As a result, we construct wide families of exact solutions of the Fokker--Planck equation,
in particular, those parameterized by an arbitrary finite number of arbitrary solutions of the (1+1)-dimensional linear heat equation.
We also establish the point similarity of the Fokker--Planck equation to the (1+2)-dimensional Kramers equations
whose essential Lie invariance algebras are eight-dimensional,
which allows us to find wide families of exact solutions of these Kramers equations in an easy way.
}\par\vspace{4mm}
	
\noprint{


35-XX Partial differential equations
  35Kxx Parabolic equations and parabolic systems {For global analysis, analysis on manifolds, see 58J35}
    35K05 Heat equation
    35K10 Second-order parabolic equations
    35K70 Ultraparabolic equations, pseudoparabolic equations, etc.
  35Qxx	Partial differential equations of mathematical physics and other areas of application [See also 35J05, 35J10, 35K05, 35L05]
    35Q84 Fokker-Planck equations {For fluid mechanics, see 76X05, 76W05; for statistical mechanics, see 82C31}
  35Axx General topics
    35A30 Geometric theory, characteristics, transformations [See also 58J70, 58J72]
  35Bxx Qualitative properties of solutions
    35B06 Symmetries, invariants, etc.
  35Cxx Representations of solutions
    35C05 Solutions in closed form
    35C06 Self-similar solutions
    35C07 Traveling wave solutions
}

\section{Introduction}

The Fokker--Planck and Kolmogorov equations provide powerful tools for adequately modeling a wide range of natural processes,
which involve considering fluctuations of a quantity under the action of a random perturbation.
Since the presence of a random noise is a common characteristic of many physical fields,
the Fokker--Planck equations have acquired high popularity in applied sciences.
However, theoretical studies of these equations have aroused lively interest as well,
in particular, in the field of group analysis of differential equations.

For the Fokker--Planck and Kolmogorov equations,
their Lie symmetries and other related objects are of course the most studied in dimension 1+1.
This study was initiated in the seminal paper~\cite{lie1881a} of Sophus Lie himself
in the course of the group analysis of the wider class of second-order linear partial differential equations
with two independent variables, including (1+1)-dimensional second-order linear evolution equations.
A number of papers that restate, specify or develop the above Lie's result have been published over the past decades,
see, e.g., \cite{popo2008a} for a review of these papers and a modern treatment of the problem.
The equivalence of (1+1)-dimensional second-order linear evolution equations with respect to point transformations
was considered in~\cite{blum1980a,cher1957a,lie1881a} and further in~\cite{gung2018a,john2001a,moro2003a}.
Darboux transformations between such equations were studied, e.g., in \cite{blum2004a,blum2020a,matv1991A,popo2008a}.
The group classification problems for the classes of (1+1)-dimensional Fokker--Planck and Kolmogorov equations
and their subclasses were solved in~\cite{popo2008a} up to the general point equivalence
and in~\cite{opan2022a} with respect to the corresponding equivalence groups using the mapping method of group classification.

The most general form of (1+2)-dimensional ultraparabolic Fokker--Planck equations is
\begin{gather}\label{eq:Fokker_Planck_superclass}
\begin{split}
&u_t+B(t,x,y)u_y=A^2(t,x,y)u_{xx}+A^1(t,x,y)u_x+A^0(t,x,y)u+C(t,x,y) \\
&\text{with}\quad A^2\neq 0,\quad B_x\neq 0.
\end{split}
\end{gather}
We denote the entire class of these equations by~$\bar{\mathcal F}$.
Thus, the tuple $\bar\theta:=(B,A^2,A^1,A^0,C)$ of arbitrary elements of the class~$\bar{\mathcal F}$
runs through the solution set of the system of the inequalities $A^2\neq0$ and $B_x\neq0$ with no restrictions on~$A^0$, $A^1$ and~$C$.
A partial preliminary group classification of the class $\bar{\mathcal F}$ was carried out in~\cite{davy2015a}.
Some subclasses of the class~$\bar{\mathcal F}$ were considered within the Lie-symmetry framework
in \cite{gung2018b,kova2013a,spic1997a,spic1999a,spic2011a,zhan2020a}.
Despite the number of papers on this subject, there are still many open problems
in the symmetry analysis of the entire class $\bar{\mathcal F}$,
its subclasses and even particular equations from this class.

In the present paper, we carry out extended symmetry analysis of the equation
\begin{gather}\label{eq:RemarkableFP}
u_t+xu_y=u_{xx},
\end{gather}
which is of the simplest form within the class $\bar{\mathcal F}$ and
corresponds to the values $B=x$, $A^2=1$ and $A^1=A^0=C=0$ of the arbitrary elements.
This equation is distinguished within the class~$\bar{\mathcal F}$ by its remarkable symmetry properties.
In particular, its essential Lie invariance algebra~$\mathfrak g^{\rm ess}$ is eight-dimensional,
which is the maximum dimension for equations from the class~$\bar{\mathcal F}$.
Moreover, it is, up to the point equivalence, a unique equation in~$\bar{\mathcal F}$
whose essential Lie invariance algebra is of this dimension.
That is why we refer to~\eqref{eq:RemarkableFP} as the remarkable Fokker--Planck equation.
The study of the equation~\eqref{eq:RemarkableFP} was initiated by Kolmogorov in 1934~\cite{kolm1934a},
and hence it is often called the Kolmogorov equation as well.
In particular, he constructed its fundamental solution,%
\footnote{%
There was a misprint in the constant factor of this solution in~\cite{kolm1934a}, which was corrected later.
}
\begin{gather}\label{eq:RemarkableFPFundSolKolm}
F(t,x,y,t',x',y')=\frac{\sqrt3H(t-t')}{2\pi(t-t')^2}
\exp\left(-\frac{(x-x')^2}{4(t-t')}-\frac{3\big(y-y'-\frac12(x+x')(t-t')\big)^2}{(t-t')^3}\right),
\end{gather}
where $H$ denotes the Heaviside step function.
A preliminary study of symmetry properties of~\eqref{eq:RemarkableFP}
was carried out in~\cite{gung2018b,kova2013a}.

The algebra~$\mathfrak g^{\rm ess}$ is nonsolvable, and its structure is complicated and specific,
which makes the classification of subalgebras of~$\mathfrak g^{\rm ess}$ nontrivial.
More specifically, this algebra is isomorphic to a semidirect sum ${\rm sl}(2,\mathbb R)\lsemioplus{\rm h}(2,\mathbb R)$
of the real order-two special linear Lie algebra ${\rm sl}(2,\mathbb R)$
and the real rank-two Heisenberg algebra ${\rm h}(2,\mathbb R)$,
where the action of the former algebra on the latter is given by the direct sum
of the one- and four-dimensional irreducible representations of ${\rm sl}(2,\mathbb R)$.
Such structure has not been studied before in the literature on symmetry analysis of differential equations
with regard to classifying subalgebras.
This has been an obstacle to the complete classification of Lie reductions of the equation~\eqref{eq:RemarkableFP},
which we successfully overcome in the present paper.

\looseness=1
Further, using the direct method, we construct the complete point symmetry pseudogroup~$G$ of this equation
and derive a nice representation for transformations from~$G$.
This essentially simplifies the subsequent classification of the one- and two-dimensional subalgebras
of the maximal Lie invariance algebra~$\mathfrak g$ of the equation~\eqref{eq:RemarkableFP} up to the $G$-equivalence,
which is required for optimally accomplishing Lie reductions of~\eqref{eq:RemarkableFP}.
We carefully analyze the structure of the pseudogroup~$G$
and modify the group operation in~$G$ by extending the domains of transformation compositions by continuity.
The main advantage of the suggested interpretation of~$G$ is that
then the pseudogroup~$G$ contains a subgroup~$G^{\rm ess}$,
which is a (finite-dimensional) Lie group with~$\mathfrak g^{\rm ess}$ as its Lie algebra.
We call the subgroup~$G^{\rm ess}$ the essential point symmetry group of the equation~\eqref{eq:RemarkableFP}.
Moreover, we surprisingly find out that
the pseudogroup~$G$ contains only one independent, up to combining with elements from the identity component of~$G$, discrete element.
As such an element, one can choose the involution that only alternates the sign of~$u$.
One more implication of the construction of the pseudogroup~$G$ is
that Kolmogorov's fundamental solution of the equation~\eqref{eq:RemarkableFP} is formally $G^{\rm ess}$-equivalent
to the function $u(t,x,y)=1-H(t)$.

The exhaustive classification of Lie reductions of the equation~\eqref{eq:RemarkableFP}
on the basis of listing $G^{\rm ess}$-inequivalent one- and two-dimensional subalgebras
of its essential Lie invariance algebra~$\mathfrak g^{\rm ess}$
leads to finding wide families of its exact solutions, which are in general $G^{\rm ess}$-inequivalent to each other.
The most interesting among these families are
three families parameterized by single arbitrary solutions of the classical (1+1)-dimensional linear heat equation
and one family parameterized by an arbitrary solution of 
the (1+1)-dimensional linear heat equation with a particular inverse square potential.
We show how to construct more general families of solutions of the equation~\eqref{eq:RemarkableFP}
using generalized reductions with respect to generalized symmetries
generated from elements of~$\mathfrak g^{\rm ess}$ via acting by recursion operators
that are counterparts of elements of~$\mathfrak g^{\rm ess}$ among first-order differential operators in total derivatives.
Since the (1+1)-dimensional linear heat equation with a particular inverse square potential
arises in the course of a Lie reduction of the equation~\eqref{eq:RemarkableFP},
we exhaustively carry out the classical symmetry analysis
of the (1+1)-dimensional linear heat equation with a general inverse square potential,
including the construction of its complete point symmetry pseudogroup by the direct method,
the comprehensive classification of its Lie reductions
and finding all its inequivalent Lie invariant solutions in closed form. 

Among (1+2)-dimensional ultraparabolic Fokker--Planck equations of specific form,
which are called (1+2)-dimensional Klein--Kramers equations or just Kramers equations,
we consider those whose essential Lie invariance algebras are eight-dimensional.
We map these equations to the equation~\eqref{eq:RemarkableFP} using point transformations
and thus reduce the entire study of these equations within the framework of symmetry analysis,
including the construction of exact solutions, to the study of the equation~\eqref{eq:RemarkableFP}.

The structure of the paper is as follows.
In Section~\ref{sec:RemarkableFPMIA}, we present the maximal Lie invariance algebra
of the equation~\eqref{eq:RemarkableFP} and describe its key properties.
Using the direct method, in Section~\ref{sec:RemarkableFPPointSymGroup},
we compute the complete point symmetry pseudogroup of the equation~\eqref{eq:RemarkableFP}
and analyze its structure, including the decomposition of this pseudogroup and the description of its discrete elements.
Section~\ref{sec:RemarkableFPClassIneqSubalgebras} is devoted to the classification
of one- and two-dimensional subalgebras of~$\mathfrak g^{\rm ess}$ and of~$\mathfrak g$.
These are the dimensions that are relevant to the framework of Lie reductions.
The classification of subalgebras lets us
comprehensively study
the codimension-one, -two and -three Lie reductions of the equation~\eqref{eq:RemarkableFP}
in Sections~\ref{sec:RemarkableFPLieRedCoD1}, \ref{sec:RemarkableFPLieRedCoD2} and~\ref{sec:RemarkableFPLieRedCoD3}, respectively,
and the subsequent construction of wide families of exact solutions of~\eqref{eq:RemarkableFP}.
In Section~\ref{sec:GenReductions},
we carry out its peculiar generalized reductions, which are associated with powers of certain Lie-symmetry operators,
and show that the corresponding families of invariant solutions can be generated by iteratively acting with Lie-symmetry operators.%
\footnote{%
A Lie-symmetry operator of a homogeneous linear differential equation~$\mathcal L$: $\mathfrak Lu=0$ is
a first-order linear differential operator~$\mathfrak Q$ in total derivatives that commutes with the operator~$\mathfrak L$
on solutions of~$\mathcal L$ or, equivalently,
such that the differential function $\mathfrak Qu$ is the characteristic of an (essential) Lie symmetry of~$\mathcal L$.
}
The similarity of Kramers equations from the class~$\bar{\mathcal F}$ with eight-dimensional essential Lie symmetry algebras
to the equation~\eqref{eq:RemarkableFP} with respect to point transformations is established explicitly in Section~\ref{sec:KramersEq}.
It leads to easily finding wide families of exact solutions of such Kramers equations.
In Section~\ref{sec:Conclusion}, we discuss how to develop results of the present paper.
Appendix~\ref{sec:SymHeatEqSquarePot} is devoted to extended symmetry analysis
of the (1+1)-dimensional linear heat equations with inverse square potentials,
including the explicit construction of their (real) exact solutions.
In view of results of Section~\ref{sec:RemarkableFPLieRedCoD1},
these solutions directly lead to the solutions of the equation~\eqref{eq:RemarkableFP}.
In Appendix~\ref{sec:HiddenSyms}, we discuss an optimized procedure of constructing
hiddenly invariant solutions of a general system of differential equations.

For readers' convenience,
the constructed exact solutions of the equation~\eqref{eq:RemarkableFP} are marked by the bullet symbol~$\bullet$\,.

\section{Lie invariance algebra}\label{sec:RemarkableFPMIA}

The classical infinitesimal approach results in the well-known algorithm
for computing the maximal Lie symmetry algebras of systems of differential equations~\cite{blum2010A,blum1989A,olve1993A}.
The maximal Lie invariance algebra of the equation~\eqref{eq:RemarkableFP} is
(see, e.g., \cite{kova2013a})
\begin{gather*}\label{eq:RemarkableFPMIA}
\mathfrak g:=\langle \mathcal P^t,\,\mathcal D,\,\mathcal K,\,
\mathcal P^3,\,\mathcal P^2,\,\mathcal P^1,\,\mathcal P^0,\,\mathcal I,\mathcal Z(f)\rangle,
\end{gather*}
where
\begin{gather*}
\mathcal P^t =\p_t,\ \
\mathcal D   =2t\p_t+x\p_x+3y\p_y-2u\p_u,\ \
\mathcal K   =t^2\p_t+(tx+3y)\p_x+3ty\p_y-(x^2\!+2t)u\p_u,\\[.5ex]
\mathcal P^3 =3t^2\p_x+t^3\p_y+3(y-tx)u\p_u,\ \
\mathcal P^2 =2t\p_x+t^2\p_y-xu\p_u,\ \
\mathcal P^1 =\p_x+t\p_y,\ \
\mathcal P^0 =\p_y,\\[.5ex]
\mathcal I   =u\p_u,\quad
\mathcal Z(f)=f(t,x,y)\p_u.
\end{gather*}
Here the parameter function $f$ of~$(t,x,y)$ runs through the solution set of the equation~\eqref{eq:RemarkableFP}.
	
The vector fields $\mathcal Z(f)$ constitute the infinite-dimensional abelian ideal $\mathfrak g^{\rm lin}$ of~$\mathfrak g$
associated with the linear superposition of solutions of~\eqref{eq:RemarkableFP}, $\mathfrak g^{\rm lin}:=\{\mathcal Z(f)\}$.
Thus, the algebra $\mathfrak g$ can be represented as a semidirect sum, $\mathfrak g=\mathfrak g^{\rm ess}\lsemioplus\mathfrak g^{\rm lin}$,
where
\begin{gather}\label{eq:RemarkableFPEssA}
\mathfrak g^{\rm ess}=\langle\mathcal P^t,\mathcal D,\mathcal K,\mathcal P^3,\mathcal P^2,\mathcal P^1,\mathcal P^0,\mathcal I\rangle
\end{gather}
is an (eight-dimensional) subalgebra of $\mathfrak g$,
called the essential Lie invariance algebra of~\eqref{eq:RemarkableFP}.
	
Up to the skew-symmetry of the Lie bracket, the nonzero commutation relations between the basis vector fields of $\mathfrak g^{\rm ess}$ are the following:
\begin{gather*}
[\mathcal P^t,\mathcal D]  = 2\mathcal P^t,\quad
[\mathcal P^t,\mathcal K]  =  \mathcal D,\quad
[\mathcal D,  \mathcal K]  = 2\mathcal K,\\[.5ex]
[\mathcal P^t,\mathcal P^3]= 3\mathcal P^2,\quad
[\mathcal P^t,\mathcal P^2]= 2\mathcal P^1,\quad	
[\mathcal P^t,\mathcal P^1]=  \mathcal P^0,\\[.5ex]
[\mathcal D,\mathcal P^3]  = 3\mathcal P^3,\quad		
[\mathcal D,\mathcal P^2]  =  \mathcal P^2,\quad
[\mathcal D,\mathcal P^1]  =- \mathcal P^1,\quad
[\mathcal D,\mathcal P^0]  =-3\mathcal P^0,\\[.5ex]
[\mathcal K,\mathcal P^2]  =- \mathcal P^3,\quad
[\mathcal K,\mathcal P^1]  =-2\mathcal P^2,\quad
[\mathcal K,\mathcal P^0]  =-3\mathcal P^1,\\[.5ex]
[\mathcal P^1,\mathcal P^2]=- \mathcal I,\quad
[\mathcal P^0,\mathcal P^3]= 3\mathcal I.
\end{gather*}
	
The algebra $\mathfrak g^{\rm ess}$ is nonsolvable.
Its Levi decomposition is given by $\mathfrak g^{\rm ess}=\mathfrak f\lsemioplus\mathfrak r$,
where the radical~$\mathfrak r$ of~$\mathfrak g^{\rm ess}$ coincides with the nilradical of~$\mathfrak g^{\rm ess}$ and
is spanned by the vector fields $\mathcal P^3$, $\mathcal P^2$, $\mathcal P^1$, $\mathcal P^0$ and~$\mathcal I$.
The Levi factor $\mathfrak f=\langle\mathcal P^t,\mathcal D,\mathcal K\rangle$ of~$\mathfrak g^{\rm ess}$
is isomorphic to ${\rm sl}(2,\mathbb R)$,
the radical~$\mathfrak r$ of~$\mathfrak g^{\rm ess}$ is isomorphic to the rank-two Heisenberg algebra ${\rm h}(2,\mathbb R)$,
and the real representation of the Levi factor~$\mathfrak f$ on the radical~$\mathfrak r$
coincides, in the basis $(\mathcal P^3,\mathcal P^2,\mathcal P^1,\mathcal P^0,\mathcal I)$,
with the real representation $\rho_3\oplus \rho_0$ of~${\rm sl}(2,\mathbb R)$.
Here $\rho_n$ is the standard real irreducible representation of~${\rm sl}(2,\mathbb R)$ in the $(n+1)$-dimensional vector space.
More specifically,
\[
\rho_n( \mathcal P^t)_{ij}=(n-j)\delta_{i,j+1},\quad
\rho_n( \mathcal D)_{ij}  =(n-2j)\delta_{ij},\quad
\rho_n(-\mathcal K)_{ij}  =j\delta_{i+1,j},
\]
where $i,j\in\{1,2,\dots,n+1\}$, $n\in\mathbb N\cup\{0\}$,
and $\delta_{kl}$ is the Kronecker delta, i.e., $\delta_{kl}=1$ if $k=l$ and $\delta_{kl}=0$ otherwise, $k,l\in\mathbb Z$.
Thus, the entire algebra~$\mathfrak g^{\rm ess}$ is isomorphic to the algebra~$L_{8,19}$
from the classification of indecomposable Lie algebras of dimensions up to eight
with nontrivial Levi decompositions, which was carried out in~\cite{turk1988a}.

Lie algebras whose Levi factors are isomorphic to the algebra ${\rm sl}(2,\mathbb R)$ often arise
within the field of group analysis of differential equations as Lie invariance algebras of parabolic partial differential equations.
At the same time, the action of Levi factors on the corresponding radicals is usually described
in terms of the representations $\rho_0$, $\rho_1$, $\rho_2$ or their direct sums.
To the best of our knowledge, algebras similar to $\mathfrak g^{\rm ess}$ were never considered
in group analysis from the point of view of their subalgebra structure.

\section{Complete point symmetry pseudogroup}\label{sec:RemarkableFPPointSymGroup}

We start computing the complete point symmetry pseudogroup~$G$ of the equation~\eqref{eq:RemarkableFP}
by presenting the exhaustive description of the equivalence groupoid of the class $\bar{\mathcal F}$.
Then we use special properties of this groupoid for deriving an explicit representation for elements of~$G$.
See~\cite{bihl2012a,bihl2017a,opan2017a,popo2006b,popo2010a,vane2020a} and references therein for definitions and theoretical results
on various structures constituted by point transformations within classes of differential equations.

\begin{theorem}\label{thm:EquivalenceGroupFPsuperClass}
The class~$\bar{\mathcal F}$ is normalized.
Its (usual) equivalence pseudogroup~$G^\sim_{\bar{\mathcal F}}$ consists of the point transformations with the components
\begin{subequations}\label{eq:EquivalenceGroupFPsuperClass}
\begin{gather}
\tilde t=T(t,y),
\quad
\tilde x=X(t,x,y),
\quad
\tilde y=Y(t,y),
\quad
\tilde u=U^1(t,x,y)u + U^0(t,x,y),
\label{eq:ClassFbarTransformationPart}\\
\tilde A^0=\dfrac{A^0}{T_t+BT_y}-\dfrac{A^1}{T_t+BT_y}\dfrac{U^1_x}{U^1}+\dfrac{A^2}{T_t+BT_y}\bigg(\left(\frac{U^1_x}{U^1}\right)^2-
\left(\frac{U^1_x}{U^1}\right)_x\bigg)+\dfrac1{U^1}\dfrac{U^1_t+BU^1_y}{T_t+BT_y},
\label{eq:A^0Transformation}\\
\tilde A^1=A^1\dfrac{X_x}{T_t+BT_y}-\dfrac{X_t+BX_y}{T_t+BT_y}+A^2\dfrac{X_{xx}-2X_xU^1_x/U^1}{T_t+BT_y},
\label{eq:A^1Transformation}\\
\tilde A^2=A^2\dfrac{X_x^2}{T_t+BT_y},
\quad
\tilde B=\dfrac{Y_t+BY_y}{T_t+BT_y},
\label{eq:EquivalenceGroupSuperClassA^2B}
\\
\tilde C=\dfrac{U^1}{T_t+BT_y}\left(C-E\dfrac{U^0}{U^1}\right),
\label{eq:CTransformation}
\end{gather}
\end{subequations}
where $T$, $X$, $Y$, $U^1$ and $U^0$ are arbitrary smooth functions of their arguments with~$(T_tY_y-T_yY_t)X_xU^1\neq0$, and $E:=\p_t+B\p_y-A^2\p_{xx}-A^1\p_x-A^0$.
\end{theorem}

The proof of this theorem is beyond the subject of the present paper and will be presented elsewhere.

The equation~\eqref{eq:RemarkableFP} corresponds to the value
$(x,1,0,0,0)=:\bar\theta_0$ of the arbitrary-element tuple~$\bar\theta=(B,A^2,A^1,A^0,C)$ of the class $\bar{\mathcal F}$.
The vertex group $\mathcal G_{\bar\theta_0}:=\mathcal G^\sim_{\bar{\mathcal F}}(\bar\theta_0,\bar\theta_0)$
is the set of admissible transformations of the class~$\bar{\mathcal F}$ with~$\bar\theta_0$ as both their source and target,
$\mathcal G_{\bar\theta_0}=\{(\bar\theta_0,\Phi,\bar\theta_0)\mid\Phi\in G\}$.
The normalization of the class~$\bar{\mathcal F}$ means that its equivalence groupoid coincides
with the action groupoid of the pseudogroup~$G^\sim_{\bar{\mathcal F}}$,
and thus the latter groupoid necessarily contains the vertex group~$\mathcal G_{\bar\theta_0}$.
This argument allows us to use Theorem~\ref{thm:EquivalenceGroupFPsuperClass} in the course of computing the pseudogroup~$G$.

\begin{theorem}\label{thm:RemarkableFPSymGroup}
The complete point symmetry pseudogroup~$G$ of the remarkable Fokker--Planck equation~\eqref{eq:RemarkableFP}
consists of the transformations of the form
\begin{gather}\label{eq:RemarkableFPSymGroup}
\begin{split}
&\tilde t=\frac{\alpha t+\beta}{\gamma t+\delta},
\quad
\tilde x=\frac{\hat x}{\gamma t+\delta}
-\frac{3\gamma\hat y}{(\gamma t+\delta)^2},
\quad
\tilde y=\frac{\hat y}{(\gamma t+\delta)^3},
\\[1ex]
&\tilde u=\sigma(\gamma t+\delta)^2\exp\left(
\frac{\gamma\hat x^2}{\gamma t+\delta}
-\frac{3\gamma^2\hat x\hat y}{(\gamma t+\delta)^2}
+\frac{3\gamma^3\hat y^2}{(\gamma t+\delta)^3}
\right)
\\
&\hphantom{\tilde u={}}
\times\exp\big(
3\lambda_3(y-tx)-\lambda_2x-(3\lambda_3^2t^3+3\lambda_3\lambda_2t^2+\lambda_2^2t)
\big)
\big(u+f(t,x,y)\big),
\end{split}
\end{gather}
where
$\hat x:=x+3\lambda_3t^2+2\lambda_2t+\lambda_1$,
$\hat y:=y+\lambda_3t^3+\lambda_2t^2+\lambda_1t+\lambda_0$;
$\alpha$, $\beta$, $\gamma$ and $\delta$ are arbitrary constants with $\alpha\delta-\beta\gamma=1$;
$\lambda_0$,~\dots, $\lambda_3$ and $\sigma$ are arbitrary constants with $\sigma\ne0$,
and $f$ is an arbitrary solution of~\eqref{eq:RemarkableFP}.	
\end{theorem}

\begin{proof}
We should integrate the system~\eqref{eq:EquivalenceGroupFPsuperClass} with~$\bar{\tilde\theta}=\bar\theta=\bar\theta_0$
with respect to the parameter functions~$T$, $X$, $Y$ and~$U^1$.
The equations~\eqref{eq:EquivalenceGroupSuperClassA^2B} take the form
\[
X=\dfrac{Y_t +xY_y}{T_t +xT_y},\quad X_x^2=T_t +xT_y.
\]
In view of the first of these equations, the second equation reduces to $(T_tY_y-T_yY_t)^2=(T_t+xT_y)^5$,
which implies $T_y=0$, $T_t>0$ and
\[
Y=\varepsilon T_t^{3/2}y+Y^0(t),
\]
where $\varepsilon=\pm1$, and~$Y^0$ is a function of~$t$ arising due to the integration with respect to~$y$.
This is why we have $X=(Y_t +xY_y)/T_t$, and hence $X_{xx}=0$.
Then the equation~\eqref{eq:A^1Transformation} simplifies to $X_t+xX_y=-2X_xU^1_x/U^1$ and thus integrates to
\[
U^1=V(t,y)\exp\left(-\frac{T_{tt}}{2T_t}x^2
-3\frac{2T_{ttt}T_t-T^2_{tt}}{8T_t^2}xy-\frac{(Y^0/T_t)_t}{2T_t^{5/2}}x
\right),
\]
where $V$ is a nonvanishing smooth function of~$(t,y)$, the explicit expression for which will be derived below.
Substituting the expression for~$U^1$ into the restricted equation~\eqref{eq:A^0Transformation},
\[
\left(\frac{U^1_x}{U^1}\right)^2-\left(\frac{U^1_x}{U^1}\right)_x+\frac{U^1_t+xU^1_y}{U^1}=0,
\]
leads to an equation whose left-hand side is a quadratic polynomial in~$x$.
Collecting the coefficients of~$x^2$, we derive the equation $T_{ttt}/T_t-\frac32(T_{tt}/T_t)^2=0$,
meaning that the Schwarzian derivative of~$T$ is zero.
Therefore, $T$ is a linear fractional function of~$t$, $T=(\alpha t+\beta)/(\gamma t+\delta)$.
Since the constants $\alpha$, $\beta$, $\gamma$ and $\delta$ are defined up to a constant nonzero multiplier
and $T_t>0$, we can assume that $\alpha\delta-\beta\gamma=1$.
Then the result of collecting the coefficients of~$x$ in the above equation can be represented, up to an inessential multiplier, in the form
\[
\frac{V_y}V=\frac{6\gamma^3}{(\gamma t+\delta)^3}y+\big((\gamma t+\delta)^3Y^0\big)_{ttt}
-6\gamma^2\big((\gamma t+\delta)Y^0\big)_t.
\]
The general solution of the last equation as an equation with respect to~$V$ is
\[
V=\phi(t)\exp\left(\frac{3\gamma^3}{(\gamma t+\delta)^3}y^2
+\big((\gamma t+\delta)^3Y^0\big)_{ttt}y
-6\gamma^2\big((\gamma t+\delta)Y^0\big)_ty\right),
\]
where $\phi$ is a nonvanishing smooth function of~$t$.
Analogously, we collect the summands without~$x$, substitute the above representation for~$V$ into the obtained equation,
split the result with respect to $y$ and in addition neglect inessential multipliers.
This gives the system of two equations
\begin{gather*}
\big((\gamma t+\delta)^3Y^0\big)_{tttt}=0,
\quad
\frac{\phi_t}{\phi}-\frac{2\gamma}{\gamma t+\delta}+\frac1{4}(\gamma t+\delta)^4\Big(\big((\gamma t+\delta)Y^0\big)_{tt}\Big)^2=0,
\end{gather*}
whose general solution can  be represented as
\begin{gather*}
Y^0=\dfrac{\lambda_3t^3+\lambda_2t^2+\lambda_1t+\lambda_0}{(\gamma t+\delta)^3},\\
\phi=\sigma(\gamma t+\delta)^2\exp\left(
\frac{\gamma\psi_t^{\,2}}{\gamma t+\delta}
-\frac{3\gamma^2\psi_t\psi}{(\gamma t+\delta)^2}
+\frac{3\gamma^3\psi^2}{(\gamma t+\delta)^3}
\right)
\exp\left(-\lambda_2^2t-3\lambda_3\lambda_2t^2-3\lambda_3^2t^3\right),
\end{gather*}
where $\psi$ is an arbitrary at most cubic polynomial of~$t$,
$\psi(t):=\lambda_3t^3+\lambda_2t^2+\lambda_1t+\lambda_0$, and $\lambda_0$,~\dots, $\lambda_3$ and~$\sigma$ are arbitrary constants with $\sigma\ne0$.
Since the constant parameters~$\alpha$, $\beta$, $\gamma$ and~$\delta$ are still defined up to the multiplier $\pm1$,
we can choose these parameters in such a way that $\varepsilon|\gamma t+\delta|=\gamma t+\delta$,
and then neglect the parameter~$\varepsilon$.

Finally, the equation~\eqref{eq:CTransformation} takes the form
\[
\left(\dfrac{U^0}{U^1}\right)_t+x\left(\dfrac{U^0}{U^1}\right)_y=\left(\dfrac{U^0}{U^1}\right)_{xx},
\]
and thus $U^0=U^1f$, where $f=f(t,x,y)$ is an arbitrary solution of~\eqref{eq:RemarkableFP}.
\end{proof}

If the transformations of the form~\eqref{eq:RemarkableFPSymGroup} and their compositions
are properly interpreted, then  the structure of the pseudogroup~$G$ is simplified.
Thus, the natural domain of a transformation~$\Phi$ of the form~\eqref{eq:RemarkableFPSymGroup}
is \[\mathop{\rm dom}\Phi=(\mathop{\rm dom}f\times\mathbb R_u)\setminus M_{\gamma\delta},\]
i.e., it can be assumed to coincide with the relative complement
of the set $M_{\gamma\delta}:=\{(t,x,y,u)\in\mathbb R^4\mid\gamma t+\delta=0\}$
with respect to the set $\mathop{\rm dom}f\times\mathbb R_u$, where~$\mathop{\rm dom}f$ is the domain of the function~$f$.
In view of the fact that parameters~$\gamma$ and~$\delta$ do not vanish simultaneously,
the set~$M_{\gamma\delta}$ is empty if $\gamma=0$
and is the hyperplane defined by the equation $t=-\delta/\gamma$ in the space $\mathbb R^4_{t,x,y,u}$ otherwise.
We modify the standard definition of composition in the case of transformations of the form~\eqref{eq:RemarkableFPSymGroup}.
More specifically, according to the standard definition,
the domain of the composition $\Phi_1\circ\Phi_2=:\tilde\Phi$ of transformations~$\Phi_1$ and~$\Phi_2$
is the preimage of the domain of~$\Phi_1$ with respect of~$\Phi_2$.
If the transformations~$\Phi_1$ and~$\Phi_2$ are of the form~\eqref{eq:RemarkableFPSymGroup}, then
\smash{$\mathop{\rm dom}\tilde\Phi=\Phi_2^{-1}(\mathop{\rm dom}\Phi_1)=(\mathop{\rm dom}\tilde f\times\mathbb R_u)
\setminus (M_{\gamma_2\delta_2}\cup M_{\tilde\gamma\tilde\delta})$},
where
$\tilde\gamma=\gamma_1\alpha_2+\delta_1\gamma_2$,
$\tilde\delta=\gamma_1\beta _2+\delta_1\delta_2$,
$\mathop{\rm dom}\tilde f=\big((\pi_*\Phi_2)^{-1}\mathop{\rm dom}f^1\big)\cap\mathop{\rm dom}f^2$,
$\pi$ is the natural projection onto~$\mathbb R^3_{t,x,y}$ in $\mathbb R^4_{t,x,y,u}$,
and the parameters with indices~1 and~2 and tildes correspond to~$\Phi_1$, $\Phi_2$ and~$\tilde\Phi$, respectively.
As the \emph{modified composition} $\Phi_1\circ^{\rm m}\Phi_2$ of transformations~$\Phi_1$ and~$\Phi_2$,
we take the extension of $\Phi_1\circ\Phi_2$ by continuity to the set
\[\smash{\mathop{\rm dom}\nolimits^{\rm m}\tilde\Phi:=(\mathop{\rm dom}\tilde f\times\mathbb R_u)\setminus M_{\tilde\gamma\tilde\delta}},\]
i.e., \smash{$\mathop{\rm dom}(\Phi_1\circ^{\rm m}\Phi_2)=\mathop{\rm dom}^{\rm m}\tilde\Phi$}.
Therefore, $\Phi_1\circ^{\rm m}\Phi_2$ is the transformation of the form~\eqref{eq:RemarkableFPSymGroup}
with the same parameters as in $\Phi_1\circ\Phi_2$ and with its natural domain.
This means that $\mathop{\rm dom}\nolimits^{\rm m}\tilde\Phi=\mathop{\rm dom}\tilde\Phi$
and the extension is trivial if $\gamma_1\gamma_2=0$,
and otherwise we redefine $\Phi_1\circ\Phi_2$ on the set
\smash{$(\mathop{\rm dom}\tilde f\times\mathbb R_u)\cap M_{\gamma_2\delta_2}$}.

Now we can analyze the structure of~$G$.
The point transformations of the form
\[
\mathscr Z(f)\colon\quad \tilde t=t,\quad \tilde x=x,\quad \tilde y=y,\quad \tilde u=u+f(t,x,y),
\]
where the parameter function $f=f(t,x,y)$ is an arbitrary solution of the equation~\eqref{eq:RemarkableFP},
are associated with the linear superposition of solutions of this equation
and thus can be considered as trivial.
They constitute the normal pseudosubgroup $G^{\rm lin}$ of the pseudogroup $G$.
The pseudogroup~$G$ splits over~$G^{\rm lin}$, $G=G^{\rm ess}\ltimes G^{\rm lin}$,
where $G^{\rm ess}$ is the \emph{subgroup} of~$G$ consisting of the transformations of the form~\eqref{eq:RemarkableFPSymGroup} with $f=0$
and with their natural domains, and thus it is an eight-dimensional Lie group.
The Lie group structure of~$G^{\rm ess}$ is the main benefit of redefining the transformation composition above.%
\footnote{%
For the standard transformation composition, we have only the representation $G=\bar G^{\rm ess}\ltimes G^{\rm lin}$,
where $\bar G^{\rm ess}$ is the pseudosubgroup of~$G$
consisting of the transformations of the form~\eqref{eq:RemarkableFPSymGroup} with $f=0$ and with all admitted domains.
Thus, the pseudosubgroup~$\bar G^{\rm ess}$ does not possess a group structure.
}
We call the subgroup~$G^{\rm ess}$ the \emph{essential point symmetry group} of the equation~\eqref{eq:RemarkableFP}.

The subgroup~$G^{\rm ess}$ itself splits over $R$, $G^{\rm ess}=F\ltimes R$.
Here $R$ and~$F$ are the normal subgroup and the subgroup of~$G^{\rm ess}$
that are singled out by the constraints $\alpha=\delta=1$, $\beta=\gamma=0$ and $\lambda_3=\lambda_2=\lambda_1=\lambda_0=0$, $\sigma=1$,
respectively.
They are isomorphic to the groups ${\rm H}(2,\mathbb R)\times\mathbb Z_2$ and ${\rm SL}(2,\mathbb R)$,
and their Lie algebras coincide with~$\mathfrak r\simeq{\rm h}(2,\mathbb R)$ and~$\mathfrak f\simeq{\rm sl}(2,\mathbb R)$.
Here ${\rm H}(2,\mathbb R)$ denotes the rank-two real Heisenberg group.
The normal subgroups~$R_{\rm c}$ and~$R_{\rm d}$ of~$R$ that are isomorphic to~${\rm H}(2,\mathbb R)$ and~$\mathbb Z_2$
are constituted by the elements of~$R$ with parameter values satisfying the constraints
$\sigma>0$ and $\lambda_3=\lambda_2=\lambda_1=\lambda_0=0$, $\sigma\in\{-1,1\}$, respectively.
The isomorphisms of~$F$ to ${\rm SL}(2,\mathbb R)$ and of~$R_{\rm c}$ to~${\rm H}(2,\mathbb R)$ are established by the correspondences
\[
(\alpha,\beta,\gamma,\delta)_{\alpha\delta-\beta\gamma=1}\mapsto
\begin{pmatrix}
\alpha & \beta\\
\gamma & \delta
\end{pmatrix},
\quad
(\lambda_3,\lambda_2,\lambda_1,\lambda_0,\sigma)_{,\;\sigma>0}\mapsto
\begin{pmatrix}
1 & 3\lambda_3 & -\lambda_2 & \ln\sigma\\
0 & 1          & 0          & \lambda_0\\
0 & 0          & 1          & \lambda_1\\
0 & 0          & 0          & 1
\end{pmatrix}.
\]
Thus, $F$ and~$R_{\rm c}$ are connected subgroups of~$G^{\rm ess}$, but~$R_{\rm d}$ is not.
The natural conjugacy action of the group~$F$ on the normal subgroup~$R$ is given by
$(\tilde\lambda_3,\tilde\lambda_2,\tilde\lambda_1,\tilde\lambda_0,\tilde\sigma)
=(\lambda_3,\lambda_2,\lambda_1,\lambda_0,\sigma)\,A$
in the parameterization~\eqref{eq:RemarkableFPSymGroup} of~$G$,
where $A=\varrho_3(\alpha,\beta,\gamma,\delta)\oplus (1)$,
and $\varrho_3$ is the standard real irreducible four-dimensional representation of~${\rm SL}(2,\mathbb R)$,
which can be identified with the action of~${\rm SL}(2,\mathbb R)$ on binary cubics,
\[
\varrho_3\colon\ (\alpha,\beta,\gamma,\delta)_{\alpha\delta-\beta\gamma=1}\mapsto
\begin{pmatrix}
\alpha^3         & 3\alpha^2\beta                       & 3\alpha\beta^2                     & \beta^3\\
\alpha^2\gamma   & 2\alpha\beta\gamma+\alpha^2\delta    & 2\alpha\beta\delta+\beta^2\gamma   & \beta^2\delta\\
\alpha\gamma^2   & 2\alpha\gamma\delta+\beta\gamma^2    & 2\beta\gamma\delta+\alpha\delta^2  & \beta\delta^2\\
\gamma^3         & 3\gamma^2\delta                      & 3\gamma\delta^2                    & \delta^3
\end{pmatrix}.
\]
Summing up, the group~$G^{\rm ess}$ is isomorphic to
$\big({\rm SL}(2,\mathbb R)\ltimes_{\varphi}{\rm H}(2,\mathbb R)\big)\times\mathbb Z_2$,
where the antihomomorphism
$\varphi\colon{\rm SL}(2,\mathbb R)\to{\rm Aut}({\rm H}(2,\mathbb R))$ is defined, in the chosen local coordinates, by
$
\varphi(\alpha,\beta,\gamma,\delta)=(\lambda_3,\lambda_2,\lambda_1,\lambda_0,\sigma)
\mapsto(\lambda_3,\lambda_2,\lambda_1,\lambda_0,\sigma)A.
$

\noprint{
\[
A^{-1}=
\begin{pmatrix}
\delta^3          & -\gamma\delta^2                      & \gamma^2\delta                      & -\gamma^3\\
-3\beta\delta^2   & 2\beta\gamma\delta+\alpha\delta^2    & -\beta\gamma^2-2\alpha\gamma\delta  & 3\alpha\gamma^2 \\
3\beta^2\delta    & -2\alpha\beta\delta-\beta^2\gamma    & \alpha^2\delta+2\alpha\beta\gamma   & -3\alpha^2\gamma\\
\beta^3           & \alpha\beta^2                        & -\alpha^2\beta                      & \alpha^3
\end{pmatrix}.
\]
}


Transformations from the one-parameter subgroups of~$G^{\rm ess}$ that are generated by the basis elements of~$\mathfrak g^{\rm ess}$
given in~\eqref{eq:RemarkableFPEssA} are of the following form:
\[\arraycolsep=0ex
\begin{array}{lllll}
\mathscr P^t(\epsilon)\colon  & \tilde t=t+\epsilon,\ \        & \tilde x=x,                        & \tilde y=y,                        & \tilde u=u,\\[.8ex]
\mathscr D  (\epsilon)\colon  & \tilde t={\rm e}^{2\epsilon}t, & \tilde x={\rm e}^\epsilon x,       & \tilde y={\rm e}^{3\epsilon}y,     & \tilde u={\rm e}^{-2\epsilon}u,\\
\mathscr K(\epsilon)  \colon  & \tilde t=\dfrac t{1\!-\!\epsilon t},\ \
& \tilde x=\dfrac x{1\!-\!\epsilon t}+\dfrac{3\epsilon y}{(1\!-\!\epsilon t)^2},\ \
& \tilde y=\dfrac y{(1\!-\!\epsilon t)^3},\ \
&\tilde u=(1\!-\!\epsilon t)^2{\rm e}^{\frac{\epsilon x^2}{1-\epsilon t}+\frac{3\epsilon^2xy}{(1-\epsilon t)^2}+\frac{3\epsilon^3y^2}{(1-\epsilon t)^3}}u,
\\[1ex]
\mathscr P^3(\epsilon)\colon  & \tilde t=t, & \tilde x=x+3\epsilon t^2,\ \ & \tilde y=y+\epsilon t^3,\ \  & \tilde u={\rm e}^{3\epsilon(y-tx)-3\epsilon^2t^3}u,\\[.8ex]
\mathscr P^2(\epsilon)\colon  & \tilde t=t, & \tilde x=x+2\epsilon t,      & \tilde y=y+\epsilon t^2,     & \tilde u={\rm e}^{-\epsilon x+\epsilon^2t}u,\\[.8ex]
\mathscr P^1(\epsilon)\colon\ & \tilde t=t, & \tilde x=x+\epsilon,         & \tilde y=y+\epsilon t,       & \tilde u=u,\\[.8ex]
\mathscr P^0(\epsilon)\colon  & \tilde t=t, & \tilde x=x,                  & \tilde y=y+\epsilon,         & \tilde u=u,\\[.8ex]
\mathscr I  (\epsilon)\colon  & \tilde t=t, & \tilde x=x,                  & \tilde y=y,\                 & \tilde u={\rm e}^\epsilon u,
\end{array}
\]
where $\epsilon$ is the group parameter.
At the same time, using this basis of~$\mathfrak g^{\rm ess}$ in the course of studying the structure of the group~$G^{\rm ess}$
hides some of its important properties and complicates its study.

Although the pushforward of the group~$G^{\rm ess}$ by the natural projection of~$\mathbb R^4_{t,x,y,u}$ onto~$\mathbb R_t$
coincides with the group of linear fractional transformations of~$t$ and is thus isomorphic to the group ${\rm PSL}(2,\mathbb R)$,
the subgroup~$F$ of~$G^{\rm ess}$ is isomorphic to the group ${\rm SL}(2,\mathbb R)$,
and its Iwasawa decomposition is given by the one-parameter subgroups of~$G^{\rm ess}$
respectively generated by the vector fields~$\mathcal P^t+\mathcal K$, $\mathcal D$ and $\mathcal P^t$.
The first subgroup, which is associated with~$\mathcal P^t+\mathcal K$, consists of the point transformations
\begin{gather}\label{eq:1ParGroupPtK}
\begin{split}&
\tilde t=\frac{\sin\epsilon+t\cos\epsilon}{\cos\epsilon-t\sin\epsilon},
\quad
\tilde x=\frac x{\cos\epsilon-t\sin\epsilon}
+\frac{3y\sin\epsilon}{(\cos\epsilon-t\sin\epsilon)^2},
\quad
\tilde y=\frac y{(\cos\epsilon-t\sin\epsilon)^3},
\\[1ex]&
\tilde u=(\cos\epsilon-t\sin\epsilon)^2\exp\left(
-\frac{x^2\sin\epsilon}{\cos\epsilon-t\sin\epsilon}
-\frac{3xy\sin^2\epsilon}{(\cos\epsilon-t\sin\epsilon)^2}
-\frac{3y^2\sin^3\epsilon}{(\cos\epsilon-t\sin\epsilon)^3}
\right)u
\end{split}
\end{gather}
parameterized by an arbitrary constant~$\epsilon$,
which is defined by the corresponding transformation up to a summand $2\pi k$, $k\in\mathbb Z$.

The equation~\eqref{eq:RemarkableFP} is invariant with respect to the involution~$\mathscr J$ simultaneously alternating the sign of~$(x,y)$,
\[\mathscr J \colon(t,x,y,u)\mapsto(t,-x,-y,u).\]
In the context of the one-parameter subgroups of~$G^{\rm ess}$
that are generated by the basis elements of~$\mathfrak g^{\rm ess}$ listed in~\eqref{eq:RemarkableFPEssA},
the involution~$\mathscr J$ looks like a discrete point symmetry transformation of~\eqref{eq:RemarkableFP},
but in fact this is not the case.
It belongs to the one-parameter subgroup of~$G^{\rm ess}$ generated by $\mathcal P^t+\mathcal K$.
More precisely, it coincides with the transformation~\eqref{eq:1ParGroupPtK} with $\epsilon=\pi$.
The value $\epsilon=-\pi/2$ corresponds to the transformation
\[
\mathscr K'\colon\ \ \tilde t=-\dfrac1t,\ \ \tilde x=\dfrac xt-3\dfrac y{t^2},\ \ \tilde y=\dfrac y{t^3},\ \
\tilde u=t^2{\rm e}^{\frac{x^2}t-\frac{3xy}{t^2}+\frac{3y^2}{t^3}}u,
\]
which also deceptively looks, in the above context, like a discrete point symmetry transformation of~\eqref{eq:RemarkableFP}
being independent with~$\mathscr J$
since the factorization $\mathscr J\circ\mathscr K'=\mathscr P^t(1)\circ\mathscr K(1)\circ\mathscr P^t(1)$ is not intuitive.
This is why it is relevant to accurately describe discrete point symmetries of the equation~\eqref{eq:RemarkableFP}.

\begin{corollary}\label{cor:RemarkableFPDiscretePointSyms}
A complete list of discrete point symmetry transformations of the remarkable Fokker--Planck equation~\eqref{eq:RemarkableFP}
that are independent up to combining with each other and with continuous point symmetry transformations of this equation
is exhausted by the single involution~$\mathscr I'$ alternating the sign of~$u$,
\[\mathscr I'\colon(t,x,y,u)\mapsto(t,x,y,-u).\]
Thus, the quotient group of the complete point symmetry pseudogroup~$G$
with respect to its identity component is isomorphic to $\mathbb Z_2$.
\end{corollary}

\begin{proof}
It is obvious that the entire pseudosubgroup~$G^{\rm lin}$ is contained in the connected component of the identity transformation in~$G$.
The same claim holds for the subgroups~$F$ and~$R_{\rm c}$
in view of their isomorphisms to the groups ${\rm SL}(2,\mathbb R)$ and~${\rm H}(2,\mathbb R)$, respectively.
Therefore, without loss of generality, a complete list of independent discrete point symmetry transformations of~\eqref{eq:RemarkableFP}
can be assumed to consist of elements of the subgroup~$R_{\rm d}$.
Thus, the only discrete point symmetry transformation of~\eqref{eq:RemarkableFP}
that is independent in the above sense is the transformation~$\mathscr I'$.
\end{proof}

Hereafter, the transformations
$\mathscr P^t(\epsilon)$, $\mathscr D(\epsilon)$, $\mathscr K(\epsilon)$,
$\mathscr P^3(\epsilon)$, $\mathscr P^2(\epsilon)$, $\mathscr P^1(\epsilon)$, $\mathscr P^0(\epsilon)$, $\mathscr I(\epsilon)$,
$\mathscr J$, $\mathscr K'$, $\mathscr J'$ and~$\mathscr Z(f)$
are called \emph{elementary point symmetry transformations} of the equation~\eqref{eq:RemarkableFP}.

In view of Theorem~\ref{thm:RemarkableFPSymGroup},
the formal application of the point symmetry transformation
$\Phi:=\mathscr I(\ln\tfrac{\sqrt 3}{2\pi})\circ\mathscr P^0(y')\circ\mathscr P^1(x')\circ\mathscr P^t(t')\circ\mathscr K'$
of the equation~\eqref{eq:RemarkableFP},
\begin{gather}\label{eq:MapFundSolutRemarkableFP}
\begin{split}
&\Phi\colon\quad
\tilde t=-\frac1t+t',\quad
\tilde x=\frac xt-3\frac y{t^2}+x',\quad
\tilde y=\frac y{t^3}+y',
\\
&\hphantom{\Phi\colon\quad}
\tilde u=\frac{\sqrt3}{2\pi}t^2
\exp\left(\frac{x^2}t-3\frac{xy}{t^2}+3\frac{y^2}{t^3}+3x'x-3x'\frac yt+3(x')^2t\right) u,
\noprint{
\\
\hphantom{\Phi\colon\quad}
\tilde u=\frac{\sqrt3}{2\pi(\tilde t-t')^2}
\exp\left(-\frac{(\tilde x-x')^2}{4(\tilde t-t')}-\frac{3\big(\tilde y-y'-\frac12(\tilde x+x')(\tilde t-t')\big)^2}{(\tilde t-t')^3}\right) u,
}
\end{split}
\end{gather}
maps the function $u(t,x,y)=1-H(t)$ to the fundamental solution~\eqref{eq:RemarkableFPFundSolKolm}
of the equation~\eqref{eq:RemarkableFP}.
Recall that $H$ denotes the Heaviside step function.
Note that attempts to interpret this fundamental solution within the framework of group analysis of differential equations
were made in~\cite{gung2018b,kova2013a}.

\section{Classification of inequivalent subalgebras}\label{sec:RemarkableFPClassIneqSubalgebras}

In order to carry out the Lie reductions of codimension one and two for the equation~\eqref{eq:RemarkableFP} in the optimal way,
we need to classify one- and two-dimensional subalgebras of~$\mathfrak g^{\rm ess}$ up to the $G^{\rm ess}$-equivalence.
In the course of this classification, we use the Levi decomposition $\mathfrak g^{\rm ess}=\mathfrak f\lsemioplus\mathfrak r$
and the fact that the Levi factor~$\mathfrak f$ is isomorphic to ${\rm sl}(2,\mathbb R)$.
This allows us to apply the technique of classifying the subalgebras of an algebra with a proper ideal suggested in~\cite{pate1975a}.
This technique becomes simpler if the algebra under consideration can be represented as the semidirect sum of a subalgebra and an ideal,
which are $\mathfrak f$ and~$\mathfrak r$ for the algebra~$\mathfrak g^{\rm ess}$, respectively.
An additional simplification is that an optimal list of subalgebras of~$\rm{sl}(2,\mathbb R)$ is well known
(see, e.g., \cite{pate1977a,popo2003a}).
Thus, for the realization~$\mathfrak f$ of~$\rm{sl}(2,\mathbb R)$, this list consists of the subalgebras
$\{0\}$, $\langle\mathcal P^t\rangle$,
$\langle\mathcal D\rangle$,
$\langle\mathcal P^t+\mathcal K\rangle$,
$\langle\mathcal P^t, \mathcal D\rangle$
and $\mathfrak f$ itself.
The subalgebras~$\mathfrak s_1$ and~$\mathfrak s_2$ of~$\mathfrak g^{\rm ess}$
are definitely $G^{\rm ess}$-inequivalent
if their projections~$\pi_{\mathfrak f}\mathfrak s_1$ and~$\pi_{\mathfrak f}\mathfrak s_2$ are $F$-inequivalent,
see Section~\ref{sec:RemarkableFPPointSymGroup}.
Here and in what follows, $\pi_{\mathfrak f}$ and $\pi_{\mathfrak r}$ denote
the natural projections of~$\mathfrak g^{\rm ess}$ onto~$\mathfrak f$ and~$\mathfrak r$
according to the decomposition $\mathfrak g^{\rm ess}=\mathfrak f\dotplus\mathfrak r$ of the vector space~$\mathfrak g^{\rm ess}$
as the direct sum of its subspaces~$\mathfrak f$ and~$\mathfrak r$.
In other words, we can use the projections~$\pi_{\mathfrak f}\mathfrak s$ of the subalgebras~$\mathfrak s$
of~$\mathfrak g^{\rm ess}$ of the same dimension for partitioning the set of these subalgebras into subsets
such that each subset contains no subalgebra being equivalent to a subalgebra from another subset.

Let us specifically describe applying the above technique to the classification
of one- and two-dimensional subalgebras of~$\mathfrak g^{\rm ess}$ up to the $G^{\rm ess}$-equivalence.
We fix the dimension~$d$ of subalgebras~$\mathfrak s$ to be classified, either $d=1$ or $d=2$,
and consider the subalgebras~$\mathfrak s_{\mathfrak f}$ of~$\mathfrak f$ from the above list with dimension~$d'$ less than or equal to~$d$.
For each of these subalgebras, we take the set of $d$-dimensional subalgebras of~$\mathfrak g^{\rm ess}$
with $\pi_{\mathfrak f}\mathfrak s=\mathfrak s_{\mathfrak f}$
and construct a complete list of $G^{\rm ess}$-inequivalent subalgebras in this set.

One usually classifies subalgebras of a Lie algebra up to their equivalence generated by the group of inner automorphisms of this algebra,
and this group is computed by summing up Lie series or solving Cauchy problems with the adjoint action of the algebra on itself,
see, e.g., \cite[Section~3.3]{olve1993A}.
At the same time, we already know that the algebra~$\mathfrak g^{\rm ess}$ is the Lie algebra of the group~$G^{\rm ess}$,
and the actions of~$G^{\rm ess}$ and of the group of inner automorphisms of~$\mathfrak g^{\rm ess}$ on~$\mathfrak g^{\rm ess}$ coincide.
Moreover, recall that for the purpose of finding Lie invariant solutions of the equation~\eqref{eq:RemarkableFP}
subalgebras of~$\mathfrak g^{\rm ess}$ have in fact to be classified up to the $G^{\rm ess}$-equivalence.
This is why following, e.g., \cite{card2011a}, we directly compute the action of~$G^{\rm ess}$ on~$\mathfrak g^{\rm ess}$
by pushing forward vector fields from~$\mathfrak g^{\rm ess}$ by transformations from~$G^{\rm ess}$.
The nonidentity pushforwards of basis elements of~$\mathfrak g^{\rm ess}$ by the elementary transformations from~$G^{\rm ess}$
are the following:
\begin{gather*}\arraycolsep=0ex
\begin{array}{l}
    \mathscr P^t(\epsilon)_*\mathcal D  =\mathcal D-2\epsilon\mathcal P^t,\\[.3ex]
    \mathscr P^t(\epsilon)_*\mathcal K  =\mathcal K-\epsilon\mathcal D+\epsilon^2\mathcal P^t,\\[.3ex]
    \mathscr P^t(\epsilon)_*\mathcal P^3=\mathcal P^3-3\epsilon\mathcal P^2+3\epsilon^2\mathcal P^1-\epsilon^3\mathcal P^0,\\[.3ex]
    \mathscr P^t(\epsilon)_*\mathcal P^2=\mathcal P^2-2\epsilon\mathcal P^1+\epsilon^2\mathcal P^0,\\[.3ex]
    \mathscr P^t(\epsilon)_*\mathcal P^1=\mathcal P^1-\epsilon\mathcal P^0,
\end{array}
\qquad
\begin{array}{l}
	\mathscr K(\epsilon)_*\mathcal D  =\mathcal D+2\epsilon\mathcal K,\\[.3ex]
	\mathscr K(\epsilon)_*\mathcal P^t=\mathcal P^t+\epsilon\mathcal D+\epsilon^2\mathcal K,\\[.3ex]
	\mathscr K(\epsilon)_*\mathcal P^2=\mathcal P^2+\epsilon\mathcal P^3,\\[.3ex]
	\mathscr K(\epsilon)_*\mathcal P^1=\mathcal P^1+2\epsilon\mathcal P^2+\epsilon^2\mathcal P^3,\\[.3ex]
	\mathscr K(\epsilon)_*\mathcal P^0=\mathcal P^0+3\epsilon\mathcal P^1+3\epsilon^2\mathcal P^2+\epsilon^3\mathcal P^3,
\end{array}
\\[1ex]\arraycolsep=0ex
\begin{array}{l}
    \mathscr D(\epsilon)_*\mathcal P^t=e^{2\epsilon}\mathcal P^t,\\[.3ex]	
    \mathscr D(\epsilon)_*\mathcal K  =e^{-2\epsilon}\mathcal K,
\end{array}\qquad
\begin{array}{l}
    \mathscr D(\epsilon)_*\mathcal P^3=e^{-3\epsilon}\mathcal P^3,\\[.3ex]
    \mathscr D(\epsilon)_*\mathcal P^2=e^{-\epsilon}\mathcal P^2,
\end{array}\qquad
\begin{array}{l}
    \mathscr D(\epsilon)_*\mathcal P^1=e^{\epsilon}\mathcal P^1,\\[.3ex]
    \mathscr D(\epsilon)_*\mathcal P^0=e^{3\epsilon}\mathcal P^0,
\end{array}
\\[1ex]\arraycolsep=0ex	
\begin{array}{l}
	\mathscr P^3(\epsilon)_*\mathcal P^t=\mathcal P^t+3\epsilon\mathcal P^2,\\[.3ex]
	\mathscr P^3(\epsilon)_*\mathcal D  =\mathcal D+3\epsilon\mathcal P^3,\\[.3ex]
	\mathscr P^3(\epsilon)_*\mathcal P^0=\mathcal P^0+3\epsilon\mathcal I,
\end{array}
\qquad
\begin{array}{l}
	\mathscr P^0(\epsilon)_*\mathcal D=\mathcal D-3\epsilon\mathcal P^0,\\[.3ex]
	\mathscr P^0(\epsilon)_*\mathcal K=\mathcal K-3\epsilon\mathcal P^1,\\[.3ex]
	\mathscr P^0(\epsilon)_*\mathcal P^3=\mathcal P^3-3\epsilon\mathcal I,
\end{array}
\\[1ex]\arraycolsep=0ex	
\begin{array}{l}
    \mathscr P^2(\epsilon)_*\mathcal P^t=\mathcal P^t+2\epsilon\mathcal P^1-\epsilon^2\mathcal I,\\[.3ex]
    \mathscr P^2(\epsilon)_*\mathcal D  =\mathcal D+\epsilon\mathcal P^2,\\[.3ex]
    \mathscr P^2(\epsilon)_*\mathcal K  =\mathcal K-\epsilon\mathcal P^3,\\[.3ex]
    \mathscr P^2(\epsilon)_*\mathcal P^1=\mathcal P^1-\epsilon\mathcal I,
\end{array}
\qquad
\begin{array}{l}
    \mathscr P^1(\epsilon)_*\mathcal P^t=\mathcal P^t+\epsilon\mathcal P^0,\\[.3ex]
    \mathscr P^1(\epsilon)_*\mathcal D  =\mathcal D-\epsilon\mathcal P^1,\\[.3ex]
    \mathscr P^1(\epsilon)_*\mathcal K  =\mathcal K-2\epsilon\mathcal P^2-\epsilon^2\mathcal I,\\[.3ex]
    \mathscr P^1(\epsilon)_*\mathcal P^2=\mathcal P^2+\epsilon\mathcal I,
\end{array}
\\[1ex]
\mathscr J_*(\mathcal P^3,\mathcal P^2,\mathcal P^1,\mathcal P^0)=(-\mathcal P^3,-\mathcal P^2,-\mathcal P^1,-\mathcal P^0),
\\
\mathscr K'_*(\mathcal P^t,\mathcal D,\mathcal K,\mathcal P^3,\mathcal P^2,\mathcal P^1,\mathcal P^0)
=(\mathcal K,-\mathcal D,\mathcal P^t,\mathcal P^0,-\mathcal P^1,\mathcal P^2,-\mathcal P^3).
\end{gather*}

The result of the classification of one- and two-dimensional subalgebras of~$\mathfrak g^{\rm ess}$ up to the $G^{\rm ess}$-equivalence
is presented in the two subsequent lemmas.

\begin{lemma}\label{lem:RemarkableFP1DSubalgs}
A complete list of $G^{\rm ess}$-inequivalent one-dimensional subalgebras of $\mathfrak g^{\rm ess}$ is exhausted by the subalgebras
\begin{gather*}
\mathfrak s_{1.1}=\langle\mathcal P^t+\mathcal P^3\rangle,\ \
\mathfrak s_{1.2}^\delta=\langle\mathcal P^t+\delta\mathcal I\rangle,\ \
\mathfrak s_{1.3}^\nu=\langle\mathcal D+\nu\mathcal I\rangle,\ \
\mathfrak s_{1.4}^\mu=\langle\mathcal P^t+\mathcal K+\mu\mathcal I\rangle,\\
\mathfrak s_{1.5}^\varepsilon=\langle\mathcal P^2+\varepsilon\mathcal P^0\rangle,\ \
\mathfrak s_{1.6}=\langle\mathcal P^1\rangle,\ \
\mathfrak s_{1.7}=\langle\mathcal P^0\rangle,\ \
\mathfrak s_{1.8}=\langle\mathcal I\rangle,
\end{gather*}
where $\varepsilon\in\{-1,1\}$, $\delta\in\{-1,0,1\}$, and $\mu$ and~$\nu$ are arbitrary real constants with $\nu\geqslant0$.
\end{lemma}

\begin{proof}
A complete set of $F$-inequivalent $d'$-dimensional subalgebras of~$\mathfrak f$ with $d'\leqslant1$
consists of the subalgebras
$\langle\mathcal P^t\rangle$,
$\langle\mathcal D\rangle$,
$\langle\mathcal P^t+\mathcal K\rangle$ and~$\{0\}$.
Therefore, without loss of generality we can consider only one-dimensional subalgebras of~$\mathfrak g^{\rm ess}$
with basis vector fields~$Q$ of the general form
\[
Q = \hat Q+a_3\mathcal P^3+3a_2\mathcal P^2+3a_1\mathcal P^1+a_0\mathcal P^0+b\mathcal I,
\]
where $\hat Q\in\{\mathcal P^t,\mathcal D,P^t+\mathcal K,0\}$ and $a_0$, \dots, $a_3$ and $b$ are real constants.
We factor out the multiplier 3 in the coefficients of~$\mathcal P^1$ and~$\mathcal P^2$
since then the coefficients~$a_3$, $a_2$, $a_1$ and~$a_0$ are changed under the action of~$F$
in the same way as the coefficients of the real cubic binary form $a_3x^3+3a_2x^2y+3a_1xy^2+a_0y^3$
are changed under the standard action of the group of linear fractional transformations on such forms,
cf.~\cite[Example 2.22]{olve1999A}.
We need to further simplify the vector fields~$Q$ of the above form
by acting with elements of~$G^{\rm ess}$ on them and/or scaling them.

Let $\hat Q=\mathcal P^t$.
If $a_3\ne0$, then we successively set $a_3=1$, $a_2=a_1=a_0=0$ and $b=0$
using $\mathscr D(e^{\epsilon_4})_*$ simultaneously with rescaling~$Q$,
$\mathscr P^1(\epsilon_1)_*\circ\mathscr P^2(\epsilon_2)_*\circ\mathscr P^3(\epsilon_3)_*$ and $\mathscr P^0(\epsilon_0)_*$
with appropriate constants~$\epsilon_i$, $i=0,\dots,4$, respectively.
This gives the subalgebra $\mathfrak s_{1.1}$.
Similarly, for $a_3=0$, we apply $\mathscr P^3(\epsilon_2)_*$, $\mathscr P^2(\epsilon_1)_*$ and $\mathscr P^1(\epsilon_0)_*$
with appropriate constants~$\epsilon_i$, $i=1,2,3$, to get rid of $a_2$, $a_1$ and $a_0$.
Then, the action of $\mathscr D(e^{\epsilon_b})_*$ with appropriate $\epsilon_b$ and rescaling~$Q$ allow us to set $b\in\{-1,1\}$ if $b\ne0$.
Thus, we obtain the subalgebras~$\mathfrak s_{1.2}^\delta$.

For $\hat Q=\mathcal D$ and $\hat Q=\mathcal P^t+\mathcal K$,
removing the summands with $\mathcal P^3$, $\mathcal P^2$, $\mathcal P^1$ and $\mathcal P^0$ is analogous.
The only possible further simplification is that in the case $\hat Q=\mathcal D$,
the sign of the obtained value of~$b$ can be alternated using $\mathscr K'_*$ and alternating the sign of~$Q$.
This results in the collections of the subalgebras~$\mathfrak s_{1.3}^\nu$ and~$\mathfrak s_{1.4}^\mu$, respectively.

Let $\hat Q=0$.
The classification of real binary cubics presented in the second table on page~26 in~\cite{olve1999A} implies that,
up to the $F$-equivalence and rescaling~$Q$, we have $a_3=0$ and $(3a_2,3a_1,a_0)\in\{(1,\varepsilon,0),(0,1,0),(0,0,1),(0,0,0)\}$.
If $(a_3,a_2,a_1,a_0)\ne(0,0,0,0)$, then using one of the actions
$\mathscr P^3(\epsilon)_*$, $\mathscr P^2(\epsilon)_*$, $\mathscr P^1(\epsilon)_*$ and $\mathscr P^0(\epsilon)_*$
we can set $b=0$; otherwise we set $b=1$ by rescaling~$Q$.
Thus, we obtain the rest of subalgebras from the list given in the lemma's statement.
\end{proof}

\begin{lemma}\label{lem:RemarkableFP2DSubalgs}
A complete list of $G^{\rm ess}$-inequivalent two-dimensional subalgebras of $\mathfrak g^{\rm ess}$ is given~by
\begin{gather*}
\mathfrak s_{2.1}^\mu=\langle\mathcal P^t,\mathcal D+\mu\mathcal I\rangle,
\ \
\mathfrak s_{2.2}^\delta=\langle\mathcal P^t+\delta\mathcal I,\mathcal P^0\rangle,\ \
\mathfrak s_{2.3}=\langle\mathcal P^t,\mathcal P^0+\mathcal I\rangle,\\
\mathfrak s_{2.4}^\mu=\langle\mathcal D+\mu\mathcal I,\mathcal P^1\rangle,\ \
\mathfrak s_{2.5}^\mu=\langle\mathcal D+\mu\mathcal I,\mathcal P^0\rangle,\ \
\\
\mathfrak s_{2.6}=\langle\mathcal P^0,\mathcal P^1\rangle,\ \
\mathfrak s_{2.7}=\langle\mathcal P^0,\mathcal P^2\rangle,\ \
\mathfrak s_{2.8}^\varepsilon=\langle\mathcal P^1,\mathcal P^3+\varepsilon\mathcal P^0\rangle,\ \
\\
\mathfrak s_{2.9}=\langle\mathcal P^t+\mathcal P^3,\mathcal I\rangle,\ \
\mathfrak s_{2.10}^\delta=\langle\mathcal P^t,\mathcal I\rangle,\ \
\mathfrak s_{2.11}=\langle\mathcal D,\mathcal I\rangle,\ \
\mathfrak s_{2.12}=\langle\mathcal P^t+\mathcal K,\mathcal I\rangle,
\\
\mathfrak s_{2.13}^\varepsilon=\langle\mathcal P^2+\varepsilon\mathcal P^0,\mathcal I\rangle,\ \
\mathfrak s_{2.14}=\langle\mathcal P^1,\mathcal I\rangle,\ \
\mathfrak s_{2.15}=\langle\mathcal P^0,\mathcal I\rangle,
\end{gather*}
where $\varepsilon\in\{-1,1\}$, $\delta\in\{-1,0,1\}$, and $\mu$ is an arbitrary real constant.
\end{lemma}

\begin{proof}
A complete set of $F$-inequivalent $d'$-dimensional subalgebras of~$\mathfrak f$ with $d'\leqslant2$
consists of the subalgebras
$\langle\mathcal P^t,\mathcal D\rangle$,
$\langle\mathcal P^t\rangle$,
$\langle\mathcal D\rangle$,
$\langle\mathcal P^t+\mathcal K\rangle$ and~$\{0\}$.
Therefore, without loss of generality, we can consider only two-dimensional subalgebras of~$\mathfrak g^{\rm ess}$
with basis vector fields of the general form $Q_i=\hat Q_i+\check Q_i$, $i=1,2$,
where
\[
(\hat Q_1,\hat Q_2)\in\big\{(\mathcal P^t,\mathcal D),\,(\mathcal P^t,0),\,(\mathcal D,0),\,(\mathcal P^t+\mathcal K,0),\,(0,0)\big\},
\]
and $\check Q_i=a_{i3}\mathcal P^3+3a_{i2}\mathcal P^2+3a_{i1}\mathcal P^1+a_{i0}\mathcal P^0+b_i\mathcal I$
with real constants $a_{ij}$ and~$b_i$, $i=1,2$, $j=0,\dots,3$.
We further need to simplify $\check Q_i$ by
linearly recombining the vector fields~$Q_i$ and
acting on them by elements from~$G^{\rm ess}$,
while simultaneously maintaining the closedness of the subalgebras with respect to the Lie bracket,
$[Q_1,Q_2]\in\langle Q_1,Q_2\rangle$, under the condition $\dim\langle Q_1,Q_2\rangle=2$
and preserving~$\hat Q_i$.

We separately consider each of the values of $(\hat Q_1,\hat Q_2)$ in the above set.

\medskip\par\noindent
$\boldsymbol{(\mathcal P^t,\mathcal D).}$
Then according to Lemma~\ref{lem:RemarkableFP1DSubalgs},
up to the $G^{\rm ess}$-equivalence,
the vector field $Q_1$ is equal to either $\mathcal P^t+\mathcal P^3$ or $\mathcal P^t+\delta\mathcal I$.
The first value is not appropriate, since then
$[Q_1,Q_2]=2\mathcal P^t-3\mathcal P^3+\tilde Q$ with $\tilde Q\in\langle\mathcal P^2,\mathcal P^1,\mathcal P^0,\mathcal I\rangle$
and thus $[Q_1,Q_2]\notin\langle Q_1,Q_2\rangle$ for any $\check Q_2\in\mathfrak r$.
For the second value, requiring the condition $[Q_1,Q_2]\in\langle Q_1,Q_2\rangle$ implies
$a_{23}=a_{22}=a_{21}=\delta=0$.
Acting by $\mathscr P^0(a_{20})_*$ on $Q_2$, we can set $a_{20}=0$.
Thus, we obtain the subalgebra~$\mathfrak s_{2.1}^\mu$ with $\mu=b_2$.

\medskip\par
In all the other cases, we have $\hat Q_2=0$.
If, in addition, $a_{2j}=0$, $j=0,\dots,3$, then $Q_2=\mathcal I$,
and thus the basis elements~$Q_1$ and~$Q_2$ commute.
The corresponding canonical forms of~$Q^1$ modulo the $G^{\rm ess}$-equivalence
are given by the basis elements of the one-dimensional algebras~$\mathfrak s_{1.1}$,~\dots,~$\mathfrak s_{1.7}$
from Lemma~\ref{lem:RemarkableFP1DSubalgs}
up to neglecting the summands with~$\mathcal I$
due to the possibility of linearly combining~$Q_1$ with~$Q_2$.
This results in the subalgebras $\mathfrak s_{2.9}$, \dots, $\mathfrak s_{2.15}$.
Hence below we assume that $(a_{23},a_{22},a_{21},a_{20})\ne(0,0,0,0)$
and, moreover, $\mathop{\rm rank}(a_{ij})^{i=1,2}_{j=0,\dots,3}=2$ if $\hat Q_1=\hat Q_2=0$.

\medskip\par\noindent
$\boldsymbol{(\mathcal P^t,0).}$
The closedness of $\langle Q_1,Q_2\rangle$ with respect to the Lie bracket implies $a_{23}=a_{22}=a_{21}=a_{13}a_{20}=0$
and thus $a_{13}=0$ since then $a_{20}\ne0$.
Hence $Q_1=\mathcal P^t+b_1\mathcal I$ and $Q_2=\mathcal P^0+b_2\mathcal I$.

If $b_2=0$, we can rescale $b_1$ to~$\delta$ by combining $\mathscr D(\epsilon)_*$
with scaling the entire vector field~$Q^1$.
This results in the subalgebras $\mathfrak s_{2.2}^\delta$.

Let $b_2\ne0$.
Simultaneously acting by~$\mathscr D(\epsilon)_*$ with the suitable value of $\epsilon$, rescaling~$b_2$
and, if $b_2<0$, acting by~$\mathscr J_*$ and alternating the sign of~$Q^2$,
we can set $b_2=1$.
Then the pushforward $\mathscr P^1(b_1)_*$ preserves~$Q^2$,
and $\mathscr P^1(b_1)_*Q_1=\mathcal P^t+b_1 Q^2$.
After linearly combining~$Q_1$ with~$Q^2$, we obtain the subalgebras $\mathfrak s_{2.3}$.

\medskip\par\noindent
$\boldsymbol{(\mathcal D,0).}$
In view of Lemma~\ref{lem:RemarkableFP1DSubalgs},
we can reduce $Q_1$ to $\mathcal D+\mu\mathcal I$ modulo the $G^{\rm ess}$-equivalence.
The condition $[Q_1,Q_2]\in\langle Q_1,Q_2\rangle$
and the commutation relations of~$\mathcal D$ with basis elements of~$\mathfrak r$ imply
that up to scaling of~$Q^2$, we have
\[
(a_{23},3a_{22},3a_{21},a_{20},b_2)\in\{(1,0,0,0,0),(0,1,0,0,0),(0,0,1,0,0),(0,0,0,1,0)\},
\]
where the first and the second values are reduced by~$\mathscr K'_*$ to the third and the fourth values, respectively.
Thus, we obtain the subalgebras~$\mathfrak s_{2.4}^\mu$ and~$\mathfrak s_{2.5}^\mu$.

\medskip\par\noindent
$\boldsymbol{(\mathcal P^t+\mathcal K,0).}$
This pair is not appropriate since, from the closedness of $\langle Q_1,Q_2\rangle$ with respect to the Lie bracket,
we straightforwardly derive $a_{2j}=0$, $j=0,\dots,3$.

\medskip\par\noindent
$\boldsymbol{(0,0).}$
Using Lemma~\ref{lem:RemarkableFP1DSubalgs}, from the very beginning we can set
$Q_1\in\{\mathcal P^2+\varepsilon\mathcal P^0,\mathcal P^1,\mathcal P^0\}$.
Since $\mathop{\rm rank}(a_{ij})^{i=1,2}_{j=0,\dots,3}=2$ here,
simultaneously with $b_1=0$ we can set $b_2=0$ using the suitable pushforward $\mathscr P^1(\epsilon)_*$.

\newcommand{\Discr}{\mathop{\rm Discr}}
For $Q=a_3\mathcal P^3+3a_2\mathcal P^2+3a_1\mathcal P^1+a_0\mathcal P^0$ with $\bar a:=(a_0,a_1,a_2,a_3)\ne(0,0,0,0)$,
we denote
\[
\Discr(Q):=a_0^2a_3^2-6a_0a_1a_2a_3+4a_0a_2^3-3a_1^2a_2^2+4a_1^3a_3.
\]
The association with cubic binary forms implies that
$Q\in\{\mathcal P^0,\mathcal P^1\}$ modulo the $G^{\rm ess}$-equivalence if and only if $\Discr(Q)=0$,
cf.\ \cite[Example~2.22]{olve1999A}.
To distinguish inequivalent subalgebras in the present classification case,
we will consider lines in the space run by~$\bar a$, where $\Discr(Q)=0$.
We call such lines singular.
It is obvious that two-dimensional subalgebras in $\langle\mathcal P^0,\mathcal P^1,\mathcal P^2,\mathcal P^3\rangle$
are $G^{\rm ess}$-inequivalent if they possess different numbers of singular lines associated with their elements
or, more precisely, different sets of multiplicities of such singular lines.

If $Q_1=\mathcal P^0$, then the condition $[Q_1,Q_2]\in\langle Q_1,Q_2\rangle$ implies $a_{23}=0$,
and we can set $a_{20}=0$ at any step by linearly combining~$Q_2$ with~$Q_1$.

The case $a_{22}=0$ corresponds to the subalgebra~$\mathfrak s_{2.6}$.
$\Discr(Q)=0$ for any $Q\in\mathfrak s_{2.6}$,
i.e., the subalgebra~$\mathfrak s_{2.6}$ possesses an infinite number of singular lines.

For $a_{22}\ne0$, we divide~$Q_2$ by $a_{22}$,
act by $\mathscr P^t(\epsilon)_*$ with $\epsilon=a_{21}/(2 a_{22})$
and reset $a_{20}=0$, linearly combining~$Q_2$ with~$Q_1$,
which gives the basis of the subalgebra~$\mathfrak s_{2.7}$.
For an arbitrary vector field $Q=a_0\mathcal P^0+3a_2\mathcal P^2$ in $\mathfrak s_{2.7}$,
we obtain $\Discr(Q)=4a_0a_2^{\,3}$,
and thus the subalgebra~$\mathfrak s_{2.7}$ possesses
one singular line of multiplicity one and
one singular line of multiplicity three.

If $Q_1=\mathcal P^1$, then the condition $[Q_1,Q_2]\in\langle Q_1,Q_2\rangle$ implies $a_{22}=0$,
and we can set $a_{21}=0$ at any step, linearly combining~$Q_2$ with~$Q_1$.
We can assume $a_{23}\ne0$ since otherwise we again obtain the subalgebra~$\mathfrak s_{2.6}$.
Rescaling the basis element~$Q_2$, we set $a_{23}=1$.
The subalgebra $\langle\mathcal P^1,\mathcal P^3\rangle$ is mapped by~$\mathscr K'_*$ to the subalgebra~$\mathfrak s_{2.7}$.
This is why the coefficient~$a_{20}$ should be nonzero as well.
The action by $\mathscr D(\epsilon)_*$ with the suitable value of~$\epsilon$ allows us to set $a_{20}=\varepsilon$,
i.e., we derive the subalgebras~$\mathfrak s_{2.8}^\varepsilon$.
For an arbitrary vector field $Q=3a_1\mathcal P^1+a_3(\mathcal P^3+\varepsilon\mathcal P^0)$ in $\mathfrak s_{2.8}^\varepsilon$,
we obtain $\Discr(Q)=a_3(4a_1^3+a_3^3)$.
This means that each of the subalgebras~$\mathfrak s_{2.8}^\varepsilon$ possesses two singular line of multiplicity one.

The consideration of the case $Q_1=\mathcal P^2+\varepsilon\mathcal P^0$ is the most complicated.
The span $\langle Q_1,Q_2\rangle$ is closed with respect to the Lie bracket if and only if $a_{21}=-\varepsilon a_{23}$.
The coefficient~$a_{22}$ is set to zero by linearly combining~$Q_2$ with~$Q_1$.
If $a_{23}=0$, then $a_{21}=0$ as well, and we again obtain the subalgebra~$\mathfrak s_{2.7}$.
Let then further $a_{23}\ne0$.
We scale $Q_2$ to set $a_{23}=1$,
which leads to $Q_2=\mathcal P^3-3\varepsilon\mathcal P^1+\alpha\mathcal P^0$ with $\alpha:=a_{20}$.
We intend to show that depending on the value of~$(\varepsilon,\alpha)$,
the subalgebra $\mathfrak s:=\langle Q_1,Q_2\rangle$ is $G^{\rm ess}$-equivalent to
either the subalgebra~$\mathfrak s_{2.7}$ or the subalgebra~$\mathfrak s_{2.8}$.

We compute that
\[
\Discr(Q)=12\varepsilon a_2^{\,4}+4\alpha a_3a_2^{\,3}+24a_3^{\,2}a_2^{\,2}
+12\varepsilon\alpha a_3^{\,3}a_2+a_3^{\,4}\alpha^2-4a_3^{\,4}\varepsilon
\]
for a nonzero vector field $Q=3a_2Q_1+a_3Q_2$ in $\mathfrak s$.
If $a_3=0$ and thus $a_2\ne0$, then $\Discr(Q)=12\varepsilon a_2^{\,4}\ne0$,
i.e., the corresponding line is not singular.
Hence we can further set $a_3\ne0$ and assume without loss of generality that $a_3=1$.
Thus, finding singular lines associated with the subalgebra~$\mathfrak s$ reduces
to solving the quartic equation $\Discr(Q)=0$ with respect to $a_2$.
The substitution $a_2=z-\varepsilon\alpha/12$ and the division of the left-hand side by $12\varepsilon$ reduce this equation
to its monic depressed counterpart $z^4+pz^2+qz+r=0$, where
\[
p:=-\frac1{24}\alpha^2+2\varepsilon,\quad
q:=\frac\varepsilon{216}\alpha^3+\frac23\alpha,\quad
r:=-\frac1{6912}\alpha^4+\frac\varepsilon{72}\alpha^2-\frac13.
\]
To find the number of (real) roots of this equation, following~\cite[p.~53]{arno1988a}
we compute
\begin{gather*}
\delta:=256r^3-128p^2r^2+144pq^2r+16p^4r-27q^4-4p^3q^2=
-\frac1{432}(\alpha^2+16\varepsilon)^4,\\
L:=8pr-9q^2-2p^3=-\frac\varepsilon{12}(\alpha^2+16\varepsilon)^2.
\end{gather*}

Thus, there are only two cases for the number of roots of the equation $\Discr(Q)=0$.
Note that in both cases, this number is not zero.
Therefore, as indicated above, the subalgebra~$\mathfrak s$ contains,
modulo the $G^{\rm ess}$-equivalence, at least one of the vector fields~$\mathcal P^0$ and~$\mathcal P^1$,
and thus we can use the considerations of the cases where $Q_1=\mathcal P^0$ or  $Q_1=\mathcal P^1$.
This is why the recognition of the corresponding representatives among subalgebras listed in the lemma's statement
depends only on the root structure of the equation $\Discr(Q)=0$.

1. If either $\varepsilon=-1$ and $\alpha\ne\pm4$ or $\varepsilon=1$,
then $\delta<0$, and hence the equation $\Discr(Q)=0$
has two roots of multiplicity one, which corresponds to
two singular lines of multiplicity one for the subalgebra~$\mathfrak s$.
Therefore, this algebra is $G^{\rm ess}$-equivalent to one of the subalgebras~$\mathfrak s_{2.8}^\varepsilon$.

2. If $\varepsilon=-1$ and $\alpha=\pm4$,
then $\delta=L=0$ and $p<0$, and thus the equation $\Discr(Q)=0$
has one root of multiplicity one and one root of multiplicity three, which corresponds to
one singular line of multiplicity one and one singular line of multiplicity three for the subalgebra~$\mathfrak s$.
Hence this algebra is $G^{\rm ess}$-equivalent to the subalgebra~$\mathfrak s_{2.7}$.
\end{proof}

We can also classify one and two-dimensional subalgebras of the entire algebra~$\mathfrak g$ up to the $G$-equivalence
and show that only subalgebras of~$\mathfrak g^{\rm ess}$ are essential
in the course of classifying Lie reductions of the equation~\eqref{eq:RemarkableFP}.
Thus, according to the decomposition $G=G^{\rm ess}\ltimes G^{\rm lin}$ (see Section~\ref{sec:RemarkableFPPointSymGroup}),
an arbitrary transformation~$\Phi$ from $G$ can be represented in the form $\Phi=\mathscr F\circ\mathscr Z(f)$,
where $\mathscr F\in G^{\rm ess}$ and $\mathscr Z(f)\in G^{\rm lin}$.
To exhaustively describe the adjoint action of~$G$ on the algebra~$\mathfrak g$,
in view of the decomposition $\mathfrak g=\mathfrak g^{\rm ess}\lsemioplus\mathfrak g^{\rm lin}$
it suffices to supplement the adjoint action of~$G^{\rm ess}$ on~$\mathfrak g^{\rm ess}$
with the adjoint actions of $G^{\rm ess}$ on $\mathfrak g^{\rm lin}$ and of $G^{\rm lin}$ on $\mathfrak g^{\rm ess}$,
$\mathscr Z(f)_*Q=Q+Q[f]\p_u$ and
$\mathscr F_*\mathcal Z(f)=\mathcal Z(\mathscr F_*f)$,
whereas the adjoint action of $G^{\rm lin}$ on $\mathfrak g^{\rm lin}$ is trivial.
Here $Q$ is an arbitrary vector field from~$\mathfrak g^{\rm ess}$,
$\mathcal Q[f]$ denotes the evaluation of the characteristic~$Q[u]$ of~$Q$ at $u=f$,
and $\mathscr F$ is an arbitrary transformation from~$G^{\rm ess}$.

The classification of subalgebras of~$\mathfrak g$ is based on the classification of subalgebras of~$\mathfrak g^{\rm ess}$.
This is due to the fact that
subalgebras $\mathfrak s_1$ and $\mathfrak s_2$ of $\mathfrak g$ are definitely $G$-inequivalent if
$\pi_{\mathfrak g^{\rm ess}}\mathfrak s_1$ and $\pi_{\mathfrak g^{\rm ess}}\mathfrak s_2$ are $G^{\rm ess}$-inequivalent.
Here $\pi_{\mathfrak g^{\rm ess}}$ denotes the natural projection of $\mathfrak g$
onto $\mathfrak g^{\rm ess}$ under the decomposition $\mathfrak g=\mathfrak g^{\rm ess}\dotplus\mathfrak g^{\rm lin}$
in the sense of vector spaces.
In the course of classifying the subalgebras~$\mathfrak s$ of~$\mathfrak g$ of a fixed (finite) dimension,
we partition the set of these subalgebras into subsets such that
each subset consists of the subalgebras with the same (up to the $G^{\rm ess}$-equivalence)
projection $\pi_{\mathfrak g^{\rm ess}}\mathfrak s$.
As a result, each of these subsets contains no subalgebra being equivalent to a subalgebra from another subset.

The classification of one- and two-dimensional subalgebras of~$\mathfrak g$ up to the $G$-equivalence
is given in the following two assertions.

\begin{lemma}\label{lem:RemarkableFP1DSubalgsOfEntireG}
A complete list of $G$-inequivalent one-dimensional subalgebras of $\mathfrak g$ consists of
the one-dimensional subalgebras of~$\mathfrak g^{\rm ess}$ listed in Lemma~\ref{lem:RemarkableFP1DSubalgs}
and the subalgebras of the form $\langle\mathcal Z(f)\rangle$,
where the function~$f$ belongs to a fixed complete set of $G^{\rm ess}$-inequivalent nonzero solutions
of the equation~\eqref{eq:RemarkableFP}.
\end{lemma}

\begin{proof}
Consider an arbitrary one-dimensional subalgebra~$\mathfrak s$ of~$\mathfrak g$.
Let $Q$ be its basis element.
We can represent~$Q$ as $Q=\hat Q+\mathcal Z(f)$
for some $\hat Q\in\mathfrak g^{\rm ess}$ and some solution~$f$ of the equation~\eqref{eq:RemarkableFP}.

If $\hat Q\ne0$, then Lemma~\ref{lem:RemarkableFP1DSubalgs} implies that, modulo the $G^{\rm ess}$-equivalence,
the vector field~$\hat Q$ can be assumed to be
the basis element of one of the subalgebras $\mathfrak s_{1.1}$, \dots, $\mathfrak s_{1.8}$ of~$\mathfrak g^{\rm ess}$.
Then we set $f$ identically equal to zero
by pushing $Q$ forward with~$\mathscr Z(h)$, where $h$ is a solution of the equation~\eqref{eq:RemarkableFP}
that in addition satisfies the constraint $\hat Q[h]+f=0$.%
\footnote{%
For $\hat Q=\mathcal I$, the constraint $\hat Q[h]+f=0$ just means $h=-f$.
For the other values of~$\hat Q$ and locally analytical solutions~$f$ of the equation~\eqref{eq:RemarkableFP},
the existence of locally analytical solutions of the system $h_t+xh_y=h_{xx}$, $\hat Q[h]+f=0$ with respect to~$h$
follows from Riquier's theorem and the fact that $\hat Q$ is a Lie symmetry vector field of~\eqref{eq:RemarkableFP}.
We suppose this existence in the smooth case as well.
In fact, the proof of this existence reduces to the existence of a solution of linear inhomogeneous partial differential equation
in two independent variables, whose homogeneous counterpart coincides with a reduction of the equation~\eqref{eq:RemarkableFP}
with respect to $\langle\hat Q\rangle$.
}

If $\hat Q=0$, then $Q=\mathcal Z(f)$.
It suffices to note that the $G$- and $G^{\rm ess}$-equivalences coincide on~$\mathfrak g^{\rm lin}$.
\end{proof}

\begin{lemma}\label{lem:RemarkableFP2DSubalgsOfEntireG}
A complete list of $G$-inequivalent two-dimensional subalgebras of $\mathfrak g$ consists of
\begin{enumerate}\itemsep=0.5ex
\item
the two-dimensional subalgebras of~$\mathfrak g^{\rm ess}$ listed in Lemma~\ref{lem:RemarkableFP2DSubalgs},
\item
the subalgebras of the form $\langle\hat Q,\mathcal Z(f)\rangle$, where
$\hat Q$ is the basis element of one of the one-dimensional subalgebras $\mathfrak s_{1.1}$, \dots, $\mathfrak s_{1.7}$
of~$\mathfrak g^{\rm ess}$ listed in Lemma~\ref{lem:RemarkableFP1DSubalgs},
and the function~$f$ belongs to a fixed complete set
of ${\rm St}_{G^{\rm ess}}(\langle\hat Q\rangle)$-inequivalent nonzero $\langle\hat Q+\lambda\mathcal I\rangle$-invariant solutions
of the equation~\eqref{eq:RemarkableFP}
with ${\rm St}_{G^{\rm ess}}(\langle\hat Q\rangle)$ denoting the stabilizer subgroup of~$G^{\rm ess}$
with respect to $\langle\hat Q\rangle$ under the action of~$G^{\rm ess}$ on~$\mathfrak g^{\rm ess}$
and with
$\lambda\in\{0,1\}$ if $\hat Q\in\{\mathcal P^1,\mathcal P^0\}$,
$\lambda\in\{-1,0,1\}$ if $\hat Q=\mathcal P^t$,
$\lambda\geqslant0$ if $\hat Q\in\{\mathcal D,\mathcal P^2+\varepsilon\mathcal P^0\}$
and $\lambda\in\mathbb R$ otherwise,
\item
the subalgebras of the form $\langle\mathcal I,\mathcal Z(f)\rangle$,
where the function~$f$ belongs to a fixed complete set of $G^{\rm ess}$-inequivalent nonzero solutions
of the equation~\eqref{eq:RemarkableFP}, and
\item
the subalgebras of the form $\langle\mathcal Z(f^1),\mathcal Z(f^2)\rangle$, where
the function pair~$(f^1,f^2)$ belongs to a fixed complete set
of $G^{\rm ess}$-inequivalent (up to linearly recombining components) pairs of linearly independent solutions
of the equation~\eqref{eq:RemarkableFP}.
\end{enumerate}
\end{lemma}

\begin{proof}
Consider an arbitrary two-dimensional subalgebra $\mathfrak s$ of~$\mathfrak g$,
and let $Q_1$ and~$Q_2$ be its basis elements,
$\mathfrak s=\langle Q_1,Q_2\rangle$ with $[Q_1,Q_2]\in\mathfrak s$.
Due to the decomposition $\mathfrak g=\mathfrak g^{\rm ess}\lsemioplus\mathfrak g^{\rm lin}$,
each $Q_i$, $i=1,2$, can be represented in the form $Q_i=\hat Q_i+\mathcal Z(f^i)$,
where $\hat Q_i\in\mathfrak g^{\rm ess}$ and the function $f^i$ is a solution of the equation~\eqref{eq:RemarkableFP}.
As the principal $G$-invariant value in the course of classifying two-dimensional subalgebras of~$\mathfrak g$,
we choose the dimension $d:=\dim\pi_{\mathfrak g^{\rm ess}}\mathfrak s$
of the projection of~$\mathfrak s$ onto~$\mathfrak g^{\rm ess}$.

\medskip\par\noindent
$\boldsymbol{d=2.}$
Up to the $G^{\rm ess}$-equivalence,
the pair $(\hat Q_1,\hat Q_2)$ can be assumed to be the chosen basis of one of the subalgebras
$\mathfrak s^{\mu}_{2.1}$, \dots, $\mathfrak s_{2.15}$ of $\mathfrak g^{\rm ess}$
listed in Lemma~\ref{lem:RemarkableFP2DSubalgs}.

For the subalgebras $\mathfrak s_{2.9}$, \dots, $\mathfrak s_{2.15}$, we have $Q_2=\mathcal I$.
Pushing $\mathfrak s$ forward with~$\mathscr Z(h)$, where $h=-f^2$,
we can set $f^2=0$.
Then the commutation relation $[Q_1,Q_2]=0$ implies $f^1=0$.

For the subalgebras $\mathfrak s_{2.1}^\mu$, \dots, $\mathfrak s_{2.8}^\varepsilon$,
let $\kappa_1$ and~$\kappa_2$ be the constants such that $[\hat Q_1,\hat Q_2]=\kappa_1\hat Q_1+\kappa_2\hat Q_2$.
Then the functions~$f^1$ and~$f^2$ satisfy the constraint $\hat Q_1[f^2]-\hat Q_2[f^1]=\kappa_1f^1+\kappa_2f^2$.
Hence we can set them to be identically zero using the pushforward~$\mathscr Z(h)_*$,
where~$h$ is a solution of the overdetermined system
\[h_t+xh_y=h_{xx},\quad \hat Q_1[h]+f^1=0,\quad \hat Q_2[h]+f^2=0\]
with respect to~$h$.%
\footnote{%
Similarly to the proof of Lemma~\ref{lem:RemarkableFP1DSubalgsOfEntireG},
for locally analytical solutions~$f^1$ and~$f^2$ of the equation~\eqref{eq:RemarkableFP}
that satisfy the above constraint,
the existence of locally analytical solutions of the last system
follows from Riquier's theorem and the fact that $\hat Q_1$ and $\hat Q_2$ are Lie symmetry vector fields of~\eqref{eq:RemarkableFP},
and we assume this existence in the smooth case as well.
}

In total, this results in the first family of subalgebras from the lemma's statement.

\medskip\par\noindent
$\boldsymbol{d=1.}$
Without loss of generality, we can assume $\hat Q_2=0$ up to linearly recombining the basis elements~$Q_1$ and~$Q_2$
and thus $f^2\ne0$,
whereas modulo the $G^{\rm ess}$-equivalence, the vector field~$\hat Q_1$
can be assumed to be the basis element of one of the subalgebras $\mathfrak s_{1.1}$, \dots, $\mathfrak s_{1.8}$ of~$\mathfrak g^{\rm ess}$ that are listed in Lemma~\ref{lem:RemarkableFP1DSubalgs}.
Then Lemma~\ref{lem:RemarkableFP1DSubalgs} also implies that $f^1=0$ up to the $G^{\rm ess}$-equivalence.

For the subalgebras $\mathfrak s_{1.1}$, \dots, $\mathfrak s_{1.7}$, we have $\hat Q_1\notin\langle\mathcal I\rangle$.
Following the proof of Lemma~\ref{lem:RemarkableFP1DSubalgsOfEntireG}, we can set the function~$f^1$ to be identically zero.
Then the closedness of the subalgebra $\mathfrak s$ with respect to the Lie bracket implies the constraint $\hat Q_1[f^2]=\lambda f^2$.
In other words, the function~$f^2$ is a $\langle\hat Q_1+\lambda\mathcal I\rangle$-invariant solution of the equation~\eqref{eq:RemarkableFP}.
Transformations from the group ${\rm St}_{G^{\rm ess}}(\langle\hat Q_1\rangle)$ allow us
to set, depending on the value of~$\hat Q_1$, the restrictions on~$\lambda$
that are presented in item~2 of the lemma's statement.
The transformations from ${\rm St}_{G^{\rm ess}}(\langle\hat Q_1\rangle)$ that preserve~$\lambda$
also induce an equivalence relation on the corresponding set run by~$f^2$.

Since $\hat Q_1=\mathcal I$ for the subalgebra~$\mathfrak s_{1.8}$,
in the way described in the proof of Lemma~\ref{lem:RemarkableFP1DSubalgsOfEntireG},
we can set the function~$f^1$ to be identically zero.
Then the only commutation relation of the algebra $\mathfrak s$ is $[Q_1,Q_2]=Q_2$,
which implies no constraint for the function~$f^2$, and ${\rm St}_{G^{\rm ess}}(\langle\mathcal I\rangle)=G^{\rm ess}$.
Hence in this case the parameter function~$f^2$ should run through
a fixed complete set of $G^{\rm ess}$-inequivalent nonzero solutions
of the equation~\eqref{eq:RemarkableFP}.

Thus, we obtained the second and third families of subalgebras listed in the lemma's statement.

\medskip\par\noindent
$\boldsymbol{d=0.}$
This case is obvious since then $\hat Q_1=\hat Q_2=0$.
In other words, the subalgebra $\langle Q_1,Q_2\rangle$ is of the form $\langle\mathcal Z(f^1),\mathcal Z(f^2)\rangle$,
where~$f^1$ and~$f^2$ are linearly independent solutions of the equation~\eqref{eq:RemarkableFP},
which should be chosen modulo the $G^{\rm ess}$-equivalence and linearly recombining.
\end{proof}

\section{Lie reductions of codimension one}\label{sec:RemarkableFPLieRedCoD1}

The classification of Lie reductions of the equation~\eqref{eq:RemarkableFP}
to partial differential equations with two independent variables is based on the classification
of one-dimensional subalgebras of the algebra~$\mathfrak g^{\rm ess}$,
which is given in Lemma~\ref{lem:RemarkableFP1DSubalgs}.
In Table~\ref{tab:LieReductionsOfCodim1}, for each of the subalgebras listed therein,
we present an associated ansatz for $u$ and the corresponding reduced equation.
Throughout this section, the subscripts~1 and~2 of functions depending on $(z_1,z_2)$
denote derivatives with respect to~$z_1$ and~$z_2$, respectively.
In this and the next sections,
we mark Lie reductions and all related objects with complex labels $d.i^*$,
where $d$, $i$ and~$*$ are the dimension of the corresponding subalgebra,
its number in the list of $d$-dimensional subalgebras of the algebra~$\mathfrak g^{\rm ess}$,
which is given in Lemma~\ref{lem:RemarkableFP1DSubalgs} for $d=1$
or in Lemma~\ref{lem:RemarkableFP2DSubalgs} for $d=2$,
and the list of subalgebra parameters for a subalgebra family, respectively.
We omit the superscript in the label when it is not essential.

\begin{table}[!ht]\footnotesize
\begin{center}
\caption{$G$-inequivalent Lie reductions of codimension one}
\label{tab:LieReductionsOfCodim1}${ }$\\[-1ex]
\renewcommand{\arraystretch}{1.5}
\begin{tabular}{|l|c|c|c|l|}
\hline
no.			     & $u$                           			 & $z_1$                  				 & $z_2$                    			    & \hfil Reduced equation\\
\hline
1.1 		     & ${\rm e}^{\frac3{10}t(t^4-5tx+10y)}w$        & $y-\frac14t^4$        				 & $x-t^3$                  				 & $z_2w_1=w_{22}-3z_1w$ \\
1.2$^\delta$     & ${\rm e}^{\delta t}w$         			    & $y$                    				 & $x$                      				 & $z_2w_1=w_{22}-\delta w$ \\
1.3$^\nu$	     & $|t|^{\frac12\nu-1}w$                	    & $|t|^{-\frac32}y$      			     & $|t|^{-\frac12}x$        			    & $(2z_2-3\varepsilon'z_1)w_1=2w_{22}+\varepsilon'z_2w_2-\varepsilon'(\nu\!-\!2)w$\\
1.4$^\mu$	     & $\dfrac{{\rm e}^{-\psi(t,x,y,\mu)}}{t^2+1}w$ & $\dfrac y{(t^2+1)^{\frac32}}$ 		 & $\dfrac{(t^2+1)x-3ty}{(t^2+1)^{\frac32}}$& $z_2w_1=3z_1w_2+w_{22}+(z_2^{\,2}+\mu)w$ \\
1.5$^\varepsilon$& $|t|^{-\frac12}{\rm e}^{-\frac{x^2}{4t}}w$   & $\tfrac13{t^3}+2\varepsilon t-t^{-1}$ & $2y-(t+\varepsilon t^{-1})x$  			 & $w_1=w_{22}$ \\
1.6 		     & $w$                                          & $\frac13t^3$             			 & $y-tx$                   				 & $w_1=w_{22}$ \\
1.7 		     & $w$                                          & $t$                    				 & $x$                      				 & $w_1=w_{22}$ \\
\hline
\end{tabular}
\end{center}
Here
$\varepsilon':=\sgn t$,
$\varepsilon\in\{-1,1\}$,
$\delta\in\{-1,0,1\}$,
$w=w(z_1,z_2)$ is the new unknown function of the new independent variables $(z_1,z_2)$, and
\[
\psi(t,x,y,\mu)=\dfrac{3t^3y^2+t(2x(t^2+1)-3ty)^2}{4(t^2+1)^3}+\mu\arctan t.
\]
\end{table}

Each of the reduced equations~$1.i$ presented in Table~\ref{tab:LieReductionsOfCodim1},
$i\in\{1,2^\delta,3^\nu,4^\mu,5^\varepsilon,6,7\}$,
is a linear homogeneous partial differential equation in two independent variables.
Hence its maximal Lie invariance algebra~$\mathfrak g_{1.i}$ contains the infinity-dimensional abelian ideal $\{h(z_1,z_2)\p_w\}$,
which is associated with the linear superposition of solutions of reduced equation~$1.i$
and thus assumed to be a trivial part of~$\mathfrak g_{1.i}$.
Here the parameter function $h=h(z_1,z_2)$ runs through the solution set of reduced equation~$1.i$.
Moreover, the algebra~$\mathfrak g_{1.i}$ is the semidirect sum of the above ideal and a finite-dimensional subalgebra~$\mathfrak a_i$
called the essential Lie invariance algebra of reduced equation~$1.i$, cf.\ Section~\ref{sec:RemarkableFPMIA}.
The algebras~$\mathfrak a_i$ are the following:
\begin{gather*}
\mathfrak a_1=\mathfrak a_3=\mathfrak a_4=\langle w\p_w\rangle,\\[.8ex]
\mathfrak a_2^0=\langle\p_{z_1},\,3z_1\p_{z_1}\!+z_2\p_{z_2},\,9z_1^2\p_{z_1}\!+6z_1z_2\p_{z_2}\!-(6z_1+z_2^3)w\p_w,\,w\p_w\rangle,\\[.8ex]
\mathfrak a_2^\delta=\langle\p_{z_1},\,w\p_w\rangle\quad\mbox{if}\quad\delta\ne0,\\[.8ex]
\mathfrak a_5=\mathfrak a_6=\mathfrak a_7=\langle\p_{z_1},\,\p_{z_2},\,2z_1\p_{z_1}\!+z_2\p_{z_2},\,2z_1\p_{z_2}\!-z_2w\p_w,\\
\phantom{\mathfrak a_5=\mathfrak a_6=\mathfrak a_7=\langle}
4z_1^2\p_{z_1}\!+4z_1z_2\p_{z_2}\!-(2z_1+z_2^2)w\p_w,\,w\p_w\rangle.
\end{gather*}

The equation~\eqref{eq:RemarkableFP} admits hidden symmetries
with respect to a Lie reduction if the corresponding reduced equation possesses Lie symmetries
that are not induced by Lie symmetries of the original equation~\eqref{eq:RemarkableFP}.
To check which Lie symmetries of reduced equation~$1.i$ are induced by Lie symmetries of~\eqref{eq:RemarkableFP},
we compute the normalizer of the subalgebra~$\mathfrak s_{1.i}$ in the algebra $\mathfrak g^{\rm ess}$:
\begin{gather*}
{\rm N}_{\mathfrak g^{\rm ess}}(\mathfrak s_{1.1})=\langle\mathcal P^t+\mathcal P^3, \mathcal I\rangle,\ \;
{\rm N}_{\mathfrak g^{\rm ess}}(\mathfrak s_{1.2}^0)=\langle\mathcal P^t,\mathcal D,\mathcal P^0,\mathcal I\rangle,\ \;
{\rm N}_{\mathfrak g^{\rm ess}}(\mathfrak s_{1.2}^\delta)=\langle\mathcal P^t,\mathcal P^0,\mathcal I\rangle\,\ \mbox{if}\,\ \delta\ne0,
\\[.8ex]
{\rm N}_{\mathfrak g^{\rm ess}}(\mathfrak s_{1.3}^\nu)=\langle\mathcal D,\mathcal I\rangle,\ \;
{\rm N}_{\mathfrak g^{\rm ess}}(\mathfrak s_{1.4}^\mu)=\langle\mathcal P^t+\mathcal K,\mathcal I\rangle,\ \;
{\rm N}_{\mathfrak g^{\rm ess}}(\mathfrak s_{1.5}^\varepsilon)=\langle\mathcal P^2,\mathcal P^0,\mathcal P^3-3\varepsilon\mathcal P^1,\mathcal I\rangle,
\\[.8ex]
{\rm N}_{\mathfrak g^{\rm ess}}(\mathfrak s_{1.6})=\langle\mathcal D,\mathcal P^3,\mathcal P^1,\mathcal P^0,\mathcal I\rangle,\ \;
{\rm N}_{\mathfrak g^{\rm ess}}(\mathfrak s_{1.7})=\langle\mathcal P^t,\mathcal D,\mathcal P^2,\mathcal P^1,\mathcal P^0,\mathcal I\rangle.
\end{gather*}
The algebra of induced Lie symmetries of reduced equation $1.i$ is isomorphic to
the quotient algebra ${\rm N}_{\mathfrak g^{\rm ess}}(\mathfrak s_{1.i})/\mathfrak s_{1.i}$.
Thus, all Lie symmetries of the reduced equation~$1.i$ are induced by Lie symmetries of the original equation~\eqref{eq:RemarkableFP}
if and only if $\dim\mathfrak a_i=\dim{\rm N}_{\mathfrak g^{\rm ess}}(\mathfrak s_{1.i})-1$.
Comparing the dimensions of ${\rm N}_{\mathfrak g^{\rm ess}}(\mathfrak s_{1.i})$ and $\mathfrak a_i$, we conclude
that all Lie symmetries of reduced equations~$1.1$, $1.3$ and~$1.4$ are induced by Lie symmetries of~\eqref{eq:RemarkableFP}.
The subalgebras of induced symmetries in the algebras $\mathfrak a_2^\delta$ and $\mathfrak a_5$, \dots, $\mathfrak a_7$ are
\begin{gather*}
\tilde{\mathfrak a}_2^0=\langle\p_{z_1},\,3z_1\p_{z_1}\!+z_2\p_{z_2},\,w\p_w\rangle,\quad 
\tilde{\mathfrak a}_2^\delta=\langle\p_{z_1},\,w\p_w\rangle\ \ \mbox{if}\ \ \delta\ne0,\ \\[1ex]
\tilde{\mathfrak a}_5=\langle\p_{z_2},\,2z_1\p_{z_2}\!-z_2w\p_w,\,w\p_w\rangle,\\[1ex] 
\tilde{\mathfrak a}_6=\langle\p_{z_2},\,2z_1\p_{z_1}\!+z_2\p_{z_2},\,2z_1\p_{z_2}\!-z_2w\p_w,\,w\p_w\rangle,\\[1ex]  
\tilde{\mathfrak a}_7=\langle\p_{z_1},\,\p_{z_2},\,2z_1\p_{z_1}\!+z_2\p_{z_2},\,2z_1\p_{z_2}\!-z_2w\p_w,\,w\p_w\rangle. 
\end{gather*}
For each $i\in\{2^0,5,6,7\}$, the elements of the complement $\mathfrak a_i\setminus\tilde{\mathfrak a}_i$ are nontrivial hidden symmetries
of the equation~\eqref{eq:RemarkableFP} that are associated with reduction $1.i$.
See Appendix~\ref{sec:HiddenSyms} on an optimized procedure for constructing exact solutions using hidden symmetries.

We discuss reduced equations~1.1--1.7
and present solutions of~\eqref{eq:RemarkableFP} that can be found using the known solutions of the reduced equations.

Using point transformations, we can further map reduced equations~1.1--1.4
to (1+1)-dimen\-sional linear heat equations with potentials,
which are of the general form $w_1=w_{22}+Vw$, where $V$ is an arbitrary function of $(z_1,z_2)$.
These transformations and mapped equations are
\begin{itemize}
\item[1.1.\,]
$\tilde z_1=\tfrac94\tilde\varepsilon z_1,\ \ \tilde z_2=|z_2|^{\frac32},\ \ \tilde w=|z_2|^{\frac14}w$
\ with \ $\tilde\varepsilon:=\sgn z_2$:\quad
$\tilde w_1=\tilde w_{22}+\big(\tfrac5{36}\tilde z_2^{-2}-\tfrac{16}{27}\tilde\varepsilon\tilde z_1\tilde z_2^{-\frac23}\big)\tilde w$;
\item[1.2$\lefteqn{^\delta\!.}\phantom{.}\,$]
$\tilde z_1=\tfrac94\tilde\varepsilon z_1,\ \ \tilde z_2=|z_2|^{\frac32},\ \ \tilde w=|z_2|^{\frac14}w$
\ with \ $\tilde\varepsilon:=\sgn z_2$:\quad
$\tilde w_1=\tilde w_{22}+\big(\tfrac5{36}\tilde z_2^{-2}-\tfrac49\delta \tilde z_2^{-\frac23}\big)\tilde w$;
\item[1.3\lefteqn{^\nu\!.}\phantom{.}\,]
$\tilde z_1=\tfrac94\tilde\varepsilon z_1,\ \
\tilde z_2=|z_2-\tfrac32\varepsilon'z_1|^{\frac32},\ \
\tilde w=|2z_2-3\varepsilon'z_1|^{\frac14}{\rm e}^{\frac18z_2(4\varepsilon'z_2-9z_1)}w$
 \ with \ $\tilde\varepsilon:=\sgn(z_2-\tfrac32\varepsilon'z_1)$:
\\[.5ex] 
$\tilde w_1=\tilde w_{22}+\big(
\tfrac5{36}\tilde z_2^{-2}
-(\tfrac1{81}\tilde z_1^2+\tfrac29\varepsilon'\nu)\tilde z_2^{-\frac23}
-\tfrac5{18}\tilde z_2^{\frac23}-\tfrac{10}{27}\varepsilon'\tilde z_1
\big)\tilde w$;
\item[1.4\lefteqn{^\mu\!.}\phantom{.}\,]
$\tilde z_1=\tfrac94\tilde\varepsilon z_1,\ \ \tilde z_2=|z_2|^{\frac32},\ \ \tilde w=|z_2|^{\frac14}{\rm e}^{\frac32z_1z_2}w$ \ with \ $\tilde\varepsilon:=\sgn z_2$:
\\[.5ex]
$\tilde w_1=\tilde w_{22}+\big(\tfrac5{36}\tilde z_2^{-2}
-(\tilde z_1^2-\tfrac49\mu)\tilde z_2^{-\frac23}+\tfrac{10}9\tilde z_2^{\frac23}\big)\tilde w$.
\end{itemize}

The essential Lie invariance algebras of reduced equations~1.1, 1.3 and~1.4 are trivial,
and thus the further consideration of these equations within the classical framework gives no results.

Mapped reduced equation~1.2$^0$ coincides with the linear heat equation with potential $V=\mu z_2^{-2}$, where $\mu=\tfrac5{36}$.
Hence we construct the following family of solutions of the equation~\eqref{eq:RemarkableFP}:
\begin{gather}\label{eq:RemarkableFPs120InvSols}
\solution u=|x|^{-\frac14}\theta^\mu(\tfrac94\tilde\varepsilon y,|x|^{\frac32})\ \ \mbox{with}\ \ \mu=\tfrac5{36},
\end{gather}
where $\theta^\mu=\theta^\mu(z_1,z_2)$ is an arbitrary solution of the equation $\theta^\mu_1=\theta^\mu_{22}+\mu z_2^{-2}\theta^\mu$,
and $\tilde\varepsilon:=\sgn x$.
A~complete collection of inequivalent Lie invariant solutions of equations of such form with general $\mu\ne0$
is presented in Appendix~\ref{sec:SymHeatEqSquarePot}.

The essential Lie invariance algebra of reduced equation~1.2$^\delta$ with $\delta\ne0$ is induced
by the subalgebra $\langle\mathcal P^0,\mathcal I\rangle$ of~$\mathfrak g^{\rm ess}$.
The subalgebras of~$\mathfrak a_2^\delta$ that are appropriate for Lie reduction of reduced equation~1.2$^\delta$
are exhausted by the subalgebras of the form $\langle\p_{z_1}\!+\nu w\p_w\rangle$,
which are respectively induced by the subalgebras $\langle\mathcal P^0+\nu \mathcal I\rangle$ of~$\mathfrak g^{\rm ess}$.
This is why the family of solutions of reduced equation~1.2$^\delta$
that are invariant with respect to $\langle\p_{z_1}\!+\nu w\p_w\rangle$ gives
a family of solutions of the equation~\eqref{eq:RemarkableFP} that is contained,
up to the $G^{\rm ess}$-equivalence, in the family of $\langle\mathcal P^0\rangle$-invariant solutions,
see reduction~1.7.

Each of reduced equations 1.5--1.7 itself coincides with the classical (1+1)-dimensional linear heat equation,
which leads to the following families of solutions of the equation~\eqref{eq:RemarkableFP}:
\begin{gather*}
\solution u=|t|^{-\frac12}{\rm e}^{-\frac{x^2}{4t}}\theta(z_1,z_2),\ \
\mbox{where}\ \
z_1=\tfrac13{t^3}+2\varepsilon t-t^{-1},\ \
z_2=2y-(t+\varepsilon t^{-1})x,
\\[.5ex]
\solution u=\theta(z_1,z_2),\ \
\mbox{where}\ \
z_1=\tfrac13t^3,\ \
z_2=y-tx,
\\[.5ex]
\solution u=\theta(z_1,z_2),\ \
\mbox{where}\ \
z_1=t,\ \
z_2=x.
\end{gather*}
Here $\theta=\theta(z_1,z_2)$ is an arbitrary solution of the (1+1)-dimensional linear heat equation, $\theta_1=\theta_{22}$.
An enhanced complete collection of inequivalent Lie invariant solutions of this equation
was presented in~\cite[Section A]{vane2021a}, following Examples 3.3 and 3.17 in~\cite{olve1993A}.
These solutions exhaust, up to combining the $G^{\rm ess}$-equivalence and linear superposition with each other,
the set of known exact solutions of this equation
that are expressed in a closed form in terms of elementary and special functions.

\section{Lie reductions of codimension two}\label{sec:RemarkableFPLieRedCoD2}

Based on the list of $G^{\rm ess}$-inequivalent two-dimensional subalgebras from Lemma~\ref{lem:RemarkableFP2DSubalgs},
we carry out the exhaustive classification of codimension-two Lie reductions of the equation~\eqref{eq:RemarkableFP}.
In Table~\ref{tab:LieReductionsOfCodim2}, for each of these subalgebras,
we present an associated ansatz for $u$ and the corresponding reduced ordinary differential equation.

\begin{table}[!ht]\footnotesize
\begin{center}
\caption{$G$-inequivalent Lie reductions of codimension two}
\label{tab:LieReductionsOfCodim2}
${ }$\\[-1ex]
\renewcommand{\arraystretch}{1.6}
\begin{tabular}{|l|c|c|l|l|}
\hline
\!no.\!         & $u$                                          		  	 & $\omega$            		& \hfil Reduced equation \\
\hline
2.1$^\mu$        &$|y|^{\frac13(\mu-2)}\varphi$                   & $y^{-1}x^3$            & $9\omega\varphi_{\omega\omega}+(\omega+6)\varphi_\omega=\frac13(\mu-2)\varphi$\\
2.2$^\delta$     &${\rm e}^{\delta t}\varphi$                     & $x$                    & $\varphi_{\omega\omega}=\delta\varphi$ \\
2.3              &${\rm e}^y\varphi$                              & $x$                    & $\varphi_{\omega\omega}=\omega\varphi$ \\
2.4$^\mu$        &$|t|^{\frac12\mu-1}\varphi$  					  & $|t|^{-\frac32}(y-tx)$ & $2\varphi_{\omega\omega}=\varepsilon'(\mu-2)\varphi-3\varepsilon'\omega\varphi_\omega$ \\
2.5$^\mu$        &$|t|^{\frac12\mu-1}\varphi$                  	  & $|t|^{-\frac12}x$ 	   & $2\varphi_{\omega\omega}=\varepsilon'(\mu-2)\varphi-\varepsilon'\omega\varphi_\omega$\\
2.6              &$\varphi$                                       & $t$  				   & $\varphi_{\omega}=0$ \\
2.7              &${\rm e}^{-\frac{x^2}{4t}}\varphi$              & $t$  				   & $2\omega\varphi_\omega+\varphi=0$  \\
2.8$^\varepsilon$&${\rm e}^{-\frac32\frac{(y-tx)^2}{2t^3-\varepsilon}}\varphi$& $t$  	   & $(2\omega^3-\varepsilon)\varphi_\omega+3\omega^2\varphi=0$\\
\hline
\end{tabular}\\[2ex]
Here $\varepsilon':=\sgn t$, $\varphi=\varphi(\omega)$ is the new unknown function of the invariant independent variable~$\omega$.
\end{center}
\end{table}

In fact, the codimension-two Lie reductions are superfluous. 
More specifically, since $\mathfrak s_{1.2}^0\subset\mathfrak s_{2.1}^\mu$ for any~$\mu$, 
all $\mathfrak s_{2.1}^\mu$-invariant solutions of the equation~\eqref{eq:RemarkableFP} 
are $\mathfrak  s_{1.2}^0$-invariant and thus are of the form~\eqref{eq:RemarkableFPs120InvSols}.
The rest of the reductions only lead to solutions of~\eqref{eq:RemarkableFP},
each of which is $G^{\rm ess}$-equivalent to a solution of form 1.6 or~1.7 given in Section~\ref{sec:RemarkableFPLieRedCoD1}.
Indeed, each of the corresponding subalgebras $\mathfrak s_{2.2}$, \dots, $\mathfrak s_{2.8}$
contains, up to the $G^{\rm ess}$-equivalence for the subalgebra $\mathfrak s_{2.3}$,
at least one of the vector fields~$\mathcal P_0$ and~$\mathcal P_1$.
Nevertheless, we discuss the inequivalent codimension-two Lie reductions below
in order to explicitly present solutions of the equation~\eqref{eq:RemarkableFP}
that are invariant with respect to two-dimensional algebras of Lie-symmetry vector fields of this equation.

We find (nonzero) exact solutions of the constructed reduced equations
and present the associated exact solutions of the equation~\eqref{eq:RemarkableFP}.
Hereafter, $C_0$, $C_1$ and $C_2$ are arbitrary constants.

\medskip\par\noindent
{\bf 2.1}. 
The substitution $\varphi=|\omega|^{\frac13}{\rm e}^{-\frac{\omega}9}\tilde\varphi(\omega)$
maps reduced equation~2.1 to a Kummer's equation
\begin{gather*}
9\omega\tilde\varphi_{\omega\omega}+(12-\omega)\tilde{\varphi}_{\omega}=\tfrac13(\mu+1)\tilde\varphi,
\end{gather*}
whose general solution is
\[
\tilde\varphi=C_1M\big(\tfrac13(\mu+1),\tfrac43,\tfrac19\omega\big)+C_2 U\big(\tfrac13(\mu+1),\tfrac43,\tfrac19\omega\big),
\]
where $M(a,b,z)$ and $U(a,b,z)$ denote the standard solutions of general Kummer's equation $z\varphi_{zz}+(b-z)\varphi_z-a\varphi=0$.
After taking $\kappa:=\frac13(\mu+1)$ instead of~$\mu$ as an arbitrary constant parameter,
the corresponding family of particular solutions of the equation~\eqref{eq:RemarkableFP} can be represented in the form
\begin{gather*}
\solution u=x|y|^{\kappa-\frac43}{\rm e}^{-\tilde\omega}
\Big(C_1M\big(\kappa,\tfrac43,\tilde\omega\big)+C_2U\big(\kappa,\tfrac43,\tilde\omega\big)\Big)\quad\mbox{with}\quad\tilde\omega:=\tfrac19y^{-1}x^3.
\end{gather*}
It also admits the equivalent representation
\begin{gather*}
\solution u=x^{-1}|y|^{\kappa-\frac23}{\rm e}^{-\frac{\tilde\omega}2}
\Big(C_1M_{\frac23-\kappa,\frac16}(\tilde\omega)+C_2W_{\frac23-\kappa,\frac16}(\tilde\omega)\Big)\quad\mbox{with}\quad\tilde\omega:=\tfrac19y^{-1}x^3
\end{gather*}
in terms of the Whittaker functions~$M_{a,b}(z)$ and~$W_{a,b}(z)$,
which constitute the fundamental solution set of the Whittaker equation~\eqref{eq:Whittaker},
due to a relation of these functions to Kummer's functions,
\begin{gather*}
M_{a,b}(z):={\rm e}^{-\frac z2}z^{b+\frac12}M\big(b-a+\tfrac12,1+2b,z\big),\\[.5ex]
W_{a,b}(z):={\rm e}^{-\frac z2}z^{b+\frac12}U\big(b-a+\tfrac12,1+2b,z\big).
\end{gather*}

\noindent
{\bf 2.2$^\delta$}. Reduced equation 2.2$^\delta$ trivially integrates to
\begin{gather*}
\varphi=
\begin{cases}
C_1{\rm e}^\omega+C_2{\rm e}^{-\omega}\quad&\mbox{if}\ \ \delta=1,\\
C_1\omega+C_2\quad&\mbox{if}\ \ \delta=0,\\
C_1\sin\omega+C_2\cos\omega\quad&\mbox{if}\ \ \delta=-1.
\end{cases}
\end{gather*}

\noindent
{\bf 2.3}.
Reduced equation 2.3 is the Airy equation, whose general solution is
$\varphi=C_1\mathop{\rm Ai}(\omega)+C_2\mathop{\rm Bi}(\omega)$.

\medskip

\noindent
{\bf 2.4, 2.5}.
Substituting
$\varphi=\omega{\rm e}^{-\frac34\varepsilon'\omega^2}\tilde\varphi(\omega)$ and
$\varphi=\omega{\rm e}^{-\frac14\varepsilon'\omega^2}\tilde\varphi(\omega)$
into reduced equations~2.4 and~2.5, respectively,
we obtain the equations
\begin{gather*}
2\omega\tilde\varphi_{\omega\omega}+(4-3\varepsilon'\omega^2)\tilde\varphi_\omega=\varepsilon'(\mu+4)\omega\tilde\varphi,\\
2\omega\tilde\varphi_{\omega\omega}+(4-\varepsilon'\omega^2)\tilde\varphi_{\omega}=\varepsilon'\mu\omega\tilde\varphi,
\end{gather*}
whose general solutions are
\begin{gather*}
\tilde\varphi=C_1M\big(\tfrac\mu6+\tfrac23,\tfrac32,\tfrac34\varepsilon'\omega^2\big)
+C_2 U\big(\tfrac\mu6+\tfrac23,\tfrac32,\tfrac34\varepsilon'\omega^2\big),
\\[1ex]
\tilde\varphi=C_1M\big(\tfrac\mu2,\tfrac32,\tfrac14\varepsilon'\omega^2\big)
+C_2 U\big(\tfrac\mu2,\tfrac32,\tfrac14\varepsilon'\omega^2\big).
\end{gather*}

\noindent
{\bf 2.6--2.8}.
Here the reduced equations are ordinary first-order differential equations, which are trivially integrated to
\begin{gather*}
{2.6.}\ \varphi=C_0,\quad
{2.7.}\ \varphi=\frac{C_0}{\sqrt{|\omega|}}\,,\quad
{2.8.}\ \varphi=\frac{C_0}{\sqrt{|2\omega^3-\varepsilon|}}\,.
\end{gather*}

Substituting the above general solutions of reduced equations~2.2--2.8 into the corresponding ansatzes,
we construct the following families of particular solutions of the equation~\eqref{eq:RemarkableFP},
which are contained (as mentioned earlier, up to the $G^{\rm ess}$-equivalence for the fourth family)
in the solution families obtained in Section~\ref{sec:RemarkableFPLieRedCoD1}
using Lie reductions~1.6 and~1.7:
\begin{gather*}
\solution u=C_1{\rm e}^{t+x}+C_2{\rm e}^{t-x},\qquad\hspace{3ex}
\solution u=C_1x+C_2,\qquad\hspace{3ex}
\solution u=C_1{\rm e}^{-t}\sin x+C_2{\rm e}^{-t}\cos x,
\\
\solution u=C_1{\rm e}^y\mathop{\rm Ai}(x)+C_2{\rm e}^y\mathop{\rm Bi}(x),
\\
\solution u=|t|^{\frac{\mu-5}2}(y-tx){\rm e}^{-\tilde\omega}
\Big(C_1M\big(\tfrac\mu6+\tfrac23,\tfrac32,\tilde\omega\big)+
C_2U\big(\tfrac\mu6+\tfrac23,\tfrac32,\tilde\omega\big)\Big)
\ \ \mbox{with}\ \ \tilde\omega:=\tfrac34t^{-3}(y-tx)^2,
\\
\solution u=|t|^{\frac{\mu-3}2}x{\rm e}^{-\tilde\omega}
\Big(C_1M\big(\tfrac\mu2,\tfrac32,\tilde\omega\big)+
C_2U\big(\tfrac\mu2,\tfrac32,\tilde\omega\big)\Big)
\ \ \mbox{with}\ \ \tilde\omega:=\tfrac14t^{-1}x^2,
\\
\solution u=C_0,\qquad\hspace{3ex}
\solution u=\frac{C_0}{\sqrt{|t|}}{\rm e}^{-\frac{x^2}{4t}},\qquad\hspace{3ex}
\solution u=\frac{C_0}{\sqrt{|2t^3-\varepsilon|}}{\rm e}^{-\frac32\frac{(y-tx)^2}{2t^3-\varepsilon}}.
\end{gather*}

Considering nonzero solutions from this section up to the $G^{\rm ess}$-equivalence,
we can set the constant~$C_0$ and one of the constants $C_1$ and~$C_2$ to one.
Taking into account the linear superposition principle for solutions of the equation~\eqref{eq:RemarkableFP},
it suffices to present only linearly independent solutions.

\section{Lie reductions to algebraic equations}\label{sec:RemarkableFPLieRedCoD3}

As for an equation in three independent variables,
the maximal codimension of Lie reductions for the equation~\eqref{eq:RemarkableFP}
is equal to three, and such reductions lead to algebraic equations.
Let us show that these reductions do not give new explicit solutions of the equation~\eqref{eq:RemarkableFP}.

Any three-dimensional subalgebra of the algebra~$\mathfrak g$
that is appropriate for Lie reduction is $G$-equivalent to a subalgebra~$\mathfrak s$ of~$\mathfrak g^{\rm ess}$.
If the subalgebra~$\mathfrak s$ nontrivially intersects the radical~$\mathfrak r$ of~$\mathfrak g^{\rm ess}$,
then any $\mathfrak s$-invariant solution of~\eqref{eq:RemarkableFP} is $G^{\rm ess}$-equivalent
to a solution from one of families~1.5--1.7 constructed in Section~\ref{sec:RemarkableFPLieRedCoD1}.
Otherwise, the natural projection of the subalgebra~$\mathfrak s$ onto
the Levi factor~$\mathfrak f=\langle\mathcal P^t,\mathcal D,\mathcal K\rangle$
under the Levi decomposition $\mathfrak g^{\rm ess}=\mathfrak f\lsemioplus\mathfrak r$
is three-dimensional.
Therefore, the subalgebra~$\mathfrak s$ is nonsolvable, and thus it is simple,
i.e., it is a Levi factor of the algebra~$\mathfrak g^{\rm ess}$.
According to the Levi--Malcev theorem, stating that any two Levi factors of a finite-dimensional Lie algebra
are conjugate by an inner automorphism generated by elements of the corresponding nilradical,
the subalgebra~$\mathfrak s$ is $G^{\rm ess}$-equivalent to~$\mathfrak f$.
An ansatz constructed using the subalgebra~$\mathfrak f$
\begin{gather*}
\solution u=C_0y^{-\frac23}{\rm e}^{-\frac{x^3}{9y}},
\end{gather*}
where $C_0$ plays the role of the unknown constant,
identically satisfies the equation~\eqref{eq:RemarkableFP}.
Up to the $G^{\rm ess}$-equivalence,
we can set the constant~$C_0$ equal to one.
This solution belongs to the family of particular solutions
obtained with reduction~2.1$^\mu$ in Section~\ref{sec:RemarkableFPLieRedCoD2},
where $\mu=0$ or, equivalently, $\kappa=1/3$,
since
$U\big(\tfrac13,\tfrac43,z\big)=z^{-\frac13}$.
Hence it also belongs to the solution family~\eqref{eq:RemarkableFPs120InvSols}.

\section{Generalized reductions and generating solutions\\with symmetry operators}\label{sec:GenReductions}

We can construct more general families of exact solutions of the equation~\eqref{eq:RemarkableFP} using generalized symmetries.
Since the equation~\eqref{eq:RemarkableFP} is linear and homogeneous,
each first-order linear differential operator in total derivatives that is associated with an essential Lie symmetry of this equation
is its recursion operator.
Thus, it is obvious that the equation~\eqref{eq:RemarkableFP}
admits, in particular, the generalized vector fields $\big((\mathfrak P^2+\varepsilon\mathfrak P^0)^nu\big)\p_u$, $\big((\mathfrak P^1)^nu\big)\p_u$
and $\big((\mathfrak P^0)^nu\big)\p_u$ as its generalized symmetries.
Here and in what follows,
\begin{gather*}
\mathfrak P^3:=3t^2\mathrm D_x+t^3\mathrm D_y-3(y-tx),\ \
\mathfrak P^2:=2t\mathrm D_x+t^2\mathrm D_y+x,\ \
\mathfrak P^1:=\mathrm D_x+t\mathrm D_y,\ \
\mathfrak P^0:=\mathrm D_y
\end{gather*}
are the differential operators in total derivatives
that are associated with the Lie-symmetry vector fields~$-\mathcal P^3$, $-\mathcal P^2$, $-\mathcal P^1$ and~$-\mathcal P^0$
of the equation~\eqref{eq:RemarkableFP},
and $\mathrm D_x$ and $\mathrm D_y$ denote the operators of total derivatives with respect to~$x$ and~$y$, respectively.
The corresponding invariant solutions are constructed by means of solving the equation~\eqref{eq:RemarkableFP}
with the associated invariant surface condition, $(\mathfrak P^2+\varepsilon\mathfrak P^0)^nu=0$, $(\mathfrak P^1)^nu=0$ or $(\mathfrak P^0)^nu=0$, respectively.
For $n=2$, it is easy to derive the representation for such solutions in terms of solutions of the (1+1)-dimensional linear heat equation,
\begin{gather*}
\solution u=|t|^{-\frac12}{\rm e}^{-\frac{x^2}{4t}}\big(
(t-\varepsilon t^{-1})\theta^1_2-\tfrac12t^{-1}x\theta^1+\theta^0
\big),\\
\mbox{where}\ \
z_1=\tfrac13t^3+2\varepsilon t-t^{-1},\ \
z_2=2y-(t+\varepsilon t^{-1})x,
\\[1ex]
\solution u=x\theta^1-t^2\theta^1_2+\theta^0,\ \
\mbox{where}\ \
z_1=\tfrac13t^3,\ \
z_2=y-tx,
\\[1ex]
\solution u=4y\theta^1_{22}+x^2\theta^1_2-x\theta^1+\theta^0,\ \
\mbox{where}\ \
z_1=t,\ \
z_2=x.
\end{gather*}
Here $\theta^i=\theta^i(z_1,z_2)$ is an arbitrary solution of the (1+1)-dimensional linear heat equation, $\theta^i_1=\theta^i_{22}$, $i=0,1$.
As in Section~\ref{sec:RemarkableFPLieRedCoD1}, the subscripts~1 and~2 of functions depending on $(z_1,z_2)$
denote derivatives with respect to~$z_1$ and~$z_2$, respectively.

In fact, for each of the above generalized symmetries,
we can construct the corresponding family of invariant solutions
using algebraic tools and the structure of the algebra~$\mathfrak g$,
or, more specifically, of its radical~$\mathfrak r$.
Since
$[\mathcal P^2+\varepsilon\mathcal P^0,\mathcal P^1]=\mathcal I$,
$[\mathcal P^1,\mathcal P^2]=-\mathcal I$ and
$[\mathcal P^0,\mathcal P^3]= 3\mathcal I$, then
\begin{gather}\label{eq:CommConsequences}
\begin{split}&
(\mathfrak P^2+\varepsilon\mathfrak P^0)^n(\mathfrak P^1)^{n-1}\big(|t|^{-\frac12}{\rm e}^{-\frac{x^2}{4t}}\theta(z_1,z_2)\big)=0,\\[1ex]&
(\mathfrak P^1)^n(\mathfrak P^2)^{n-1}\big(\theta(z_1,z_2)\big)=0,\\[1ex]&
(\mathfrak P^0)^n(\mathfrak P^3)^{n-1}\big(\theta(z_1,z_2)\big)=0
\end{split}
\end{gather}
for arbitrary solutions $\theta=\theta(z_1,z_2)$ of the (1+1)-dimensional linear heat equation, $\theta_1=\theta_{22}$,
whereas the expression in the left-hand side of each equality in~\eqref{eq:CommConsequences}
in general becomes nonvanishing after replacing the $n$th power of the first operator by its $(n-1)$th power.
The expressions for the variables~$z_1$ and~$z_2$ should be taken from the corresponding item in the above list.
Recall that the operators~$\mathfrak P^0$, \dots,~$\mathfrak P^3$ preserve the solution set of the equation~\eqref{eq:RemarkableFP}.
Therefore, the families of its solutions that are invariant with respect to its generalized symmetries
$\big((\mathfrak P^2+\varepsilon\mathfrak P^0)^nu\big)\p_u$, $\big((\mathfrak P^1)^nu\big)\p_u$
or $\big((\mathfrak P^0)^nu\big)\p_u$ respectively take the form
\begin{gather*}
\solution u=\sum_{k=0}^{n-1}(\mathfrak P^1)^k\big(|t|^{-\frac12}{\rm e}^{-\frac{x^2}{4t}}\theta^k(z_1,z_2)\big)
\quad\mbox{with}\quad
z_1=\tfrac13t^3+2\varepsilon t-t^{-1},\ \
z_2=2y-(t+\varepsilon t^{-1})x,
\\
\solution u=\sum_{k=0}^{n-1}(\mathfrak P^2)^k\theta^k(z_1,z_2)
\quad\mbox{with}\quad
z_1=\tfrac13t^3,\ \
z_2=y-tx,
\\
\solution u=\sum_{k=0}^{n-1}(\mathfrak P^3)^k\theta^k(z_1,z_2)
\quad\mbox{with}\quad
z_1=t,\ \
z_2=x,
\end{gather*}
where $\theta^k=\theta^k(z_1,z_2)$ are arbitrary solutions of the (1+1)-dimensional linear heat equation,
$\theta^k_1=\theta^k_{22}$, $k=0,\dots,n-1$.
Note that for $n=2$, the first two families in this list have exactly the same representations
as the first two families in the previous list,
and
\[\mathfrak P^3\tilde\theta^1+\tilde\theta^0=4y\theta^1_{22}+x^2\theta^1_2-x\theta^1+\theta^0\]
if $\tilde\theta^1=-\tfrac43\theta^1_1$ and
$\tilde\theta^0=\theta^0+4t^2\theta^1_{12}+4tx\theta^1_{22}+(x^2-2t)\theta^1_2-2t\theta^1_2-x\theta^1$.

\noprint{
\begin{gather*}
\mathfrak P^3\tilde\theta^1+\tilde\theta^0=3t^2\tilde\theta^1_x-3(y-tx)\tilde\theta^1+\tilde\theta^0
=|\tilde\theta^1=-\tfrac43\theta^1_1|=-4t^2\theta^1_{12}+4(y-tx)\theta^1_1+\tilde\theta^0 \\
=4y\theta^1_{22}+x^2\theta^1_2-x\theta^1+\theta^0,
\\\mbox{hence}\quad
\tilde\theta^0=\theta^0+(4t^2\theta^1_{12}+4tx\theta^1_{22}+(x^2-2t)\theta^1_2)-(2t\theta^1_2+x\theta^1)
\end{gather*}
}

Up to the $G^{\rm ess}$-equivalence, the generalized vector fields
$\big((\mathfrak P^2+\varepsilon\mathfrak P^0)^nu\big)\p_u$,
$\big((\mathfrak P^1)^nu\big)\p_u$ and
$\big((\mathfrak P^0)^nu\big)\p_u$
exhaust those generalized symmetries of the equation~\eqref{eq:RemarkableFP}
that are obtained using powers of operators associated with vector fields from~$\mathfrak r$.
It is much more challenging to use vector fields from $\mathfrak g\setminus\mathfrak r$
for generalized reductions in a similar way.
The question of whether operators associated with vector fields from $\mathfrak g\setminus\mathfrak r$
can generate solutions that are $G^{\rm ess}$-inequivalent to the above solutions
needs a careful analysis.

\section{On Kramers equations}\label{sec:KramersEq}

An important subclass of the class~$\bar{\mathcal F}$ of (1+2)-dimensional ultraparabolic Fokker--Planck equations,
which are of the general form~\eqref{eq:Fokker_Planck_superclass},
is the class~$\mathcal K$ of Kramers equations (more commonly called the Klein--Kramers equations)
\[
u_t+xu_y=F(y)u_x+\gamma(xu+u_x)_x.
\]
These equations describe the evolution of the probability density function $u(t,x,y)$ of a Brownian particle in the phase space $(y,x)$
in one spatial dimension.
Here the variables~$t$, $y$ and~$x$ play the role of the time, the position and the momentum, respectively,
$F$~is an arbitrary smooth function of $y$ that is the derivative of the external potential with respect to~$y$,
and $\gamma$~is an arbitrary nonzero constant, which is related to the friction coefficient.
Up to a simple scale equivalence of Kramers equations,
without loss of generality, one can assume $\gamma=1$.
The subclass~$\mathcal K$ is singled out from the class~$\bar{\mathcal F}$
by the auxiliary differential constraints $C=0$, $A^0=A^1_x=A^2$, $A^0_t=A^0_x=A^0_y=0$ and~$A^1_t=0$.

A preliminary group classification of the class~$\mathcal K$ was presented in~\cite{spic1997a}.
In particular, it was proved that the essential Lie invariance algebra of an equation from the class $\mathcal K$
is eight-dimensional if and only if $F(y)=ky$ up to shifts of~$y$, where $k=-\tfrac34\gamma^2$ or $k=\tfrac3{16}\gamma^2$.

Any equation from the class~$\bar{\mathcal F}$ with eight-dimensional essential Lie symmetry algebra
is $\mathcal G_{\bar{\mathcal F}}$-equivalent to the remarkable Fokker--Planck equation~\eqref{eq:RemarkableFP}.
(We will present the proof of this fact as well as the complete group classification of the class~$\bar{\mathcal F}$ in a future paper.)
Therefore, there exist point transformations that map the equations
\begin{gather}
u_t+xu_y=\gamma u_{xx}+\gamma(x-\tfrac34\gamma y)u_x+\gamma u,\label{eq:Kramers34}\\
u_t+xu_y=\gamma u_{xx}+\gamma(x+\tfrac3{16}\gamma y)u_x+\gamma u \label{eq:Kramers316}
\end{gather}
to the equation~\eqref{eq:RemarkableFP}.

Bases of the essential Lie invariance algebras~$\mathfrak g^{\rm ess}_{\eqref{eq:Kramers34}}$ and~$\mathfrak g^{\rm ess}_{\eqref{eq:Kramers316}}$
of the equations~\eqref{eq:Kramers34} and~\eqref{eq:Kramers316} consist of the vector fields
\begin{gather*}
\hat{\mathcal P}^t={\rm e}^{-\gamma t}\left(\tfrac1\gamma\p_t-\tfrac32y\p_y+\tfrac12(3\gamma y-x)\p_x
-\big(\tfrac14(x+\tfrac32\gamma y)^2-\tfrac32\big)u\p_u\right),\quad
\hat{\mathcal D}=\tfrac2\gamma\p_t+u\p_u,\\
\hat{\mathcal K}={\rm e}^{\gamma t}\left(\tfrac1\gamma\p_t+\tfrac32y\p_y+\tfrac12(3\gamma y+x)\p_x
-\big(\tfrac34(x+\tfrac32\gamma y)(x-\tfrac12\gamma y)+\tfrac12\big)u\p_u\right),\\
\hat{\mathcal P}^3={\rm e}^{\frac32\gamma t}\left(\tfrac1\gamma\p_y+\tfrac32\p_x-\tfrac32(x-\tfrac12\gamma y)u\p_u\right),\quad
\hat{\mathcal P}^2={\rm e}^{\frac12\gamma t}\left(\tfrac1\gamma\p_y+\tfrac12\p_x\right),\\
\hat{\mathcal P}^1={\rm e}^{-\frac12\gamma t}\left(\tfrac1\gamma\p_y-\tfrac12\p_x+\tfrac12(x+\tfrac32\gamma y)u\p_u\right),\quad
\hat{\mathcal P}^0={\rm e}^{-\frac32\gamma t}\left(\tfrac1\gamma\p_y-\tfrac32\p_x\right),\quad
\hat{\mathcal I}=u\p_u\\
\hspace{-\mathindent}\mbox{and}\\
\check{\mathcal P}^t={\rm e}^{-\frac12\gamma t}\left(\tfrac1\gamma\p_t-\tfrac34y\p_y-\tfrac14(x-\tfrac32\gamma y)\p_x+u\p_u\right),\quad
\check{\mathcal D}=\tfrac4\gamma\p_t+2u\p_u,\\
\check{\mathcal K}={\rm e}^{\frac12\gamma t}\left(\tfrac4\gamma\p_t+3y\p_y+(x+\tfrac32\gamma y)\p_x-(x+\tfrac34\gamma y)^2u\p_u\right),\\
\check{\mathcal P}^3={\rm e}^{\frac34\gamma t}\left(\tfrac8\gamma\p_y+6\p_x-6(x+\tfrac14\gamma y)u\p_u\right),\quad
\check{\mathcal P}^2={\rm e}^{\frac14\gamma t}\left(\tfrac4\gamma\p_y+\p_x-(x+\tfrac34\gamma y)u\p_u\right),\\
\check{\mathcal P}^1={\rm e}^{-\frac14\gamma t}\left(\tfrac2\gamma\p_y-\tfrac12\p_x\right),\quad
\check{\mathcal P}^0={\rm e}^{-\frac34\gamma t}\left(\tfrac1\gamma\p_y-\tfrac34\p_x\right),\quad
\check{\mathcal I}=u\p_u,
\end{gather*}
respectively.
We choose the basis elements of the algebra~$\mathfrak g^{\rm ess}_{\eqref{eq:Kramers34}}$ (resp.\ $\mathfrak g^{\rm ess}_{\eqref{eq:Kramers316}}$)
in such a way that the commutation relations between them
coincide with those between the basis elements of the algebra~$\mathfrak g^{\rm ess}$ presented in Section~\ref{sec:RemarkableFPMIA}.
Thus, the fact that the algebras~$\mathfrak g^{\rm ess}$,
\smash{$\mathfrak g^{\rm ess}_{\eqref{eq:Kramers34}}$}
and~\smash{$\mathfrak g^{\rm ess}_{\eqref{eq:Kramers316}}$}
are isomorphic becomes obvious.

Point transformations~$\Phi_{\eqref{eq:Kramers34}}$ and~$\Phi_{\eqref{eq:Kramers316}}$
that respectively map the equations~\eqref{eq:Kramers34} and~\eqref{eq:Kramers316}
to the remarkable Fokker--Planck equation~\eqref{eq:RemarkableFP}
are constructed using the conditions that
\smash{$\Phi_{\eqref{eq:Kramers34}*}\hat{\mathcal V}=\mathcal V$} and $\Phi_{\eqref{eq:Kramers316}*}\check{\mathcal V}=\mathcal V$
for $\mathcal V\in\{\mathcal P^t, \mathcal D, \mathcal K, \mathcal P^3, \mathcal P^2, \mathcal P^1, \mathcal P^0, \mathcal I\}$.
Thus, the pushforwards~$\Phi_{\eqref{eq:Kramers34}*}$ and~$\Phi_{\eqref{eq:Kramers316}*}$ establish
the isomorphisms of the algebras~$\mathfrak g^{\rm ess}_{\eqref{eq:Kramers34}}$ and~$\mathfrak g^{\rm ess}_{\eqref{eq:Kramers316}}$
to the algebra~$\mathfrak g^{\rm ess}$,
$\Phi_{\eqref{eq:Kramers34}*}\mathfrak g^{\rm ess}_{\eqref{eq:Kramers34}}=\mathfrak g^{\rm ess}$
and
$\Phi_{\eqref{eq:Kramers316}*}\mathfrak g^{\rm ess}_{\eqref{eq:Kramers316}}=\mathfrak g^{\rm ess}$.
As a result, we obtain the point transformations

\begin{gather*}
\Phi_{\eqref{eq:Kramers34}}:\quad
\tilde t={\rm e}^{\gamma t},\quad
\tilde x={\rm e}^{\frac12\gamma t}(x+\tfrac32\gamma y),\quad
\tilde y=\gamma{\rm e}^{\frac32\gamma t}y,\quad
\tilde u={\rm e}^{-\frac14(x+\frac32\gamma y)^2-\frac32\gamma t}u,
\\[1ex]
\Phi_{\eqref{eq:Kramers316}}:\quad
\tilde t=2{\rm e}^{\frac12\gamma t},\quad
\tilde x={\rm e}^{\frac14\gamma t}(x+\tfrac34\gamma y),\quad
\tilde y=\gamma{\rm e}^{\frac34\gamma t}y,\quad
\tilde u={\rm e}^{-\gamma t}u,
\end{gather*}
where the variables with tildes correspond to the equation~\eqref{eq:RemarkableFP},
and the variables without tildes correspond to the source equation~\eqref{eq:Kramers34} or~\eqref{eq:Kramers316}.
By direct computation, we check that indeed the transformations~$\Phi_{\eqref{eq:Kramers34}}$ and~$\Phi_{\eqref{eq:Kramers316}}$
respectively map the equations~\eqref{eq:Kramers34} and~\eqref{eq:Kramers316} to the equation~\eqref{eq:RemarkableFP}.
In other words, an arbitrary solution $\tilde u=f(t,x,y)$ of~\eqref{eq:RemarkableFP}
is pulled back by~$\Phi_{\eqref{eq:Kramers34}}$ and~$\Phi_{\eqref{eq:Kramers316}}$
to the solutions $u=\Phi_{\eqref{eq:Kramers34}}^*f$ and $u=\Phi_{\eqref{eq:Kramers316}}^*f$
of the equations~\eqref{eq:Kramers34} and~\eqref{eq:Kramers316}, respectively,
\begin{gather*}
u=\Phi_{\eqref{eq:Kramers34}}^*f
={\rm e}^{\frac14(x+\frac32\gamma y)^2+\frac32\gamma t}f
\Big(
{\rm e}^{\gamma t},
{\rm e}^{\frac12\gamma t}\big(x+\tfrac32\gamma y\big),
\gamma{\rm e}^{\frac32\gamma t}y
\Big),
\\[1ex]
u=\Phi_{\eqref{eq:Kramers316}}^*f
={\rm e}^{\gamma t}f
\Big(
2{\rm e}^{\frac12\gamma t},
{\rm e}^{\frac14\gamma t}\big(x+\tfrac34\gamma y\big),
\gamma{\rm e}^{\frac34\gamma t}y
\Big).
\end{gather*}
Using the families of solutions of the equation~\eqref{eq:RemarkableFP}
that have been constructed in Sections~\ref{sec:RemarkableFPLieRedCoD1} and~\ref{sec:RemarkableFPLieRedCoD2},
we find the following families of solutions for the equation~\eqref{eq:Kramers34}:
\begin{gather*}
u={\rm e}^{\frac14(x+\frac32\gamma y)^2+\frac{11}8\gamma t}|x+\tfrac32\gamma y|^{-\frac14}\,
\theta^\mu\big(
\tfrac94\tilde\varepsilon\gamma{\rm e}^{\frac32\gamma t}y,\,
{\rm e}^{\frac34\gamma t}|x+\tfrac32\gamma y|^{\frac32}
\big) \\\qquad\text{with}\quad\mu=\tfrac5{36},\ \tilde\varepsilon:=\sgn(x+\tfrac32\gamma y),
\\[1ex]
u={\rm e}^{\gamma t}\,\theta\Big(
\tfrac13{\rm e}^{3\gamma t}+2\varepsilon{\rm e}^{\gamma t}-{\rm e}^{-\gamma t},\,
{\rm e}^{\frac32\gamma t}(\tfrac12\gamma y-x)-\varepsilon{\rm e}^{-\frac12\gamma t}(x+\tfrac32\gamma y)
\Big),
\\[1ex] 
u={\rm e}^{\frac14(x+\frac32\gamma y)^2+\frac32\gamma t}\,\theta
\big(
\tfrac13{\rm e}^{3\gamma t},\,
{\rm e}^{\frac32\gamma t}(x+\tfrac12\gamma y)
\big),
\\[1ex] 
u={\rm e}^{\frac14(x+\frac32\gamma y)^2+\frac32\gamma t}\,\theta
\big(
{\rm e}^{\gamma t},\,
{\rm e}^{\frac12\gamma t}(x+\tfrac32\gamma y)
\big),
\\[1ex]
u=\zeta|y|^{\kappa-\frac43}{\rm e}^{\frac12\tilde\zeta-\frac{x\zeta^2}{24\gamma y}+\frac32\kappa\gamma t}
\Big(C_1M\big(\kappa,\tfrac43,\tilde\zeta\big)+C_2U\big(\kappa,\tfrac43,\tilde\zeta\big)\Big)\\
\qquad\mbox{with}\quad\zeta:=2x+3\gamma y,\quad \tilde\zeta:=\tfrac1{72}(\gamma y)^{-1}\zeta^3.
\end{gather*}
In the same way, we can also obtain solutions of the equation~\eqref{eq:Kramers316},
\begin{gather*}
u={\rm e}^{\frac{15}{16}\gamma t}|x+\tfrac34\gamma y|^{-\frac14}\,\theta^\mu\Big(
\tfrac94\tilde\varepsilon\gamma{\rm e}^{\frac34\gamma t}y,\,
{\rm e}^{\frac38\gamma t}|x+\tfrac34\gamma y|^{\frac32}
\Big)\quad\text{with}\quad\mu=\tfrac5{36},\ \tilde\varepsilon:=\sgn(x+\tfrac34\gamma y),
\\[.5ex]
u={\rm e}^{-\frac18(x+\frac34\gamma y)^2+\frac34\gamma t}\,\theta\Big(
\tfrac83{\rm e}^{\frac32\gamma t}+4\varepsilon{\rm e}^{\frac12\gamma t}-\tfrac12{\rm e}^{-\frac12\gamma t},\,
{\rm e}^{\frac34\gamma t}(\tfrac12\gamma y-2x)-\tfrac\varepsilon2{\rm e}^{-\frac14\gamma t}(x+\tfrac34\gamma y)
\Big),
\\[.5ex] 
u={\rm e}^{\gamma t}\,\theta\big(
\tfrac23{\rm e}^{\frac32\gamma t},\,
{\rm e}^{\frac34\gamma t}(x+\tfrac14\gamma y)
\big),
\quad
u={\rm e}^{\gamma t}\,\theta\big(
2{\rm e}^{\frac12\gamma t},\,
{\rm e}^{\frac14\gamma t}(x+\tfrac34\gamma y)
\big),
\\[1.5ex]
u=\zeta^{1/3}|y|^{\kappa-1}{\rm e}^{-\zeta+\frac14(3\kappa+1)\gamma t}
\Big(C_1M\big(\kappa,\tfrac43,\zeta\big)+C_2U\big(\kappa,\tfrac43,\zeta\big)\Big)\\
\qquad\mbox{with}\quad\zeta:=(9\gamma y)^{-1}\big(x+\tfrac34\gamma y\big)^3.
\end{gather*}
Recall that $\theta^\mu=\theta^\mu(z_1,z_2)$ and $\theta=\theta(z_1,z_2)$ denote
arbitrary solutions
of the (1+1)-dimensional linear heat equation with potential $\mu z_2^{-2}$ and
of the (1+1)-dimensional linear heat equation, $\theta^\mu_1=\theta^\mu_{22}+\mu z_2^{-2}\theta^\mu$ and $\theta_1=\theta_{22}$, respectively.
$M(a,b,\omega)$ and $U(a,b,\omega)$ are the standard solutions of Kummer's equation $\omega\varphi_{\omega\omega}+(b-\omega)\varphi_\omega-a\varphi=0$.

Similarly to Section~\ref{sec:GenReductions},
we can also construct more general families of solutions of the equations~\eqref{eq:Kramers34} and~\eqref{eq:Kramers316},
carrying out their generalized reductions or acting by Lie-symmetry operators on known solutions.
At the same time, we can again use the transformations~$\Phi_{\eqref{eq:Kramers34}}$ and~$\Phi_{\eqref{eq:Kramers316}}$
just to pull back the solutions of the equation~\eqref{eq:RemarkableFP} that are obtained in Section~\ref{sec:GenReductions}.

Thus, the families of solutions of the equation~\eqref{eq:Kramers34}
that are invariant with respect to the generalized symmetries
$\big((\hat{\mathfrak P}^2+\varepsilon\hat{\mathfrak P}^0)^nu\big)\p_u$,
$\big((\hat{\mathfrak P}^1)^nu\big)\p_u$ or
$\big((\hat{\mathfrak P}^0)^nu\big)\p_u$
are respectively constituted by the solutions of the form
\begin{gather*}
u=\sum_{k=0}^{n-1}(\hat{\mathfrak P}^1)^k\big({\rm e}^{\gamma t}\,\theta^k(z_1,z_2)\big)
\\
\qquad\mbox{with}\quad
z_1=\tfrac13{\rm e}^{3\gamma t}+2\varepsilon{\rm e}^{\gamma t}-{\rm e}^{-\gamma t},\ \
z_2={\rm e}^{\frac32\gamma t}(\tfrac12\gamma y-x)-\varepsilon{\rm e}^{-\frac12\gamma t}(x+\tfrac32\gamma y),
\\[1ex]
u=\sum_{k=0}^{n-1}(\hat{\mathfrak P}^2)^k
\big(
{\rm e}^{\frac14(x+\frac32\gamma y)^2+\frac32\gamma t}\,
\theta^k(z_1,z_2)
\big)
\quad\mbox{with}\quad
z_1=\tfrac13{\rm e}^{3\gamma t},\ \
z_2={\rm e}^{\frac32\gamma t}(x+\tfrac12\gamma y),
\\[1ex]
u=\sum_{k=0}^{n-1}(\hat{\mathfrak P}^3)^k
\big(
{\rm e}^{\frac14(x+\frac32\gamma y)^2+\frac32\gamma t}\,
\theta^k(z_1,z_2)
\big)
\quad\mbox{with}\quad
z_1={\rm e}^{\gamma t},\ \
z_2={\rm e}^{\frac12\gamma t}(x+\tfrac32\gamma y).
\end{gather*}
Here and in what follows, $\theta^k=\theta^k(z_1,z_2)$ are arbitrary solutions of the (1+1)-dimensional linear heat equation,
$\theta^k_1=\theta^k_{22}$, $k=0,\dots,n-1$.
The differential operators in total derivatives
\begin{gather*}
\hat{\mathfrak P}^3:={\rm e}^{\frac32\gamma t}\left(\tfrac1\gamma{\rm D}_y+\tfrac32{\rm D}_x+\tfrac32(x-\tfrac12\gamma y)\right),\\[1ex]
\hat{\mathfrak P}^2:={\rm e}^{\frac12\gamma t}\left(\tfrac1\gamma{\rm D}_y+\tfrac12{\rm D}_x\right),\\[1ex]
\hat{\mathfrak P}^1:={\rm e}^{-\frac12\gamma t}\left(\tfrac1\gamma{\rm D}_y-\tfrac12{\rm D}_x-\tfrac12(x+\tfrac32\gamma y)\right),\\[1ex]
\hat{\mathfrak P}^0:={\rm e}^{-\frac32\gamma t}\left(\tfrac1\gamma{\rm D}_y-\tfrac32{\rm D}_x\right)
\end{gather*}
are Lie-symmetry operators of the equation~\eqref{eq:Kramers34}
and are associated with the Lie-symmetry vector fields
$-\hat{\mathcal P}^3$,
$-\hat{\mathcal P}^2$,
$-\hat{\mathcal P}^1$,
$-\hat{\mathcal P}^0$, respectively.

Analogously, the solutions of the equation~\eqref{eq:Kramers316}
that are invariant with respect to the generalized symmetries
$\big((\check{\mathfrak P}^2+\varepsilon\check{\mathfrak P}^0)^nu\big)\p_u$,
$\big((\check{\mathfrak P}^1)^nu\big)\p_u$ or
$\big((\check{\mathfrak P}^0)^nu\big)\p_u$
are of the form
\begin{gather*}
u=\sum_{k=0}^{n-1}(\check{\mathfrak P}^1)^k\big(
{\rm e}^{-\frac18(x+\frac34\gamma y)^2+\frac34\gamma t}\,\theta^k(z_1,z_2)\big)
\\
\qquad\mbox{with}\quad
z_1=\tfrac83{\rm e}^{\frac32\gamma t}+4\varepsilon{\rm e}^{\frac12\gamma t}-\tfrac12{\rm e}^{-\frac12\gamma t},\ \
z_2={\rm e}^{\frac34\gamma t}(\tfrac12\gamma y-2x)-\tfrac\varepsilon2{\rm e}^{-\frac14\gamma t}(x+\tfrac34\gamma y),
\\[1ex]
u=\sum_{k=0}^{n-1}(\check{\mathfrak P}^2)^k
\big({\rm e}^{\gamma t}\,\theta^k(z_1,z_2)\big)
\quad\mbox{with}\quad
z_1=\tfrac23{\rm e}^{\frac32\gamma t},\ \
z_2={\rm e}^{\frac34\gamma t}(x+\tfrac14\gamma y),
\\[1ex]
u=\sum_{k=0}^{n-1}(\check{\mathfrak P}^3)^k
\big(
{\rm e}^{\gamma t}\,\theta^k(z_1,z_2)
\big)
\quad\mbox{with}\quad
z_1=2{\rm e}^{\frac12\gamma t},\ \
z_2={\rm e}^{\frac14\gamma t}(x+\tfrac34\gamma y),
\end{gather*}
where the differential operators in total derivatives
\begin{gather*}
\check{\mathfrak P}^3:={\rm e}^{\frac34\gamma t}\left(\tfrac8\gamma{\rm D}_y+6{\rm D}_x+6(x+\tfrac14\gamma y)\right),\quad
\check{\mathfrak P}^2:={\rm e}^{\frac14\gamma t}\left(\tfrac4\gamma{\rm D}_y+{\rm D}_x+x+\tfrac34\gamma y\right),\\[1ex]
\check{\mathfrak P}^1:={\rm e}^{-\frac14\gamma t}\left(\tfrac2\gamma{\rm D}_y-\tfrac12{\rm D}_x\right),\quad
\check{\mathfrak P}^0:={\rm e}^{-\frac34\gamma t}\left(\tfrac1\gamma{\rm D}_y-\tfrac34{\rm D}_x\right)
\end{gather*}
are Lie-symmetry operators of the equation~\eqref{eq:Kramers316} and are associated with the Lie symmetry vector fields
$-\check{\mathcal P}^3$,
$-\check{\mathcal P}^2$,
$-\check{\mathcal P}^1$,
$-\check{\mathcal P}^0$.

Note that
$\Phi_{\eqref{eq:Kramers34}*}\hat{\mathfrak Q}=\mathfrak Q$ and $\Phi_{\eqref{eq:Kramers316}*}\check{\mathfrak Q}=\mathfrak Q$
for $\mathfrak Q\in\{\mathfrak P^3, \mathfrak P^2, \mathfrak P^1,\mathfrak P^0,1\}$.

\noprint{
\todo\todo

The fundamental solution $F$ of the equation~\eqref{eq:RemarkableFP} can be also pulled back
to the fundamental solutions~$\Phi_{\eqref{eq:Kramers34}}^*F$ and~$\Phi_{\eqref{eq:Kramers316}}^*F$ of the equations~\eqref{eq:Kramers34} and~\eqref{eq:Kramers316}, respectively,

We fix a point $(t',x',y')$
\begin{gather*}
\Phi_{\eqref{eq:Kramers34}}^*F(t,x,y,t',x',y')=
{\rm e}^{\frac14(x+\frac32\gamma y)^2-\frac1{16}(3\gamma y'+2x')^2+\frac32\gamma(t-t')}
\end{gather*}
\begin{gather*}
\Phi_{\eqref{eq:Kramers34}}^*F(t,x,y,t',x',y')=
{\rm e}^{\frac1{16}(3\gamma(y-y')+2(x-x'))(3\gamma(y+y')+2(x+x'))+\frac32\gamma(t-t')}
\end{gather*}

\begin{equation*}
F(t,y,x,t',y',x')=\frac{\sqrt3H(t-t')}{2\pi(t-t')^2}
\exp\left(-\frac{(x-x')^2}{4(t-t')}-\frac{3\big(y-y'-\frac12(x+x')(t-t')\big)^2}{(t-t')^3}\right),
\end{equation*}
\begin{gather*}
u=\Phi_{\eqref{eq:Kramers34}}^*f
={\rm e}^{\frac14(x+\frac32\gamma y)^2+\frac32\gamma t}f
\Big(
{\rm e}^{\gamma t},
{\rm e}^{\frac12\gamma t}\big(\tfrac32\gamma y+x\big),
\gamma{\rm e}^{\frac32\gamma t}y
\Big),
\\
u=\Phi_{\eqref{eq:Kramers316}}^*f
={\rm e}^{\gamma t}f
\Big(
{\rm e}^{\frac12\gamma t},
\tfrac1{2\sqrt2}{\rm e}^{\frac14\gamma t}\big(x+\tfrac34\gamma y\big),
\tfrac1{2\sqrt2}\gamma{\rm e}^{\frac34\gamma t}y
\Big).
\end{gather*}

\todo\todo
}


\section{Conclusion}\label{sec:Conclusion}

The remarkable ultraparabolic Fokker--Planck equation~\eqref{eq:RemarkableFP} had been intensively studied in the mathematical literature
from various points of view, including the classical group analysis of differential equations.
Nevertheless, there were no exhaustive and accurate results even on its complete point symmetry pseudogroup
and on classification of its Lie reductions, and the present paper has filled up this gap.
Moreover, some of the obtained results were unexpected and may affect the entire field of classical group analysis of differential equations.

The complete point symmetry pseudogroup~$G$ of the equation~\eqref{eq:RemarkableFP}
is given in Theorem~\ref{thm:RemarkableFPSymGroup}.
To simplify the computation of this pseudogroup, we have applied the two-step version of the direct method.
In the first step, we have considered the class~$\bar{\mathcal F}$
of (1+2)-dimensional ultraparabolic Fokker--Planck equations of the general form~\eqref{eq:Fokker_Planck_superclass},
which contains the equation~\eqref{eq:RemarkableFP},
have proven its normalization in the usual sense and have found its equivalence pseudogroup~$G^\sim_{\bar{\mathcal F}}$,
see Theorem~\ref{thm:EquivalenceGroupFPsuperClass}.
Exhaustively describing the equivalence groupoid of the class~$\bar{\mathcal F}$ in this way,
we have derived the principal constraints for point symmetries of the equation~\eqref{eq:RemarkableFP}.
In the second step, we have in fact looked for admissible transformations of the class~$\bar{\mathcal F}$
that preserve the equation~\eqref{eq:RemarkableFP}, i.e., that constitute its vertex group.
This has led to a highly coupled overdetermined system of nonlinear partial differential equations
for the transformation components.
We have successfully found its general solution
and constructed a nice representation~\eqref{eq:RemarkableFPSymGroup} for elements of~$G$.
A similar splitting in the course of computing point symmetries by the direct method was used earlier, e.g.,
in \cite{bihl2011c}, but there the equivalence groupoid of the corresponding class had been known.
In the present paper, we have first found the equivalence groupoid of a class of differential equations
in order to compute the point symmetry (pseudo)group of a single element of this class,
and this is indeed the optimal way of computing in spite of looking peculiar.

The description of~$G$ in Theorem~\ref{thm:RemarkableFPSymGroup} implies
that its elements with $f=0$ and with their natural domains
are not necessarily defined on the entire space~$\mathbb R^4_{t,x,y,u}$.
This requires considering a pseudogroup constituted by all possible restrictions of such transformations,
when the usual composition of point transformations is taken as the group operation in~$G$,
complicates the entire analysis of the structure of~$G$.
In Section~\ref{sec:RemarkableFPPointSymGroup}, we have avoided the indicated problem by defining
the multiplication of elements in~$G$ as their composition modified with the extension of its domain by continuity.
Under the suggested formalism,
the representation~\eqref{eq:RemarkableFPSymGroup} for the elements of~$G$
has allowed us to comprehensively analyze the structure of~$G$.
The pseudogroup~$G$ contains the abelian normal pseudosubgroup~$G^{\rm lin}$,
which is associated with the linear superposition of solutions.
Moreover, the pseudogroup~$G$ splits over~$G^{\rm lin}$, $G=G^{\rm ess}\ltimes G^{\rm lin}$.
Here $G^{\rm ess}$ is a subgroup of~$G$, which is a (finite-dimensional) Lie group
and admits the factorization $G^{\rm ess}=(F\ltimes R_{\rm c})\times R_{\rm d}$,
where $F\simeq{\rm SL}(2,\mathbb R)$, $R_{\rm c}\simeq{\rm H}(2,\mathbb R)$ and $R_{\rm d}\simeq\mathbb Z_2$.
Since ${\rm SL}(2,\mathbb R)$ and ${\rm H}(2,\mathbb R)$ are connected Lie groups
and $G^{\rm lin}$ is a connected Lie pseudogroup,
the above factorizations directly imply Corollary~\ref{cor:RemarkableFPDiscretePointSyms},
which states the surprising fact that the equation~\eqref{eq:RemarkableFP} admits
a single independent discrete point symmetry transformation.
The simplest choice for such a transformation is
the involution $\mathscr I'\colon(t,x,y,u)\mapsto(t,x,y,-u)$,
which generates the group~$R_{\rm d}$.

In addition, we have constructed the point transformation~\eqref{eq:MapFundSolutRemarkableFP},
which maps the function $u=1-H(t)$ to the fundamental solution of the equation~\eqref{eq:RemarkableFP}.
A similar construction arises in the Lie symmetry analysis of linear (1+1)-dimensional heat equation.

The accurate consideration of the representation of the subgroup~$F$ of~$G$
on the radical $\mathfrak r$ of the algebra $\mathfrak g^{\rm ess}$,
which coincides with the representation of ${\rm SL}(2,\mathbb R)$ on the space of real binary cubic forms,
has allowed us to successfully classify $G^{\rm ess}$-inequivalent one- and two-dimensional subalgebras of $\mathfrak g^{\rm ess}$,
which for a long time had been a stumbling block in the course of constructing inequivalent Lie invariant solutions
of the equation~\eqref{eq:RemarkableFP} before.

Overcoming this obstacle has allowed us to exhaustively classify Lie reductions of codimensions one and two of the equation~\eqref{eq:RemarkableFP} and its Lie invariant solutions.
In particular, it has been reduced to the (1+1)-dimensional linear heat equations
with the zero and an inverse quadratic potentials. 
The former equation had been comprehensively studied within the framework of Lie reduction.
All inequivalent Lie invariant solutions of linear (1+1)-dimensional heat equations with inverse square potentials 
over the real field have been presented in closed form in Appendix~\ref{sec:SymHeatEqSquarePot}.
It has also been shown that Lie reductions to ordinary differential and algebraic equations, 
which are of codimensions two and three, respectively,
give no new solutions in comparison with those constructed using Lie reductions of codimension one.
Nevertheless, for each of the reduced ordinary differential equations, we have found its general solution 
and thus presented the corresponding Lie invariant solutions in explicit form. 
Mapping solutions of reduced equations to those of the original equation~\eqref{eq:RemarkableFP}
has led to finding wide families of its explicit closed-form solutions.
Up to the $G^{\rm ess}$-equivalence, essential among these families are
three families parameterized by single arbitrary solutions of the classical (1+1)-dimensional linear heat equation and
one family parameterized by an arbitrary solution of the (1+1)-dimensional linear heat equation with a particular inverse square potential.
One can further extend the constructed families of solutions,
iteratively acting on them by Lie-symmetry operators.
In this way, we have obtained families of exact solutions of the equation~\eqref{eq:RemarkableFP}
that are nontrivially parameterized by an arbitrary finite number of arbitrary solutions of the (1+1)-dimensional linear heat equation.
Modulo the $G^{\rm ess}$-equivalence,
these families exhaust solutions of the equation~\eqref{eq:RemarkableFP}
that are invariant with respect to its generalized symmetries
associated with powers of operators from
$\langle\mathfrak P^0,\mathfrak P^1,\mathfrak P^2,\mathfrak P^3,1\rangle$.
We have showed that the equation~\eqref{eq:RemarkableFP} admits hidden symmetries
that are associated with reductions to the (1+1)-dimensional linear heat equations with the zero or inverse square potentials.
This is why, in Appendix~\ref{sec:HiddenSyms}, we have presented a general optimized algorithm for constructing hiddenly invariant solutions.

Studying Kramers equations, which constitute the subclass $\mathcal K$ of the class $\bar{\mathcal F}$,
we have found nontrivial point transformations that map all equations from the class $\mathcal K$
with eight-dimensional essential Lie invariance algebras,
i.e., the equations~\eqref{eq:Kramers34} and~\eqref{eq:Kramers316} up to shifts of~$y$,
to the equation~\eqref{eq:RemarkableFP}.
Moreover, the presented formulas for mapping solutions of the equation~\eqref{eq:RemarkableFP}
to solutions of the equations~\eqref{eq:Kramers34} and~\eqref{eq:Kramers316}
give the simplest way of obtaining explicit solutions of the latter equations.

The results obtained in the paper can be extended in many directions.
In particular, one can compute reduction modules~\cite{boyk2016a} of the equation~\eqref{eq:RemarkableFP}
or, more promisingly, extend the consideration from Section~\ref{sec:GenReductions}
on generalized symmetries of this equation
and on generating its new solutions by acting with Lie-symmetry differential operators.
Conservation laws and potential symmetries of the equation~\eqref{eq:RemarkableFP}
also require a comprehensive study.

Instead of considering various properties of the equation~\eqref{eq:RemarkableFP},
we can extend results of the present paper
by studying other linear (1+2)-dimensional ultraparabolic equations,
including the group classification of classes of such equations,
all of which are reduced by point transformations to the form~\eqref{eq:Fokker_Planck_superclass}.
We are working on completing the solution of the group classification problem for
the subclass of~$\bar{\mathcal F}$ that consists of
the Fokker--Planck equations of the form $u_t+xu_y=|x|^\beta u_{xx}$
parameterized by an arbitrary real constant~$\beta$.
This subclass contains the equation~\eqref{eq:RemarkableFP}
and is of interest from the point of view of applications and due to its mathematical properties.
In particular, it admits a non-obvious discrete point equivalence transformation
\[
\tilde t=y\sgn x,\quad
\tilde x=\frac1x,\quad
\tilde y=t\sgn x,\quad
\tilde u=\frac ux,\quad
\tilde\beta=5-\beta.
\]

Of course, the most prominent among the above group classification problems
is the group classification of the entire class of linear (1+2)-dimensional ultraparabolic equations,
which reduces to the group classification of the class~$\bar{\mathcal F}$
of the linear (1+2)-dimensional ultraparabolic Fokker--Planck equations,
which are of the form~\eqref{eq:Fokker_Planck_superclass}.
The first step of the latter classification is the description of the equivalence groupoid of the class~$\bar{\mathcal F}$
by proving its normalization in the usual sense and constructing its usual equivalence pseudogroup,
which is presented in Theorem~\ref{thm:EquivalenceGroupFPsuperClass}.
The proof of this theorem and the complete solution of the group classification problem
for the class~$\bar{\mathcal F}$ will be a subject of another paper.

\appendix

\section{Exact solutions of (1+1)-dimensional linear heat equations\\ with inverse square potential}\label{sec:SymHeatEqSquarePot}

Lie reduction of the equation~\eqref{eq:RemarkableFP} with respect to the subalgebra $\mathfrak s_{1.2}^0=\langle\mathcal P^t\rangle$
leads to reduced equation~1.2$^0$, which is the linear heat equation with potential $V=\mu z_2^{-2}$, where $\mu=\tfrac5{36}$.
The linear heat equations with general (nonzero) inverse square potentials $V=\mu z_2^{-2}$, where $\mu\ne0$,
constitute an important case of Lie-symmetry extension in the class of linear second-order partial differential equations
in two independent variables.
This result was obtained by Sophus Lie himself~\cite{lie1881a}.
Lie invariant solutions of such heat equations over the complex field were considered in~\cite{gung2018b}.
The essential point symmetry group of these equations
was constructed as a by-product in the course of proving Theorems~18 in~\cite{opan2022a},
and a background that can be used for exhaustively classifying Lie invariant solutions of these equations over the real field
was developed in the proof of Theorems~51 therein.
We revisit the above results, present them in an enhanced and closed form
and complete the study of Lie invariant solutions of linear heat equations with (nonzero) inverse square potentials.
Recall that Lie reductions of the (1+1)-dimensional linear heat equation, which corresponds to the value $\mu=0$,
were comprehensively studied in~\cite[Section A]{vane2021a}, following Examples 3.3 and 3.17 in~\cite{olve1993A}.
Hereafter, we use notations that differ from those in the other sections of this paper.

Consider a linear heat equation with inverse square potential,
\begin{gather}\label{eq:HeatEqSquarePot}
u_t=u_{xx}+\frac\mu{x^2}u,
\end{gather}
where $\mu\ne0$.
It belongs to the class $\mathcal E$ of linear (1+1)-dimensional second-order evolution equations of the general form
\begin{gather}
u_t=A(t,x)u_{xx}+B(t,x)u_x+C(t,x)u+D(t,x)\quad \mbox{with}\quad A\ne0.
\end{gather}
Here the tuple of arbitrary elements of $\mathcal E$ is $\theta:=(A,B,C,D)\in\mathcal S_{\mathcal E}$,
where $S_{\mathcal E}$ is the solution set of the auxiliary system
consisting of the single inequality $A\ne0$ with no constraints on $B$, $C$ and $D$.

Since the equation~\eqref{eq:HeatEqSquarePot} is a linear homogeneous equation,
its maximal Lie invariance algebra~$\mathfrak g$ contains the infinite-dimensional abelian ideal $\mathfrak g^{\rm lin}:=\{h(t,x)\p_u\}$,
where the parameter function~$h$ runs through its solution set, cf.\ Section~\ref{sec:RemarkableFPMIA}.
This ideal is associated with the linear superposition of solutions of~\eqref{eq:HeatEqSquarePot}.
The algebra~$\mathfrak g$ splits over the ideal~$\mathfrak g^{\rm lin}$,
$\mathfrak g=\mathfrak g^{\rm ess}\lsemioplus\mathfrak g^{\rm lin}$,
where the complement subalgebra $\mathfrak g^{\rm ess}$ is four-dimensional,
\begin{gather*}
\mathfrak g^{\rm ess}=\big\langle
\mathcal P^t=\p_t,\
\mathcal D=t\p_t+\tfrac12x\p_x-\tfrac14u\p_u,\
\mathcal K=t^2\p_t+tx\p_x-\tfrac14(x^2+2t)u\p_u,\
\mathcal I=u\p_u
\big\rangle.
\end{gather*}
The subalgebra $\mathfrak g^{\rm ess}$ is called the essential Lie invariance algebra of~\eqref{eq:HeatEqSquarePot}.
It is the algebra $\mathfrak a_2^{\delta=0}$ in the notation of Section~\ref{sec:RemarkableFPLieRedCoD1}.
Up to the skew-symmetry of the Lie bracket, the nonzero commutation relations of this algebra are exhausted by
$[\mathcal P^t,\mathcal D]=\mathcal P^t$,
$[\mathcal D,\mathcal K]=\mathcal K$,
$[\mathcal P^t,\mathcal K]=2\mathcal D$.
Therefore, the algebra~$\mathfrak g^{\rm ess}$ is isomorphic to ${\rm sl}(2,\mathbb R)\oplus A_1$.

To find the complete point symmetry pseudogroup~$G$ of the equation~\eqref{eq:HeatEqSquarePot},
we start with considering the equivalence groupoid of the class~$\mathcal E$,
which in its turn is the natural choice for a (normalized) superclass for the equation~\eqref{eq:HeatEqSquarePot}.
We use the papers~\cite{opan2022a,popo2008a} as reference points for known results on admissible transformations of the class~$\mathcal E$.

\begin{proposition}[\cite{popo2008a}]\label{prop:EquivalenceGroupE}
The class~$\mathcal E$ is normalized in the usual sense.
Its usual equivalence pseudogroup~$G^\sim_{\mathcal E}$ consists of the transformations of the~form%
\begin{subequations}\label{eq:PointTransInEA}
\begin{gather}\label{eq:GenFormOfPointTransInEA}
\tilde t=T(t),\quad
\tilde x=X(t,x),\quad
\tilde u=U^1(t,x)u+U^0(t,x),\\ \label{eq:GenFormOfPointTransInEB}
\tilde A=\frac{X_x^2}{T_t}A,\quad
\tilde B=\frac{X_x}{T_t}\left(B-2\frac{U^1_x}{U^1}A\right)-\frac{X_t-X_{xx}A}{T_t},\quad
\tilde C=-\frac{U^1}{T_t}\mathrm E\frac1{U^1},\\
\label{eq:GenFormOfPointTransInEC}
\tilde D=\frac{U^1}{T_t}\left(D+\mathrm E\frac{U^0}{U^1}\right),
\end{gather}
\end{subequations}
where~$T$, $X$, $U^0$ and~$U^1$ are arbitrary smooth functions of their arguments with~$T_tX_xU^1\ne0$, and
$\mathrm E:=\p_t-A\p_{xx}-B\p_x-C$.
\end{proposition}

The normalization of the class~$\mathcal E$ means that its equivalence groupoid coincides
with the action groupoid of the pseudogroup~$G^\sim_{\mathcal E}$.

\begin{theorem}\label{thm:HeatEqSquarePotSymGroup}
The complete point symmetry pseudogroup~$G$
of the (1+1)-dimensional linear heat equation with inverse square potential~\eqref{eq:HeatEqSquarePot}
consists of the point transformations of the form
\[
\tilde t=\frac{\alpha t+\beta}{\gamma t+\delta},\quad
\tilde x=\frac x{\gamma t+\delta},\quad
\tilde u=\sigma\sqrt{|\gamma t+\delta|}(u+h(t,x))\exp{\frac{\gamma x^2}{4(\gamma t+\delta)}},
\]
where $\alpha$, $\beta$, $\gamma$, $\delta$ and $\sigma$ are arbitrary constants with $\alpha\delta-\beta\gamma=1$ and $\sigma\ne0$,
and $h$ is an arbitrary solution of~\eqref{eq:HeatEqSquarePot}.
\end{theorem}

\begin{proof}
The linear heat equation with inverse square potential~\eqref{eq:HeatEqSquarePot}
corresponds to the value $(1,0,\mu x^{-2},0)=:\theta^\mu$
of the arbitrary element tuple $\theta=(A,B,C,D)$ of class $\mathcal E$.
Its vertex group $\mathcal G_{\theta^\mu}:=\mathcal G^\sim_{\mathcal E}(\theta^\mu,\theta^\mu)$
is the set of admissible transformations of the class~$\mathcal E$ with~$\theta^\mu$ as both their source and target,
$\mathcal G_{\theta^\mu}=\{(\theta^\mu,\Phi,\theta^\mu)\mid\Phi\in G\}$.
This argument allows us to use Proposition~\ref{prop:EquivalenceGroupE} in the course of computing the pseudogroup~$G$.
	
We should integrate the equations~\eqref{eq:PointTransInEA},
where both the source value~$\theta$ of the arbitrary-element tuple and its target value~$\tilde\theta$
coincide with $\theta^\mu$,
with respect to the parameter functions $T$, $X$, $U^1$ and $U^0$.
After a simplification, the equations~\eqref{eq:GenFormOfPointTransInEB} take the form
\begin{gather}\label{eq:PointTransInETransPart}
X_x^2=T_t,\quad
\frac{U^1_x}{U^1}=-\frac{X_t}{2X_x},\quad
\frac\mu{X^2}=-\frac{U^1}{T_t}\mathrm E\frac1{U^1},
\end{gather}
where ${\rm E}:=\p_t-\p_{xx}-\mu x^{-2}$.
The first equation in~\eqref{eq:PointTransInETransPart} implies that $T_t>0$,
and the first two equations in~\eqref{eq:PointTransInETransPart} can be easily integrated to
\begin{gather*}
X=\varepsilon\sqrt{T_t}x+X^0(t),\quad
U^1=\phi(t)\exp\left(-\frac{T_{tt}}{8T_t}x^2-\frac\varepsilon2\frac{X^0_t}{\sqrt{T_t}}x\right),
\end{gather*}
where $\varepsilon:=\pm1$ and $\phi$ is a nonvanishing smooth function of $t$.
By substituting these expressions for~$X$ and~$U^1$ into the third equation from~\eqref{eq:PointTransInETransPart}
and splitting the obtained equation with respect to powers of $x$,
we derive that $X^0=0$, $4T_t\phi_t+T_{tt}\phi=0$ and $T_{ttt}/T_t-\frac32(T_{tt}/T_t)^2=0$.
The last equation means that the Schwarzian derivative of~$T$ is zero.
Therefore, $T$ is a linear fractional function of~$t$, $T=(\alpha t+\beta)/(\gamma t+\delta)$.
Since the constant parameters~$\alpha$, $\beta$, $\gamma$ and~$\delta$ are defined up to a constant nonzero multiplier
and $T_t>0$, we can assume that $\alpha\delta-\beta\gamma=1$.
Then these parameters are still defined up to a multiplier in $\{-1,1\}$,
and hence we can choose them in such a way that $\varepsilon|\gamma t+\delta|=\gamma t+\delta$,
thus neglecting the parameter~$\varepsilon$.
The equation $4T_t\phi_t+T_{tt}\phi=0$ takes the form $2(\gamma t+\delta)\phi_t-\gamma\phi=0$
and integrates, in view of~$\phi\ne0$, to $\phi=\sigma\sqrt{|\gamma t+\delta|}$ with $\sigma\in\mathbb R\setminus\{0\}$.

Finally, the equation~\eqref{eq:GenFormOfPointTransInEC} takes the form
\[
\left(\dfrac{U^0}{U^1}\right)_t=\left(\dfrac{U^0}{U^1}\right)_{xx}+\frac\mu{x^2}\left(\dfrac{U^0}{U^1}\right).
\]
Therefore, $U^0=U^1h$, where $h=h(t,x)$ is an arbitrary solution of~\eqref{eq:HeatEqSquarePot}.
\end{proof}

Accurate analysis of the structure of the pseudogroup~$G$ needs a proper interpretation of the corresponding group operation,
which we assume to be done in the same way as that in Section~\ref{sec:RemarkableFPPointSymGroup}.
The point transformations of the form $\tilde t=t$, $\tilde x=x$, $\tilde u=u+h(t,x)$,
where the parameter function $h=h(t,x)$ is an arbitrary solution of the equation~\eqref{eq:HeatEqSquarePot},
constitute the normal pseudosubgroup~$G^{\rm lin}$ of~$G$,
which is associated with the linear superposition of solutions of~\eqref{eq:HeatEqSquarePot},
cf.\ Section~\ref{sec:RemarkableFPPointSymGroup}.
Moreover, the pseudogroup~$G$ splits over $G^{\rm lin}$, $G=G^{\rm ess}\ltimes G^{\rm lin}$,
where the subgroup~$G^{\rm ess}$ consists of the elements of~$G$ with $h=0$
and is a four-dimensional Lie group within the framework of the above interpretation.
We call this subgroup the \emph{essential point symmetry group} of the equation~\eqref{eq:HeatEqSquarePot}.
In turn, the group~$G^{\rm ess}$ contains the normal subgroup $R$ that is constituted by
the point transformations of the form $\tilde t=t$, $\tilde x=x$, $\tilde u=\sigma u$ with $\sigma\in\mathbb R\setminus\{0\}$
and is isomorphic to the multiplicative group~$\mathbb R^\times$ of the real field.
The group~$G^{\rm ess}$ splits over~$R$, $G^{\rm ess}=F\ltimes R$,
where the subgroup~$F$ is singled out from~$G^{\rm ess}$ by the constraint~$\sigma=1$
and is isomorphic to the group ${\rm SL}(2,\mathbb R)$.

\begin{corollary}
A complete list of discrete point symmetry transformations of the equation~\eqref{eq:HeatEqSquarePot}
that are independent up to combining with each other and with continuous point symmetry transformations of this equation
is exhausted by the single involution~$\mathscr I'$ alternating the sign of~$u$,
$\mathscr I'\colon(t,x,u)\mapsto(t,x,-u).$
Thus, the quotient group of the complete point symmetry pseudogroup~$G$ of~\eqref{eq:HeatEqSquarePot}
with respect to its identity component is isomorphic to $\mathbb Z_2$.
\end{corollary}

\begin{proof}
It is obvious that $G^{\rm lin}$ and~$F$ are connected pseudosubgroup and subgroup of~$G$, respectively.
Jointly with the splitting of~$G$ over~$R$, this implies that elements of the required list
can be selected from the subgroup~$R$.
Factoring out the elements of the identity component of~$R$,
which is isomorphic to~$\mathbb R^\times_{>0}$,
we obtain either the identity transformation or~$\mathscr I'$.
\end{proof}

\begin{lemma}(\cite{pate1977a, popo2003a})
A complete list of $G^{\rm ess}$-inequivalent  subalgebras of $\mathfrak g^{\rm ess}$ is exhausted by the subalgebras
\begin{gather*}
\mathfrak s_{1.1}^\delta=\langle\mathcal P^t+\delta\mathcal I\rangle,\ \
\mathfrak s_{1.2}^\nu=\langle\mathcal D+\nu\mathcal I\rangle,_{\ \nu\geqslant0},\ \
\mathfrak s_{1.3}^\nu=\langle\mathcal P^t+\mathcal K+2\nu\mathcal I\rangle,\ \
\mathfrak s_{1.4}=\langle\mathcal I\rangle,
\\
\mathfrak s_{2.1}^\nu=\langle\mathcal P^t,\mathcal D+\nu\mathcal I\rangle,\ \
\mathfrak s_{2.2}=\langle\mathcal P^t,\mathcal I\rangle,\ \
\mathfrak s_{2.3}=\langle\mathcal D,\mathcal I\rangle,\ \
\mathfrak s_{2.4}=\langle\mathcal P^t+\mathcal K,\mathcal I\rangle,
\\
\mathfrak s_{3.1}=\langle\mathcal P^t,\mathcal D,\mathcal K\rangle,\ \
\mathfrak s_{3.1}=\langle\mathcal P^t,\mathcal D,\mathcal I\rangle,
\ \
\mathfrak s_{4.1}=\mathfrak g^{\rm ess},
\end{gather*}
where $\delta\in\{-1,0,1\}$, and~$\nu$ is an arbitrary real constant satisfying indicated constraints.
The first number in the subscript of~$\mathfrak s$ denotes the dimension of the corresponding subalgebra.
\end{lemma}

We present the exhaustive classification of Lie invariant solutions of the equation~\eqref{eq:HeatEqSquarePot} over the field of real numbers
using results from the proof of Theorem 51 in \cite{opan2022a} for integrating the obtained reduced equations.
Hereafter, $C_1$ and $C_2$ are arbitrary real constants, $\kappa:=\tfrac12\sqrt{1-4\mu}$, $\kappa':=\frac12\kappa=\tfrac14\sqrt{1-4\mu}$,
and $\nu$ is a constant parameter.

The reduction with respect to the subalgebra~$\mathfrak s_{1.1}^\varepsilon$ with $\varepsilon\in\{-1,1\}$ results in the solution
\begin{gather*}
u={\rm e}^{\varepsilon t}\sqrt{|x|}\mathsf{Z}_{|\kappa|}(x),
\end{gather*}
where the cylinder function $\mathsf Z_{|\kappa|}(x)$ is
$\mathcal Z_\kappa$, $\mathcal C_\kappa$, $\tilde{\mathcal Z}_{|\kappa|}$ or $\tilde{\mathcal C}_{|\kappa|}$
if $4\mu\leqslant1$ and $\varepsilon=-1$,
$4\mu\leqslant1$ and $\varepsilon=1$,
$4\mu>1$ and $\varepsilon=-1$ or
$4\mu>1$ and $\varepsilon=1$, respectively.
Here $\mathcal Z_\kappa$ and $\mathcal C_\kappa$ are linear combination of Bessel functions and linear combination of modified Bessel functions, respectively,
$\mathcal Z_\kappa=C_1{\rm J}_\kappa(x)+C_2{\rm Y}_\kappa(x)$,
$\mathcal C_\kappa=C_1{\rm I}_\kappa(x)+C_2{\rm K}_\kappa(x)$.
The cylinder function $\tilde{\mathcal Z}_{|\kappa|}$ is a linear combination
of modifications of the Hankel functions ${\rm H}^{(1)}_\kappa$ and ${\rm H}^{(2)}_\kappa$,
\begin{gather*}
\tilde{\mathcal Z}_{|\kappa|}=C_1\tilde{\rm H}^{(1)}_\kappa(x)+C_2\tilde{\rm H}^{(2)}_\kappa(x)\\
\hphantom{\tilde{\mathcal Z}_{|\kappa|}}
=\frac{C_1}2\big({\rm e}^{-\frac12\kappa\pi}{\rm H}^{(1)}_\kappa(x)+{\rm e}^{\frac12\kappa\pi}{\rm H}^{(2)}_\kappa(x)\big)+
\frac{C_2}{2{\rm i}}\big({\rm e}^{-\frac12\kappa\pi}{\rm H}^{(1)}_\kappa(x)-{\rm e}^{\frac12\kappa\pi}{\rm H}^{(2)}_\kappa(x)\big),
\end{gather*}
and the cylinder function $\tilde{\mathcal C}_{|\kappa|}$ is a linear combination
of the modified Bessel function~${\rm K}_{\kappa}$ and the function~$\tilde {\rm I}_{\kappa}$,
which is a modification of the modified Bessel function~${\rm I}_{\kappa}$ (for $\kappa\ne0$),
\begin{gather*}
\tilde{\mathcal C}_{|\kappa|}=C_1\tilde{\rm I}_{\kappa}(x)+C_2{\rm K}_{\kappa}(x),\quad
\tilde{\rm I}_{\kappa}(x):=\frac{\pi{\rm i}}{2\sin(\kappa\pi)}\big({\rm I}_\kappa(x)+{\rm I}_{-\kappa}(x)\big).
\end{gather*}
Each of the functions $\mathcal Z_\kappa$, $\mathcal C_\kappa$, $\tilde{\mathcal Z}_{|\kappa|}$ or $\tilde{\mathcal C}_{|\kappa|}$
is real-valued and represents, for the corresponding values of the parameters~$\kappa$ and~$\varepsilon$,
the general solution of the equation $\varphi_{\omega\omega}+\mu\omega^{-2}\varphi-\varepsilon\varphi=0$,
which is obtained by the Lie reduction of~\eqref{eq:HeatEqSquarePot}
with respect to the subalgebra $\langle\mathcal P^t+\varepsilon\mathcal I\rangle$
using the ansatz $u={\rm e}^{\varepsilon t}\varphi(\omega)$ with $\omega=x$.

The Lie reduction of the equation~\eqref{eq:HeatEqSquarePot} with respect to the subalgebra $\mathfrak s_{1.1}^0$
leads to the ansatz $u=\varphi(\omega)$ with $\omega=x$ and the Euler equation $\omega^2\varphi_{\omega\omega}+\mu\varphi=0$
as the corresponding reduced equation.
We integrate this equation depending on the value of $\sgn(1-4\mu)\in\{-1,0,1\}$
and obtain the following solutions of~\eqref{eq:HeatEqSquarePot}:
\begin{gather*}
u=\sqrt{|x|}\big(C_1\cos(|\kappa|\ln|x|)+C_2\sin(|\kappa|\ln|x|)\big),\quad
u=\sqrt{|x|}\big(C_1+C_2\ln|x|\big),\\
u=\sqrt{|x|}(C_1|x|^\kappa+C_2|x|^{-\kappa}).
\end{gather*}

For each of the subalgebras $\mathfrak s_{1.2}^\nu$ and~$\mathfrak s_{1.3}^\nu$,
we first construct an associated ansatz for~$u$, derive the corresponding reduced equation,
map this equation to a canonical form, which turns out to be the Whittaker equation
\begin{gather}\label{eq:Whittaker}
\varphi_{\omega\omega}+\left(-\frac14+\frac a\omega+\frac{1/4-b^2}{\omega^2}\right)\varphi=0
\end{gather}
with certain $a$ and~$b$,
by a point transformation of the invariant independent and dependent variables,
and then use this transformation for modifying the ansatz.
The general solution of the above Whittaker equation is
the general linear combination of the Whittaker functions $W_{a,b}(z)$ and $M_{a,b}(z)$,
which are linearly independent and whose properties are comprehensively described, e.g., in~\cite{abra1970A}.

Thus, for the subalgebra~$\mathfrak s_{1.2}^\nu$ with a fixed $\nu\geqslant0$,
the used modified ansatz and the values of~$a$ and~$b$ are
\[
u=|t|^{\nu}|x|^{-\frac12}{\rm e}^{-\frac{x^2}{8t}}\varphi(\omega)\quad\mbox{with}\quad
\omega:=\frac{x^2}{4|t|},\quad a=-\varepsilon'\nu\quad\mbox{with}\quad\varepsilon':=\sgn t,\quad b=\kappa'.
\]
The representation of the corresponding solutions of~\eqref{eq:HeatEqSquarePot} over the real field
depends on the value of $\sgn(1-4\mu)$, either $\sgn(1-4\mu)\geqslant0$ or $\sgn(1-4\mu)<0$, and respectively is
\begin{gather*}
u=|t|^{\nu}|x|^{-\frac12}{\rm e}^{-\frac{x^2}{8t}}\left(
C_1M_{-\varepsilon'\nu,\,\kappa'}\left(\frac{x^2}{4|t|}\right)+
C_2W_{-\varepsilon'\nu,\,\kappa'}\left(\frac{x^2}{4|t|}\right)
\right),\\
u= |t|^{\nu}|x|^{-\frac12}{\rm e}^{-\frac{x^2}{8t}}{\rm Re}\left(
(C_1-{\rm i}C_2)W_{-\varepsilon'\nu,\,{\rm i}|\kappa'|}\left(\frac{x^2}{4|t|}\right)
\right).
\end{gather*}
We distinguish the case $\sgn(1-4\mu)<0$ since then the Whittaker function $W_{-\varepsilon'\nu,\kappa'}(\omega)$ is complex-valued.
However, its real and imaginary parts are linearly independent real solutions of the corresponding Whittaker equation.

For each of the subalgebras~$\mathfrak s_{1.3}^\nu$,
we use the modified ansatz
\[
u=|x|^{-\frac12}{\rm e}^{-\frac14\frac{x^2t}{t^2+1}+2\nu\arctan t}\varphi(\omega)\quad\mbox{with}\quad
\omega:=\frac{{\rm i}x^2}{2(t^2+1)}.
\]
The associated reduced equation is of the form~\eqref{eq:Whittaker} with $a={\rm i}\nu$ and $b=\kappa'$.
Hence the corresponding solutions of~\eqref{eq:HeatEqSquarePot} over the real field are
\begin{gather*}
u=|x|^{-\frac12}{\rm e}^{-\frac14\frac{x^2t}{t^2+1}+2\nu\arctan t}\mathop{\rm Re}
\left((C_1-{\rm i}C_2)W_{{\rm i}\nu,\,\kappa'}\left(\frac{{\rm i}x^2}{2(t^2+1)}\right)\right).
\end{gather*}
If $\kappa\in2\mathbb N_0+1$, these solutions can be represented in terms of regular and irregular Coulomb functions,
\[
u=|x|^{-\frac12}{\rm e}^{-\frac14\frac{x^2t}{t^2+1}+2\nu\arctan t}
\left(C_1 F_{\kappa'-\frac12}\left(\nu,\frac{x^2}{4(t^2+1)}\right)+C_2 G_{\kappa'-\frac12}\left(\nu,\frac{x^2}{4(t^2+1)}\right)\right).
\]

\section{On construction and generation of new exact solutions\\ using hidden symmetries}\label{sec:HiddenSyms}

Let $\mathfrak g$ and $G$ be the maximal Lie invariance algebra and the point symmetry (pseudo)group
of a system of differential equations~$\mathcal L$,
and let $\mathfrak s$ be a subalgebra of~$\mathfrak g$ that is appropriate for Lie reduction of the system~$\mathcal L$.
By $\hat{\mathcal L}$, $\hat{\mathfrak g}$ and~$\hat G$ we denote
a reduced system of~$\mathcal L$ with respect to a Lie reduction~$\rho$ associated with~$\mathfrak s$,
its maximal Lie invariance algebra and its point symmetry group, respectively.
The system~$\hat{\mathcal L}$ has the same number of the dependent variables as the system~$\mathcal L$,
and its number of the independent variables is that of~$\mathcal L$ reduced by $\dim\mathfrak s$.
Here~$\rho$ denotes an $\mathfrak s$-invariant map of maximal rank
from the space whose coordinates are the independent and dependent variables of~$\mathcal L$
to the analogous space for the system~$\hat{\mathcal L}$,
such that $\rho_*\mathcal L$ is equal to $\hat{\mathcal L}$ up to an nondegenerate differential-functional matrix multiplier.

Lie-symmetry vector fields of the system~$\mathcal L$
from the normalizer ${\rm N}_{\mathfrak g}(\mathfrak s)$ of the subalgebra~$\mathfrak s$ in the algebra~$\mathfrak g$
induce Lie-symmetry vector fields of the system~$\hat{\mathcal L}$,
which constitute a subalgebra~$\tilde{\mathfrak g}$ of~$\hat{\mathfrak g}$, $\tilde{\mathfrak g}=\rho_*{\rm N}_{\mathfrak g}(\mathfrak s)$.
Similarly, $\tilde G=\rho_*{\rm St}_G(\mathfrak s)$ is the group of induced point symmetries of the system~$\hat{\mathcal L}$.
Any element of $\hat{\mathfrak g}\setminus\tilde{\mathfrak g}$ is a genuine hidden Lie-symmetry vector field
of the original system~$\mathcal L$
with respect to the Lie reduction obtained with the subalgebra~$\mathfrak s$,
see~\cite[Example~3.5]{olve1993A}, \cite{abra2008a} and the discussion in~\cite[Remark~15]{poch2017a}.

Subalgebras~$\hat{\mathfrak s}$ of~$\hat{\mathfrak g}$ can used for further Lie reductions of the system~$\hat{\mathcal L}$.
The question is which subalgebras of~$\hat{\mathfrak g}$ may lead to solutions of~$\hat{\mathcal L}$,
whose pullbacks by~$\rho$ are \emph{hiddenly invariant solutions} of the system~$\mathcal L$,
i.e., they are not invariant with respect to subalgebras of~$\mathfrak g$ of greater dimensions than~$\dim\mathfrak s$.
It is obvious that this is definitely not the case for any subalgebra of~$\hat{\mathfrak g}$
each element of whose $\hat G$-orbit is contained in~$\tilde{\mathfrak g}$.
Moreover, if solutions $\hat f^1$ and $\hat f^2$ of the system~$\hat{\mathcal L}$ are $\tilde G$-equivalent,
then their counterparts $\rho^*\hat f^1$ and $\rho^*\hat f^2$ among solutions of the system~$\mathcal L$ are necessarily $G$-equivalent.
Taking into account these remarks,
we can formulate the following optimized procedure
for constructing hiddenly invariant solutions of~$\mathcal L$:
\begin{itemize}\itemsep=0ex
\item
Construct a complete list of $\hat G$-inequivalent subalgebras of the algebra~$\hat{\mathfrak g}$.
\item
From this list, exclude all such subalgebras~$\hat{\mathfrak s}$ that
any element of the $\hat G$-orbit of~$\hat{\mathfrak s}$ is contained in~$\tilde{\mathfrak g}$.
\item
For each of the remaining subalgebras in the list,
carry out, if possible, the Lie reduction of the system~$\hat{\mathcal L}$ with respect to it
and find the corresponding invariant solutions of~$\hat{\mathcal L}$.
\item
Extend the obtained solutions using a complete set of $\tilde G$-inequivalent transformations from~$\hat G$
under the left action of $\tilde G$ on~$\hat G$.
\item
Pull back by~$\rho$ the extended set of solutions.
\item
Make a final arrangement of the constructed set of solutions of~$\mathcal L$,
selecting, up to the $G$-equivalence, only those solutions
that are really hiddenly invariant.
\end{itemize}

\subsection*{Acknowledgements}

The authors are grateful to Vyacheslav Boyko, Yevhenii Azarov, Yevhen Chapovskyi, Dmytro Popovych and Galyna Popovych for valuable discussions.
The authors also sincerely thank the two reviewers for their helpful suggestions and comments,
which led to essentially improving the presentation of results.
This research was undertaken thanks to funding from the Canada Research Chairs program,
the InnovateNL LeverageR{\&}D program and the NSERC Discovery Grant program.
It was also supported in part by the Ministry of Education, Youth and Sports of the Czech Republic (M\v SMT \v CR)
under RVO funding for I\v C47813059.
ROP expresses his gratitude for the hospitality shown by the University of Vienna during his long staying at the university.
The authors express their deepest thanks to the Armed Forces of Ukraine and the civil Ukrainian people
for their bravery and courage in defense of peace and freedom in Europe and in the entire world from russism.



\end{document}